\documentclass[11pt]{amsart}
\usepackage[margin=3cm]{geometry}
\usepackage{lipsum}

\usepackage{amssymb,amsfonts, amsmath, amsthm}
\usepackage[all,arc]{xy}
\usepackage{enumerate}
\usepackage{mathrsfs}
\usepackage{tikz-cd}
\newtheorem{thm}{Theorem}[section]
\newtheorem{cor}[thm]{Corollary}
\newtheorem{prop}[thm]{Proposition}
\newtheorem{lem}[thm]{Lemma}
\newtheorem{conj}[thm]{Conjecture}

\newtheorem{prob}[thm]{Problem}

\theoremstyle{definition}
\newtheorem{defn}[thm]{Definition}
\newtheorem{defns}[thm]{Definitions}

\newtheorem{exmp}[thm]{Example}

\newtheorem{assump}[thm]{Assumption}

\newtheorem{defn/prop}[thm]{Definition/Proposition}

\theoremstyle{remark}
\newtheorem{rem}[thm]{Remark}
\newtheorem{rems}[thm]{Remarks}

\DeclareMathOperator{\E}{\mathcal{E}}
\DeclareMathOperator{\Pic}{Pic}

\DeclareMathOperator{\Hom}{Hom}

\DeclareMathOperator{\Spec}{Spec}
\DeclareMathOperator{\Ad}{Ad}

\DeclareMathOperator{\bO}{\overline{\mathcal{O}}}
\DeclareMathOperator{\O'}{\mathcal{O}}
\DeclareMathOperator{\S'}{\mathcal{S}}
\DeclareMathOperator{\A'}{\mathbb{A}}
\DeclareMathOperator{\bA'}{\bar{\mathbb{A}}}

\DeclareMathOperator{\N}{\mathcal{N}}
\DeclareMathOperator{\bC}{\overline{C}}
\DeclareMathOperator{\bun}{Bun}

\DeclareMathOperator{\g}{\mathfrak{g}}
\DeclareMathOperator{\omegac}{\omega_{\overline{C}}}

\DeclareMathOperator{\tS}{\tilde{\mathcal{S}}}
\DeclareMathOperator{\ind}{Ind}
\DeclareMathOperator{\M}{\mathcal{M}}
\DeclareMathOperator{\diag}{diag}
\DeclareMathOperator{\Hk}{\mathcal{H}_K}
\DeclareMathOperator{\ad}{ad}
\DeclareMathOperator{\inv}{^{-1}}
\newcommand{\str}{\mathrm{Strng}}
\newcommand{\bF}{\overline{F}}
\DeclareMathOperator{\rk}{rk} 
\newcommand{\PGL}{\mathrm{PGL}}
\newcommand{\St}{\mathrm{St}}
\newcommand{\GL}{\mathrm{GL}}
\newcommand{\Na}{N_\mathbb{A}}
\newcommand{\Nf}{N_F}

\makeatletter
\let\c@equation\c@thm
\makeatother
\numberwithin{equation}{section}

\title{Automorphic functions for square-zero extensions of curves over finite fields}
\author{Ka Fai Wong}

\date{}

\begin{document}
\begin{abstract}
We study automorphic functions for square-zero extensions $C$ of curves $\overline{C}$ over finite fields, a study initiated by Braverman-Kazhdan-Polishchuk in [BKP23]. More precisely, we study the cuspidality and Hecke-finiteness of the functions in the orbit decomposition introduced in loc. cit. for split connected reductive groups $G$, generalizing some of the results for $G=\mathrm{PGL}_2$. As a result, for $G=\mathrm{PGL}_3$, we prove a new case of a conjecture in [BK23] concerning the finite-dimensionality of the space of unramified Hecke-finite functions. We also introduce a formulation of support bounds for spherical cuspidal and Hecke-finite functions in terms of the Harder-Narasimhan stratification of $G$-bundles on the reduced curve. Using representation-theoretic constructions together with their geometric interpretations in terms of $G$-bundles on $C$ and certain twisted $G$-Higgs bundles on $\overline{C}$, we compute the optimal bounds in several cases and, in particular, determine the optimal bound for $G=\mathrm{PGL}_3$.
\end{abstract}

\maketitle
\section{Introduction}
\subsection{Motivation}
\subsubsection{The classical Langlands program over global function fields and the geometric Langlands}

Let $C$ be a smooth projective curve over a finite field $k$, and let $G$ be a connected split reductive group. The classical (unramified) Langlands program aims to understand the Hecke module $\S'_{cusp}(\bun_G(C))$ of cuspidal functions on the groupoid of $G$-bundles on $C$, i.e. the space of unramified cuspidal automorphic functions, in terms of Galois-theoretic data. Roughly speaking, it seeks to describe the decomposition of $\S'_{cusp}(\bun_G(C))$ into Hecke eigenspaces in terms of conjugacy classes of irreducible homomorphisms $\pi_1(C)\to G^\vee(\overline{\mathbb{Q}}_\ell)$.

This correspondence admits a geometric reformulation in which automorphic functions are replaced by sheaves on $\bun_G(C)$. In this setting, Hecke eigenfunctions are replaced by Hecke eigensheaves on the automorphic side, while the homomorphisms $\pi_1(C)\to G^\vee(\overline{\mathbb{Q}}_\ell)$ are replaced by irreducible $G^\vee$-local systems on $C$.

This geometric reformulation not only provides powerful geometric tools for studying automorphic forms over global function fields and the Langlands correspondence, but also leads naturally to a new direction obtained by replacing the finite field $k$ with $\mathbb{C}$, namely the geometric Langlands program.

\subsubsection{The analytic Langlands program}
Let $C$ be a smooth projective curve over a local field $F$, for example, $F=\mathbb{C} \text{ or  } \mathbb{Q}_p$. Motivated by a question of Langlands, Etingof, Frenkel and Kazhdan proposed several approaches to an analytic theory of automorphic objects on $C$, in which one studies suitable spaces of ``automorphic functions'' rather than automorphic sheaves (see \cite{EFK1}, \cite{EFK2} and \cite{EFK3}).

When $F$ is a nonarchimedean local field with ring of integers $\O'$, and $G$ is a reductive group over $\mathbb{Z}$, one possible approach to the automorphic side of the analytic Langlands program concerns the space $\S':=\S'(\bun_G(C), |\omega|^{1/2})$ of half-densities on the groupoid of $G$-bundles on $C$.  This approach, which we will briefly describe below, shares some similarities with the theory of automorphic functions for curves over finite fields, while also exhibiting several novel features (see \cite{BK2} for more details).

When $C$ admits a smooth model $C_{\O'}$ over $\O'$, one considers the subgroupoid $\bun_G^{\O'}(C) \subset \bun_G(C)$ of $G$-bundles that can be extended to $C_{\O'}$. Under some mild conditions on $G$, it is known that $\bun_G^{\O'}(C)$ coincides with the subgroupoid of the generically trivial $G$-bundles, and that the half-densities supported on  $\bun_G^{\O'}(C)$ arise from smooth functions on the groupoid $\bun_G(C_{\O'})$ of $G$-bundles on $C_{\O'}$ (see \cite{BKPW}).

Let $\mathbf{m}$ be the maximal ideal of $\O'$. For 
$$C_n:= C_{\O'}\times_{\Spec \O'} \Spec(\O'/\mathbf{m}^n),$$ there is a natural map
$$ \iota_n : \S'(\bun_G(C_n)) \to \S'(\bun_G(C_{\O'})),$$
where $\S'(\bun_G(C_n))$ (resp. $\S'(\bun_G(C_{\O'}))$) denote the space of smooth functions on $\bun_G(C_n)$ (resp. $\bun_G(C_{\O'})$). Furthermore, both spaces are equipped with actions of Hecke algebras and $\iota_n$ is compatible with the Hecke actions. 

Therefore, it is natural to study the Hecke modules $\S'(\bun_G(C_n))$ for $n\geq 1$. When $n=1$, this is the classical study of automorphic functions for curves over finite fields. When $n=2$ i.e. $C_2$ is a square-zero extension of a curve over a finite field, a detailed study was initiated in \cite{BKP}, and the theory for $G=\PGL_2$ is well understood. The purpose of this paper is to extend the theory to split connected reductive groups.

\subsection{Main conjectures and results}
Let $\overline{C}$ be a smooth projective curve of genus $g$ over a finite field $k$, and let $C$ be a square-zero extension of $\bar{C}$. We denote the function field of $\overline{C}$ (resp. the stalk of $\mathcal{O}_C$ at the generic point of $C$) by $\bF$ (resp. $F$), its ring of adeles by $\bA'$ (resp. $\A'$) and its ring of integral adeles by $\bO$ (resp. $\O'$). Let $\N \subset \mathcal{O}_C$ be the nilradical. Then we say that $C$ is a special nilpotent extension of $\overline{C}$ if $\N$ is a line bundle of degree zero on $\overline{C}$. In the present paper, we always assume that our square-zero extension is special.

Let $G$ be a connected split reductive group over $k$. The automorphic theory is concerned with the space $\S'(G(F) \backslash G(\A'))$ of automorphic functions, namely, the space of locally constant functions on $G(F) \backslash G(\A')$ with compact support, and its subspace of right $G(\O')$-invariant functions, i.e. the unramified functions.

Fix a Borel subgroup $B$ of $G$. Let $P$ be a standard parabolic subgroup of $G$ with  unipotent radical $U_P$ and Levi quotient $M_P$.  The \textit{constant-term operator}  $$CT_P: \S'(G(F)\backslash G(\A')) \to \mathbb{C}(M_P(F)U_P(\A')\backslash G(\A'))$$ is defined as
$$CT_P(f)(g)= \int_{u\in U_P(F) \backslash U_P(\A')} f(ug) du. $$
\begin{defn}
    A function $f$ is cuspidal if $CT_P(f)=0$ for all proper standard parabolic subgroups $P$. 
\end{defn}

Another notion closely related to cuspidality (indeed conjecturally equivalent) is Hecke-finiteness. For every compact open subgroup $K \subset G(\A')$, we have the global Hecke algebra $\Hk := \mathbb{C}_c(K\backslash G(\A')/ K)$ acting on the subspace $\S'(G(F)\backslash G(\A'))^K$ of $K$-invariants (though, unlike the classical theory, $\Hk$ is not commutative).  

\begin{defn}
    A function $f \in \S'(G(F)\backslash G(\A'))^K$ is $\Hk$-finite if the Hecke orbit $\Hk \cdot f $ is finite-dimensional. Given a smooth $G(\A')$-subrepresentation $V \subset \S'(G(F)\backslash G(\A'))$, we say that $V$ is Hecke-finite, if for every compact open subgroup $K\subset G(\A')$, every element in the subspace $V^K$ is $\Hk$-finite.
\end{defn}

\subsubsection{Braverman-Kazhdan-Polishchuk Conjectures}
In \cite{BK2} and \cite{BKP}, the authors proposed the following conjectures and proved them for the case $G=\PGL_2$.
\begin{conj}\label{conj1}\footnote{The second part of the conjecture is stated in \cite{BK2}; in \cite{BKP} it is stated for cuspidal functions instead of Hecke-finite ones. }
Let $G$ be a connected split reductive group over $\mathbb{Z}$.
\begin{enumerate}
    \item An unramified function $f$ is cuspidal if and only if it is Hecke-finite.
    \item The space $$\S'_{HF}(G(F) \backslash G(\A')/G(\O'))$$ of unramified Hecke-finite functions  is finite-dimensional if $G$ is semisimple.

\end{enumerate}
    
\end{conj}

\subsubsection{A formulation on cuspidal support and the Harder-Narasimhan stratification}
In the classical theory of automorphic functions on $\bun_G(\overline{C})$, one has a bound for the total support of the cuspidal functions , i.e., the union of the supports of all cuspidal functions, in terms of Harder-Narasimhan stratification of $\bun_G(\overline{C})$. When $G=\GL_n$, an easy way to formulate such a bound is to require the Harder-Narasimhan slope sequence $(\mu_1,...,\mu_n)$ of a rank-n vector bundle satisfy $\mu_i-\mu_{i+1}\leq 2g-2$ for every $i$. When $G$ is an arbitrary connected split reductive group, the formulation is similar using simple roots, but is more subtle and related to a notion of strangeness by Drinfeld-Gaitsgory (see Section \ref{drinfeld-gaitsgory} and  \cite[10.3]{DG2} for details).

We now propose a formulation of cuspidal/Hecke-finite support bounds in the setting of square-zero extensions. We follow \cite{Beh}, \cite{DG2} and \cite{Sc} for our convention of Harder-Narasimhan stratification regarding $\bun_G(\overline{C})$, although for the purpose of this paper we are only concerned with the $k$-points. Let $\Lambda$ be the coweight lattice of $G$, $\Lambda_\mathbb{Q}:= \Lambda \otimes \mathbb{Q}$, and $\Lambda^+_\mathbb{Q} \subset \Lambda_\mathbb{Q} $ be the set of dominant coweights. The Harder-Narasimhan stratification for $\bun_G(\overline{C})$ is parametrized by $\lambda \in \Lambda^+_\mathbb{Q} $. Let $m\in \mathbb{Q}$ and $\Sigma_m$ be the set of $\lambda \in \Lambda^+_\mathbb{Q} $ such that its image $\lambda'$ in the space $\Lambda_\mathbb{Q}^{ad}$ of coweights of the adjoint group of $G$  satisfies $\langle \lambda', \alpha_h \rangle \leq (2g-2)m$ for the highest root $\alpha_h$ of each simple factor of $G^{ad}$ . Let 
$$ \mathfrak{U}_m= \bigcup_{\lambda \in \Sigma_m} \bun_G^\lambda(\overline{C}). $$

Let $n$ be the semisimple rank of $G$,  $c_G$ the Coxeter number\footnote{For a connected reductive group $G$, the Coxeter number of $G$ is the maximum of the Coxeter numbers of the simple factors of $G^{ad}$.} of $G$ and  $r:\bun_G(C) \to \bun_G(\overline{C})$ the reduction map.  Let $S_c$ (resp. $S_h$) be the union of the supports of all cuspidal (resp. Hecke-finite) functions in $\S'(\bun_G(\overline{C}))$. We are interested in the following problem:

\begin{prob}\label{HN problem}
    Find the smallest constant $b_h=b_h(C,G) \in \mathbb{Q}$ (resp. $b_c=b_c(C,G) \in \mathbb{Q}$ ) such that $S_h$ (resp. $S_c$) is contained in $r^{-1}(\mathfrak{U}_{b_h})$ (resp. $r^{-1}(\mathfrak{U}_{b_c}))$.
\end{prob}

\begin{conj}\label{conj2}Let $G$ be a split connected reductive group.  
\begin{enumerate}
    \item $S_c$ (resp. $S_h$) is contained in $r^{-1}(\mathfrak{U}_c)$, where $c$ is a constant depending only on the Coxeter number $c_G$ of $G$, the genus $g$ of $\overline{C}$, and the strangeness \footnote{See Section \ref{drinfeld-gaitsgory} for the definition.} of all simple roots of $G$.
    \item When $G$ is a semisimple group whose simple factors are all of type $A$,  $b_h(C,G)$ (and $b_c(C,G)$) grows at most quadratically\footnote{The growth may in fact be linear.} with the rank of $G$.
\end{enumerate}
 
\end{conj}

\begin{rems} \label{hn formulation rem} 
In contrast to the use of simple roots in formulating the Harder-Narasimhan bound for unramified cuspidal functions in the classical setting, it is essential for us to formulate the bound using the highest roots to capture the optimal bounds. Another notable feature of the square-zero setting is that the optimal bound can often be determined exactly (see Section 1.5 for a more detailed discussion).

\end{rems}

\subsubsection{Main results}
The main result of the paper is the following:
\begin{thm}\label{main theorem}
     Let $G=\PGL_3$, and assume that  $\text{char}(k)>3$.
     \begin{enumerate}
         \item Conjecture \ref{conj1}(2) holds.
         \item  If $f \in \S'(G(F) \backslash G(\A')/G(\O')) $ is Hecke-finite,  then it is cuspidal.
         \item $S_h$ is contained in $r^{-1}(\mathfrak{U_3})$.
     \end{enumerate}
     
\end{thm}

Our approach to proving Theorem \ref{main theorem} (1)-(2) is as follows. For more general groups $G$, using the machinery of the orbit decomposition introduced in \cite{BKP} and a more general version of that which we call \textit{global Mackey theory} (this is formulated and discussed in Section 3), we are able to decompose the automorphic representation $\S'(G(F) \backslash G(\A'))$ into subrepresentations for which we can directly show that the subspaces of spherical vectors for some components are either Hecke-finite, cuspidal and finite-dimensional (and there are finitely many such components), or Hecke-infinite. Specializing these results to $G=\PGL_3$ yields (1) and (2). The proof of (3) relies on a geometric interpretation of our representation-theoretic construction, which allows us to compute the support of spherical functions in terms of functions on Hitchin fibers. Using the Fourier transform of \cite{KP}, which relates functions on Hitchin fibers to functions on $\bun_G(C)$, we can compute the reduced image of the support in $\bun_G(\overline{C})$.

\subsection{Orbit decomposition}
One important feature of the automorphic theory in the setting of square-zero extensions is that the automorphic representation admits an orbit decomposition. While the discussion of the general machinery of global Mackey theory is deferred to Section \ref{orbitdecomp}, here we briefly explain how one obtains an orbit decomposition for $ \S'(G(F) \backslash G(\A'))$.

Given the square-zero extension introduced in the previous subsection, we have the following short exact sequence:
$$1 \to  \g\otimes \N(\bA') \to G(\A') \to G(\bA') \to 1, $$

The coadjoint action of $G(\bF)$ on $\mathfrak{g}^\vee \otimes \N \inv \omega_{\bC} (\bF)$ decomposes the space of automorphic functions $\S'=\S'(G(F) \backslash G(\A'))$ into $\bigoplus_\Omega \S'_\Omega$, where $\Omega$ ranges over all coadjoint orbits. We will also identify the coadjoint orbits with the adjoint orbits by identifying $\mathfrak{g}$ with $\mathfrak{g}^\vee$ via an $\Ad(G)$-invariant nondegenerate bilinear form.

In \cite{BKP}, the authors study the orbit decomposition for $G=\PGL_2$ in full detail. For general reductive groups $G$, they also prove that the zero summand corresponds to the classical automorphic functions on $\bar{C}$,  and the regular semisimple summands that are not elliptic contain no admissible subrepresentations, i.e., contains no finitary function. Our goal is to study the orbit decomposition for general reductive groups $G$, and obtain the cuspidality and/or Hecke-finiteness for each orbit.

\subsection{Summary of the results on the orbit decomposition}
Here is a summary of our results on cuspidality and Hecke-finiteness/infiniteness for various orbits:

\begin{enumerate}
    \item For any connected split reductive group $G$, the direct sum $\bigoplus_\Omega \S'_\Omega$, where $\Omega$ runs through the \textit{almost regular elliptic} orbits\footnote{See Definition \ref{are def} for the definition of almost regular elliptic element. Any regular elliptic element is almost regular elliptic, and the converse is true if $G=\GL_n$.}, is precisely the space of strongly cuspidal functions. If $G$ is semisimple, then the space of spherical vectors is finite-dimensional.

    \item Let $\eta= \eta_{ss}+\eta_n$ be the Jordan-Chevalley decomposition of $\eta$. If $\eta_{ss}$ is a nonzero semisimple element that is not elliptic, then all functions in $\S'_{\Omega_\eta}$ are Hecke-infinite.

    \item When $G$ is semisimple of adjoint type, the regular nilpotent summand further admits a decomposition indexed by characters $\chi=(\chi_1,...,\chi_r)$ of a compact group $(\bA'/\bF)^r$, where $r$ is the number of simple factors of $G$.  If all $\chi_i$'s are nontrivial, then all the functions in that component are cuspidal, and the subspace of spherical functions is finite-dimensional. If we further assume that all simple factors of $G$ are of types $A$, $B$, $C$, or $G_2$, then all nonzero functions in the components corresponding to $\chi$, where at least one of $\chi_i$'s is trivial, are Hecke-infinite.

    \item Let $G=\PGL_n$, where $n \geq 3$. The subregular nilpotent orbit is related to the automorphic representation of an extension of $\mathbb{G}_m$ by a Heisenberg group. We decompose this subrepresentation $\S'_\Omega$ explicitly into a Hecke-finite component and a Hecke-infinite component, and show that the spherical Hecke-finite subspace is finite-dimensional. For one component in the Hecke-finite part, we also show its cuspidality. When $n=3$, the entire Hecke-finite component is cuspidal.

\end{enumerate}
\begin{rem}
    The proofs for the Hecke-finiteness and Hecke-infiniteness in the regular and subregular nilpotent orbits can be viewed as a generalization of the results of Kazhdan-Yom Din in \cite{KY} on the nilpotent orbit in $\mathfrak{pgl}_2$.
\end{rem}

\begin{rem}
    By the above results, the 
study of orbit decompositions for connected split reductive groups beyond \(\PGL_3\) 
reduces to understanding elliptic orbits that are not almost regular elliptic, mixed orbits with nonregular elliptic semisimple part, and nilpotent orbits beyond the regular and subregular orbits.
\end{rem}

\subsection{Summary of the results on the support bounds for Hecke-finite functions in various orbits} Passing to the spherical vectors, we have a decomposition $ \S'(\bun_G(C)) = \bigoplus_\Omega \S'_\Omega^{G(\O')}$ of Hecke modules. Let $S_{h, \Omega}$ be the union of the supports of all Hecke-finite functions in $ \S'_\Omega^{G(\O')}$. Here is a summary of our results on the bounds of $S_{h, \Omega}$ for various orbits:
\begin{enumerate}
    \item When $G$ is any connected split reductive group and $\Omega$ is the zero orbit or an almost regular elliptic orbit, the classical bound for $\S'_{cusp}(\bun_G(\overline{C}))$ gives a bound for $r(S_{h, \Omega})$. In particular, for $G=\PGL_n$, we have $r(S_{h, \Omega}) \subset \bigcup_\lambda \bun_G(\overline{C})$ where $\lambda$ ranges over the dominant coweights with $\langle \lambda, \alpha \rangle \leq 2g-2$ for all simple roots $\alpha$, which implies that $\langle \lambda, \alpha_h \rangle \leq (n-1)(2g-2)$.

    \item When $G$ is semisimple of adjoint type and all simple factors of $G$ are of types $A$, $B$, $C$ or $G_2$, and $\Omega$ is the regular nilpotent orbit, then $r(S_{h, \Omega}) \subset \mathfrak{U}_{c-1}$, where $c$ is the Coxeter number of $G$.

    \item When $G=\PGL_n$ and $\Omega$ is the subregular nilpotent orbit, then $r(S_{h, \Omega}) \subset \mathfrak{U}_{2n-3}$.

    \item When $G=\PGL_{2n}$, we obtain a class of spherical Hecke-finite functions in the nilpotent orbit of type $(2,2,..,2)$, and show that the support is contained in $r^{-1}(\mathfrak{U}_{3n+3})$.

    \item Combining (1)-(3), for $G=\PGL_3$, we obtain that $S_h \subset r^{-1}(\mathfrak{U}_3)$.
    
\end{enumerate}

\begin{rems} We have a few important remarks.

\begin{enumerate}
  \item For each $G(\bF)$-orbit $\Omega \subset \mathfrak{g} \otimes \N \inv \omega (\bF)$, we can define $b_\Omega$ to be the smallest number $l$ such that $S_{h,\Omega} \subset r^{-1}(\mathfrak{U}_l)$. Then we have $b_h(C,G)= \sup_\Omega b_\Omega $.

    \item In the classical theory, for example when $G=\PGL_n$, it is known that  $b_h(\overline{C},G)=b_c(\overline{C},G)\leq n-1$ but the exact value is not known. But in the square-zero extension, when $\Omega$ is a regular (resp. subregular) nilpotent orbit, our representation-theoretic construction actually shows that the above bound $n-1$ (resp. $2n-3$) is sharp. We can use this to deduce that $b_h(C, \PGL_3)=3$.

    \item The invariant $b_h(C,G)$ also appears to be more accessible than its conjecturally equivalent counterpart $b_c(C,G)$ (Conjecture \ref{conj1}(1)). The reason is that our representation-theoretic description of the individual orbit contributions gives considerably finer control over the support of Hecke-finite functions than is presently available for cuspidal functions. Finding a conceptual geometric framework for understanding $b_c(C,G)$, even for $G=\PGL_2$, would be an interesting direction for future work.
  
\end{enumerate}

\end{rems}

\subsection{Organization of the paper}
In Section \ref{basictheory}, we fix the notation and recall the necessary definitions and concepts in the context of the automorphic theory for square-zero extensions. We prove an analog of the Iwasawa decomposition for $G=\GL_n$ over square-zero extensions, generalizing the case $G=GL_2$ in \cite{BKP}, and discuss the Eisenstein operators as adjoints to the constant-term operators. 

In Section \ref{secorbitdecomp}, we formulate a global Mackey theory generalizing the representation-theoretic construction underlying the orbit decomposition in \cite{BKP}. Although the relevant results in loc. cit. are formulated and applied to the particular short exact sequences arising from square-zero extensions, the underlying arguments extend to a more general setting. For the present paper, this broader formulation is essential for our analysis of the more complicated nilpotent orbits arising beyond $\PGL_2$. We also recall the connection between the subrepresentations $\S'_\Omega$ and certain spaces of functions on Hitchin fibers and their identification via the Fourier transform described in \cite{KP}. This gives a geometric interpretation of our representation-theoretic constructions in Sections 4-6, and will be used to compute the supports of spherical Hecke-finite functions in Section 7 when we prove Theorem \ref{main theorem}(3).

In Section \ref{regellip and mixed}, we prove the main results for $\S'_{\Omega_\eta}$, where $\eta$ is an almost regular elliptic or mixed-type element whose semisimple part is nonelliptic.

In Sections \ref{regnilpmain} and \ref{subregnilp}, we study regular nilpotent orbits for connected split reductive groups $G$ and subregular nilpotent orbits for $\PGL_n$, respectively. 

In Section 7, we gather the results from the previous sections to prove Theorem \ref{main theorem}. 

In Section \ref{misc}, we study the nilpotent orbits of type $(2,...,2)$ for $\PGL_{2n}$ for $n\geq 2$. Interestingly, this is related to the automorphic representation of $\PGL_n$ for the split square-zero extension. We obtain a class of spherical Hecke-finite functions and compute their support.

\section*{Acknowledgement}
I would like to thank my advisor, Alexander Polishchuk, for introducing me to the subject of this paper, for sharing many of his ideas with me, and for the many conversations we have had over the years.

\section{Automorphic functions for square-zero extensions of curves over finite fields}\label{basictheory}
In this section, we recall the basic setup for automorphic functions for square-zero extensions of curves over finite fields.

\subsection{Square-zero extension}
Throughout this paper, we let $\overline{C}$ be a smooth projective curve over a finite field $k$, and let $C$ be a square-zero extension of $\overline{C}$, and we denote by $\N$ the nilradical of the structure sheaf $\mathcal{O}_C$. In other words, the induced reduced scheme of $C$ is $\overline{C}$ and $\N^2=0$. We assume throughout that $\N$ is a line bundle of degree zero on $\overline{C}$. Based on this setup, we fix the following notation and conventions:

\begin{enumerate}
    \item[\textbullet] The genus of $\overline{C}$ is denoted by $g$, and any number denoted by $g$ is the genus, unless specified otherwise. We also denote by $\omega$ the canonical sheaf $\omegac$ sometimes.
    \item[\textbullet] $\bF=k(\bC)$: the field of rational functions on $\bC$.
    \item[\textbullet] $F$: the stalk of the structure sheaf of $C$ at the generic point.
    \item[\textbullet] For a closed point $p$ of $C$, let $\O'_p$ (resp. $F_p$) denote the completed local ring at $p$ (resp. its total ring of fractions). 
    \item[\textbullet] $\O'=\Pi_{p \in \bC} \O'_p$ is the ring of integral adeles for $F$, and $\A'$ is the ring of adeles i.e. the restricted product $\Pi_{p\in \bC} F_p$ with respect to $\O'_p$.
    \item[\textbullet] $\bA'$ (resp. $\bO$) is the ring of adeles (resp. the ring of integral adeles) of $\bF$.
    \item[\textbullet] For a line bundle $M$ on $\bC$, we let the $\bA'-$module $M(\bA')$ of twisted adeles be the restricted product of $M_p \otimes \bF_p$, and denote the principal twisted adeles by $M(\bF)$.
    \item[\textbullet]
     The ring $F_p$ is a square-zero extension of the local field $\overline{F}_p$ by the completed stalk $\N_p$. We fix a generator $\epsilon_p \in \N _p$ for each $p$ and let
    $$ \epsilon=(\epsilon_p) \in \N \A'. $$

    \item[\textbullet]
    When $X$ is a topological space, we denote by $\mathbb{C}(X)$ the space of all locally constant $\mathbb{C}$-valued functions, and by $\S'(X) \subset \mathbb{C}(X)$ the subspace of functions with compact support. In particular, when $X$ discrete, $\mathbb{C}(X)$ (resp. $\S'(X)$) is the space of all functions (resp. finitely supported functions).

\end{enumerate}

\subsection{Groups and automorphic functions}
Throughout this paper, we let $G$ be a connected split reductive group over $\mathbb{Z}$, unless specified otherwise, and denote its Lie algebra by $\mathfrak{g}$. We also fix a Borel subgroup $B$ with a split maximal torus $T$, and let $U$ be the unipotent radical of $B$. A parabolic subgroup of $G$ that contains $B$ is called standard.

Most of the time, we will assume that the characteristic of $k$ is greater than the semisimple rank of $G$, which guarantees the existence of Jordan-Chevalley decompositions of elements in $\mathfrak{g}(\bF)$. We need this for two reasons. First, this allows us to classify the (co)adjoint orbits systematically. Second, for a technical reason, it gives a neater proof for the fact that all strongly cuspidal functions are in the regular elliptic summands.

\subsection{Cuspidality}
We recall the following notion of strong cuspidality introduced in \cite{BKP}. Let $P$ be a standard parabolic subgroup of $G$ with  unipotent radical $U_P$ and Levi quotient $M_P$. A function $f$ is strongly cuspidal if for any proper standard parabolic $P$
$$\int_{u\in U_P(\N F) \backslash U_P(\N\A')} f(ug) du=0. $$
Note that the strong cuspidality implies cuspidality, because $$\int_{u\in U_P( F) \backslash U_P(\A')} f(ug) du= \int_{u'\in U_P(F)U_P(\N \A') \backslash U_P(\A')}\int_{u\in U_P(\N F) \backslash U_P(\N\A')} f(uu'g) dudu'.$$

\subsection{Unramified functions and geometric interpretations}

Let $K$ be a compact open subgroup of $G(\A')$. We consider the subspace $\S' (G(F)\backslash G(\A')/K) \subset \S' (G(F)\backslash G(\A'))$ of right $K$-invariant functions. We denote the groupoid $G(F)\backslash G(\A')/K$ by $\bun_G(C,K)$. When $K=G(\O')$, we have the identification of groupoids
 $$ \bun_G(C,K) \simeq \text{$\bun_G(C)$  : the groupoid of generically trivial $G$-bundles on $C$}.  $$

The proof of this isomorphism is similar to the classical case when the curve is over finite field (see \cite{BKP} Proposition 1.2). We call $\S'(\bun_G(C))$ the space of unramified functions.

There is also a geometric interpretation of the constant-term operator. But here we only consider the formal groupoid interpretation. For that we  consider the groupoids
$$ Q\bun_{M}(C,K):= M(F)U(\A')\backslash G(\A') /K$$
and 
$$ Q\bun_{P}(C,K):= P(F)\backslash G(\A') /K, $$
and the correspondence
$$ Q\bun_{M}(C,K) \xleftarrow{p} Q\bun_{P}(C,K) \xrightarrow{q} \bun_G(C,K), $$
where the maps $p$ and $q$ are the natural quotient maps. Here we write $M=M_P$ and $U=U_P$ to simplify the notation.

The constant-term operator $f \mapsto CT_P(f)$ restricts to a map 
$$CT_P^K: \S'(\bun_G(C,K)) \to \mathbb{C}(Q\bun_M(C,K)),$$
and is essentially the same as the operator $p_*q^*$ (see Proposition 2.6 for a precise formulation).

\subsection{Iwasawa decomposition for square-zero extension} In this subsection, for simplicity, we assume $G=\GL_n$ (the results and proofs work also for any semisimple group of type A) \footnote{We choose to work with these groups for two technical reasons. First, this guarantees that every standard maximal parabolic subgroup of $G$ is cominuscule, and that implies that $U_P$ is commutative and that the adjoint action of $U_P$ on $\mathfrak{g}/\mathfrak{p}$ is trivial. Second, the classification of the orbits of the adjoint action of $M$ on $\mathfrak{g}/\mathfrak{p}$ is easy to describe. See the proof of Lemma \ref{local iwasawa} for why these properties are desired. }, and fix the Borel $B$ to be the upper triangular matrices.

Let $P$ be the standard parabolic subgroup of type $(n_1,n_2)$ of $G$ and let $m=\min \lbrace n_1,n_2\rbrace$. We consider $K=G(\O')$. Let $\Phi(U_P)$ denotes the set of roots associated to $U_P$. We let $U_P^-$ be the unipotent radical of the opposite parabolic subgroup of $P$ and identify it with the space of $n_2 \times n_1$ matrices. Let $D$ denote a collection of effective divisors $(D_1,...,D_m)$ on $\overline{C}$, and $f_{D_i}$ denote the idele associated to $D_i$ i.e. $v_p(f_{D_i})$ equals the multiplicity of a point $p$ in $D_i$. We set
$$ g_D = \diag(f_{D_1}^{-1},..., f_{D_m}^{-1},0,...,0) \in U_P^{-}(\bA'),$$
and  let $D_1 \leq ... \leq D_m$.

\begin{prop}
\begin{enumerate}
    \item (Iwasawa decomposition) $G(\A')=\bigsqcup_DP(\A')(1+\epsilon g_D)G(\mathcal{O})$
    \item We have the following decomposition of groupoids: 
    $$Q\bun_P(C) = \bigsqcup_D P(F) \backslash P(\A')(1+\epsilon g_D)G(\O')/G(\O')$$  $$Q\bun_M(C) = \bigsqcup_D M_P(F)U_P(\A') \backslash P(\A')(1+\epsilon g_D)G(\O')/G(\O')   $$

    \item Let $P(\O')[D]:= P(\A') \cap g_D G(\O')g_D^{-1}$ and $M_P(\O')[D]$ be the image of $P(\O')[D] \to M_P(\A')$. Then we have the following identification of groupoids 
     $$ M_P(F) \backslash M_P(\A')/ M_P(\O')[D] \simeq Q\bun_{M_P}(C,D)$$
\end{enumerate}
    
\end{prop}
\begin{proof}
    The proof for (1) follows by the next local lemma, and (2) is a immediate consequence of (1).  The identification in (3) is given by 
    $$  M_P(F) \backslash M_P(\A')/ M_P(\O')[D] \xrightarrow{\simeq} Q\bun_{M_P}(C,D) : m \mapsto m\cdot g_D. $$
    The proof that this is well-defined and indeed a bijection is the same as that in \cite[Corollary 4.5]{BKP}.
    
\end{proof}
\begin{rem}
    When $G=GL_3$,the decomposition has a simpler description: it is parametrized by effective divisors on $\overline{C}$.
\end{rem}

Let $\overline{K}$ be a local field with valuation $v$. Let $a_i \in \overline{K}$ and $0 \leq v(a_1)\leq ... \leq v(a_m)$. Let $$\phi_{a_1,...,a_m} \in \mathfrak{u}_P^-(\overline{K})$$ be the diagonal element with the $i$-th entry being $a_i$ for the first m entries and zero for $i>m$, and $$g_{a_1,...a_m}=1+\epsilon \phi_{a_1,...,a_m} $$.
\begin{lem}\label{local iwasawa}
    Let $K$ be a square-zero extension of the local field $\overline{K}$ with ring of integers $\bO$.  We have:
    \begin{enumerate}
        \item $\g(\overline{K}) / (\mathfrak{p}(\overline{K})+\g(\bO))= \bigsqcup_{(a_1,...,a_m)} \Ad (P(\bO))\cdot \phi_{a_1,...,a_m}$
        \item $G(K)=\bigsqcup_{(a_1,...,a_m)}P(K)g_{a_1,...,a_m} G(\O')$
        
    \end{enumerate}
\end{lem}
\begin{proof}
    First, we note that the adjoint action of $P(\bO)$ on $\g(\overline{K})$ descends to $\g(\overline{K}) / (\mathfrak{p}(\overline{K})+\g(\bO))$, so to prove (1) it suffices to prove the following claims:
   \begin{itemize}
       \item $\phi_{a_1,...,a_m}$ are in different $P(\bO)$-orbits 
       \item Every element $x$ in $\g(\overline{K}) $ is in the orbit of some $\phi_{a_1,...,a_m}$.
   \end{itemize}
  In fact, the adjoint action of $P(\bO)$ descends to an action of $M(\bO)$ because $U_P(\bO)$ acts trivially on $\g(\overline{K}) / (\mathfrak{p}(\overline{K})+\g(\bO))$ (one can check this by explicit matrix calculation for $G=GL_n$, but in general for any reductive group $G$, this is also true when $P$ is cominuscule). Now we note that given $(A,B) \in M_P(\bO)$, and $X \in \mathfrak{u}_P^-$, $$(A,B) \cdot X = A^{-1}XB.$$
    So now both claims follow by the elementary divisor theory for PIDs. This finishes the proof for (1).

    Now we prove (2). By the Iwasawa decomposition $G(\overline{K})= P(\overline{K})G(\bO)$, the $P(K)-G(\O')$ double cosets in $G(\overline{K})$ have representatives of the form $1 + \epsilon X$, where $X \in \mathfrak{g}(\overline{K})$. We show that $X$ can be modified to be some $\phi_{a_1,...,a_m}$. Indeed, let $p \in P(\O')$ and $g\in G(\O')$ and consider 
    $$1+\epsilon X' := p(1+\epsilon X) g,$$
    then $\overline{p}\overline{g}=1 \in G(\overline{K})$ so we can write $p=(1+\epsilon Y)\overline{p}$ and $g=\overline{p}^{-1}(1+\epsilon Z)$, where  $Y\in \mathfrak{p}(\overline{K}) $ and $Z \in \mathfrak{g}(\bO)$. It follows that $X'= Y+Z+\Ad(\overline{p})\cdot X$. By (1), we see that $X'$ can be of the form $\phi_{a_1,...,a_m}$ with a suitable choice of $Y$ and $Z$. Also, if $\phi_{a_1,...,a_m}$ and $\phi_{a_1',...,a_m'}$  are in different $P(\bO)$-orbits, then 
$P(K)g_{a_1,...,a_m}G(\O') \neq P(K)g_{a_1',...,a_m'}G(\O')$.

\end{proof}

Returning to the global situation,  we set  $$Q\bun_{M,D}(C):=M_P(F)U_P(\A') \backslash P(\A')g_DG(\O')/G(\O'),$$
then we can view the constant-term operator $CT_P$ as a collection of operators $$CT_{P,D}:\S'(\bun_G(C)) \to \mathbb{C}(Q\bun_{M,D}(C)): CT_{P,D}f (l):= \int_{U_P(F)\backslash U_P(\A')} f(ulg_D)du,$$
and the space of unramified cuspidal functions is just the intersection of the kernels of all $CT_{P,D}$'s. 

\subsection{Groupoid theoretic formulation of constant-terms and Eisenstein series}

We keep the notations from the previous subsection. We will work with groupoids, functions on their isomorphism classes, and the groupoid theoretic (also called stacky) pushforward of functions. We refer to Appendix A of BKP for the general definitions and background.

We consider the following correspondence $$ Q\bun_{M_P,D}(C) \xleftarrow{p_D} Q\bun_{P,D}(C) \xrightarrow{q_D} \bun_G(C) $$
where $p_D$ and $q_D$ are the restrictions of the maps $p$ and $q$. 

Similar to the theory of classical automorphic forms for curves over finite fields, we can define scalar products on our functional spaces.

\begin{defn}
    Let $f,h \in \S'(\bun_G(C))$, we define
    $$\langle f,h\rangle = \sum_{x\in \bun_G(C)} \frac{f(x)\overline{h(x)}}{|Aut(x)|}.$$
     When $f,h \in \mathbb{C}(Q\bun_{M,D}(C))$ and assume at least one of them has finite support, then we define 
    $$ \langle f,h\rangle= \sum_{x\in Q\bun_{M,D}(C)} \frac{f(x)\overline{g(x)}}{vol(Aut(x))}. $$

\end{defn}

\begin{rem}
    We can also define for $f,h \in \S'(G(F)\backslash G(\A'))$ the scalar product $\langle f,h\rangle_{adelic} = \int_{G(F)\backslash G(\A')} f(x)\overline{h(x)} dx$, where the Haar measure on $G(F)\backslash G(\A')$ is normalized such that $G(\O')$ has volume 1. The two definitions agree when we restrict the adelic version to the subspace of right $G(\O')$-invariant functions.
\end{rem}

\begin{prop}
Let $f\in \S'(\bun_G(C))$ and $h\in \S'(Q\bun_M(C))$.
\begin{enumerate}
    \item We have $$CT_{P,D}f(l)= \frac{1}{vol(U_{l,D})}p_{D*}q_D^*f(l),$$ where $p_{D*}$ and $q_D^*$ are the groupoid-theoretic pushforward and pullback, and 
    $$ U_{l,D}:= U(\A') \cap lg_DG(\O')g_D^{-1}l^{-1}.$$
    \item Let us denote $E_{P,D}=q_{D,*}p_D^*$,
    then we have $$\langle E_{P,D}f,h \rangle= \langle f, p_{D,*}q_D^*h\rangle$$
    \item  Let $Eis^{unr}$ be the sum of the images of $E_{P,D}$ for all maximal standard parabolic $P$. Then we have $\S'_{cusp}(\bun_G(C)) \oplus Eis^{unr} \subset\S'(\bun_G(C)) $\footnote{The inclusion is an equality if and only if $\S'_{cusp}(\bun_G(C)^\perp \cap \mathbb{C}_{cusp}(\bun_G(C))=0$ by an argument in \cite{F}. But we do not know if this is true.  }.
\end{enumerate}
    
\end{prop}
\begin{proof}
    (1) follows from direct computations and comparison for the $CT_{P,D}$ and $p_{D*}q_D^*$. (2) is a formal consequence of the stacky inner product and the adjunction of stacky pull-push. (3) is just a direct consequence of (2).
\end{proof}
\subsection{Ramified automorphic functions and Eisenstein series}
We keep the notation and assumptions from  the previous section. We can generalize our groupoid theoretic constuctions of constant-terms and Eisenstein series to more general compact open subgroup $K$ of $G(\A')$.

Let $K$ be a normal compact open subgroup of $G(\O')$. We generalize the construction of $CT_{P,D}$ to $\S'(\bun_G(C))$. We have

\begin{prop}
Let $E_{P,D,K}: \S'(Q\bun_{M_P}(C,D,K) \to \S'(\bun_G(C,K)) :$
\begin{enumerate}
    \item We have 
    $$ \langle E_{P,D,K}(f),g\rangle= \langle f, CT_{P,D}^K (g) \rangle  $$
    \item Let $Eis_K$ be the sum of the images of $(E_{P,D,K})$ for all standard maximal parabolic $P$ and all $D$, then we have $ \S'_{cusp}(\bun_G(C,K)) \oplus Eis_K \subset \S'(\bun_G(C,K)) $.
\end{enumerate}
\end{prop}

\begin{defn}\label{eis defn}
    Let $Eis= \sum_{K} Eis_K$, and we call an element in $Eis$ an Eisenstein series.
\end{defn}

\subsection{Hecke-finite functions and finitary functions}
Recall that a smooth subrepresentation $V \subset \S'(G(F)\backslash G(\A'))$ is $admissible$ if $\dim_\mathbb{C}V^K$ is finite for all compact open subgroups $K$ of $G(\A')$. A function $f$ is $finitary$ if it is contained in some admissible representation, or equivalently, $V_f$ is admissible, where $V_f$ is the $G(\A')$-representation generated by $f$.

We have the following proposition analogous to a standard fact in the classical theory.
\begin{prop}\label{heckefinite}
    Let $f \in \S'(G(F)\backslash G(\A'))^K $ and $V_f$ be the $G(\A')$-representation generated by $f$. Then we have $V_f^K = \Hk \cdot f$. In particular, we have that $f$ is $\Hk$-finite if and only if $V_f^K$ is finite-dimensional. Thus, a spherical vector is Hecke-finite if and only if $V_f^{G(\O')}$ is finite-dimensional.
    
\end{prop}
\begin{proof}
    Given $h \in K\backslash G(\A')/K$, we have $T_h \cdot f (g) = \int_K f(gkh)dk$. Let $H_h= K \cap hKh^{-1}$, we can rewrite the integral as a finite sum
    $$ \int_K f(gkh)dk = vol(H_h) \sum_{h' \in K/H_h} f(gh'h),$$
    which shows that $\Hk \cdot f \subset  V^K $.
    Conversely, suppose we are given $v \in V^K$, we can write it as $v=\sum_{i}c_i g_i \cdot f$, where $c_i \in \mathbb{C}$ and $g_i \in G(\A')$. Since $v$ is right $K$-invariant, we have 
    $$ v(g)=\int_Kv(gk)dk= \sum_{i}c_i \int_K (g_i \cdot f)(gk)dk= \sum_{i}c_i \int_K f(gkg_i)dk,$$
    and hence it follows that $ V^K \subset \Hk \cdot f $

\end{proof}
The proposition implies:
\begin{prop}\label{heckeadmiss}
    Let $V$ be a smooth subrepresentation of $\S'(G(F)\backslash G(\A'))$. If for every compact open subgroup $K$ of $G(\A')$ and every nozero element $f\in V^K$, $f$ is $\Hk$-infinite, then $V$ has no admissible subrepesentation. 
  
\end{prop}

\section{Automorphic representation theory and the orbit decomposition}\label{secorbitdecomp}
In this section, we discuss the representation-theoretic approach in the study of automorphic functions in our setting. The main tool is a global Mackey theory for totally disconnected locally compact groups (see \ref{globalmackey} below). Using this we will recall the orbit decomposition for $\S'(G(F)\backslash G(\A'))$ in \ref{orbitdecomp}. Here (and for the rest of the paper), we will work with smooth representations. All induction functors are compact inductions unless specified otherwise.

\subsection{A global Mackey theory for totally disconnected locally compact groups}\label{globalmackey}
Now we discuss a global Mackey theory following \cite[Section 3.2]{BKP}. In loc. cit. the global Mackey theory is developed for the extension
$$ 1 \to \g\otimes \N (\bA') \to G(\A') \to G(\bA') \to 1,$$
yielding the orbit decomposition for $\S'(G(F)\backslash G(\A'))$ (which we will recall in \ref{orbitdecomp}).

The theory presented in loc. cit. can be adapted in a more general setting, and our goal here is to give a treatment of this generalization, which will be essential for our study of more complicated nilpotent orbits.

\subsubsection{First construction  }
Let $G$ be a totally disconnected locally compact group and $N$ be a commutative normal subgroup with quotient $H$; that is, we have a short exact sequence
$$ 1 \to N \to G \to H \to 1.$$
Let $\Gamma$ be a discrete subgroup of $G$, and set $N_\Gamma := N \cap \Gamma$ and let $H_\Gamma$ to be the image of $\Gamma$ in $H$. Also, we assume that the quotient $N/N_\Gamma$ is compact. Let $\Xi$ be the group of continuous characters of $N/N_\Gamma$, which is equipped with an action of $G$ by conjugation.

Throughout this section, we assume the following:

\begin{assump}
   For any compact open subgroup $K$ of $G$, there are finitely many continuous characters $N/N_\Gamma$ such that the restrictions to $ K \cap N$ are trivial. 
\end{assump}

Let $\S'(\Gamma \backslash G)$ be the locally constant functions on $\Gamma \backslash G$ with compact support. It is equipped with a $G$-action by right translations.

Let $\mathbb{C}_{lc}(N_\Gamma \backslash G)$ be the space of locally constant functions on $N_\Gamma \backslash G$. It is equipped with a $G$-action by right translations and a $N/N_\Gamma$-action by left translations. Suppose $\eta \in \Xi$, we let 
$$\mathbb{C}_{lc}(N_\Gamma \backslash G)_\eta = \lbrace f \in \mathbb{C}_{lc}(N_\Gamma \backslash G) | f(ng)=\eta(n) f(g) \text{    for $n\in N$} \rbrace, $$
and we have a collection of mutually orthogonal projectors 
$$ \Pi_\eta: \mathbb{C}_{lc}(N_\Gamma \backslash G) \to \mathbb{C}_{lc}(N_\Gamma \backslash G)_\eta : \Pi_\eta f(g):= \int_{N/N_\Gamma} f(ng)\eta(n)^{-1}dn,$$
where the measure $dn$ is normalized such that $\int_{N/N_\Gamma} dn=1$.

There is a well-defined action of $\Gamma$ on $\Xi$ defined by $(g \cdot \eta) (n) :=  \eta (gng^{-1})$, and this action descends to an action of $H_\Gamma$ on $\Xi$ (because $N_\Gamma$ acts trivially). Now, we denote by $\Omega_\eta$ the $H_\Gamma$-orbit of $\eta$ in $\Xi$.

\begin{prop}\label{basicmackey}

\begin{enumerate}
    \item The sum 
    $\sum_{\eta \in \Xi} (\Pi_\eta f)(g)$
    is finite for every $g \in G$, and is equal to $f(g)$.

    \item If $f$ is left $\Gamma$-invariant, then for any $H_\Gamma$-orbit $\Omega$ in $\Xi$,
    the sum $$\Pi_\Omega f :=\sum_{ \eta \in \Omega} \Pi_\eta f$$ is also left $\Gamma$-invariant, and if $f$ has compact support, then so does $\Pi_\Omega f$. In other words, we obtain
    $$ \Pi_\Omega: \S'(\Gamma\backslash G) \to \S'(\Gamma\backslash G). $$

    \item Let $\S'(\Gamma\backslash G)_\Omega \subset \S'(\Gamma\backslash G)$ be the image of $\Pi_\Omega$, then we have $$ \S'(\Gamma\backslash G)= \bigoplus_\Omega \S'(\Gamma\backslash G)_\Omega $$
    
\end{enumerate}
    
\end{prop}

\begin{proof}
    The proofs for (1) and (2) are the same as the proofs for \cite[Lemmas 3.9 and 3.10]{BKP}. 
    
    (3) follows by (1) and (2).
\end{proof}

\subsection{A global Mackey description for $\S'(\Gamma\backslash G)_\Omega$}
Let $\tilde{\Xi}$ be the group of characters of $N$ (we view $\Xi \subset \tilde{\Xi}$). Then we have an action of $G$ on $\tilde{\Xi}$ defined by $g \cdot \eta (n)= \eta(gn g^{-1})$. If $K$ is an compact open subgroup of $G$, we denote $\tilde{\Xi}_K$ be the subgroup of characters that are trivial on $K \cap N$. We denote $\Delta_\Omega(K)$ to be the set of $K$-orbits in $\tilde{\Xi}_K$. 

Given $\eta \in \Xi $ we denote the stabilizer of $\eta$ in $G$ by $G_\eta$ and the image of $G_\eta$ in $H$ by $H_\eta$. We also set $\Gamma_\eta= \Gamma \cap G_\eta$. It is clear that $N \subset G_\eta$ and hence we have a short exact sequence 
$$ 1 \to N \to G_\eta \to H_\eta \to 1.$$

\begin{defn}(\cite[Definition 3.2]{BKP}) Let 
 $\tS_\eta \subset \mathbb{C}_{lc}(\Gamma_\eta \backslash G)_\eta$ consist of the subspace of locally constant functions $f$ with compact support modulo $N$ and $f(ng)=\eta(n)f(g)$ for every $n \in N$.
\end{defn}

The following proposition is useful for proving cuspidality of functions in $\S'_{\Omega_\eta}$.

\begin{prop}\label{pushforward formula}
Assume that for every compact open subgroup $K$ of $G$, the intersection $\Gamma \cdot \eta \cap (N\cap K)^\perp $ is finite, where $(N\cap K)^\perp$ denotes the set of characters of $N$ whose restrictions to $N \cap K$ are trivial. Then the map 
    $$\kappa_\eta: \tilde{S}_\eta \to \S'_{\Omega_\eta}: (\kappa_\eta f)(g) := \sum_{\gamma \in \Gamma_\eta \backslash \Gamma}f(\gamma g)  $$
    defines an isomorphism of $G$-representations. 
\end{prop}
\begin{proof}
    This is essentially Lemma 3.13 in \cite{BKP}. The assumption that $\Gamma \cdot \eta \cap (N\cap K)^\perp $ is finite guarantees that the map $\kappa_\eta$ is well-defined, i.e. $ \sum_{\gamma \in \Gamma_\eta \backslash \Gamma}f(\gamma g)$ is a finite sum for each $g$. The proof that $\kappa_\eta$ is bijective is the same as in $loc.$ $ cit.$
\end{proof}

Also, let $\S'_\eta$ be the space of locally constant functions $f$ with compact support on $\Gamma_\eta \backslash G_\eta$ such that $f(ng)=\eta(n)f(g)$ for $n \in N$. This is equipped with a $G_\eta$-action. The next proposition follows immediately:

\begin{prop}\label{cpt induction}
    We have an isomorphism of $G$-representations $$\ind_{G_\eta}^G \S'_\eta \simeq \tilde{\S'}_\eta.$$
\end{prop}

\subsubsection{Multiple-step Induction}
Given a short exact sequence
$$ 1 \to U(G_\eta) \to G_\eta \to M(G_\eta) \to 1,$$
where $U(G_\eta)$ is a commutative normal closed subgroup of $G_\eta$, and assume that $U(G_\eta)/ (\Gamma_\eta\cap U(G_\eta))$ is compact, we can further decompose the $G_\eta$-representation $S_\eta$ into subrepresentations $S_{\eta,\chi}$ parametrized by conjugacy classes of characters $\chi$ of $U(G_\eta)/ (\Gamma_\eta\cap U(G_\eta))$ using the procedures defined previously. More precisely, let $$S_{\eta,\chi} = \ind_{\St(\chi)}^{G_\eta} \S'_{\St(\chi)},$$ where $\St(\chi)$ is the centralizer of $\chi$ in $G_\eta$ and $\S'_{\St(\chi)}$ denotes the space of locally constant functions $f$ with compact support on $(\St(\chi)\cap \Gamma_\eta) \backslash \St(\chi)$ such that $f(ng)=\chi(n)f(g)$ for $n\in U(G_\eta)$. Let $\S'_{\Omega_\eta,\chi}= \kappa_\eta(\ind_{G_\eta}^G \kappa_\chi S_{\eta,\chi})$, where $(\kappa_\chi f)(g):=\sum_{\gamma \in (\St(\chi)\cap \Gamma_\eta) \backslash \Gamma_\eta} f(\gamma g)$,  then we have 
$$ \S'_{\Omega_\eta} = \bigoplus_\chi \S'_{\Omega_\eta,\chi}. $$

\begin{prop} We have an isomorphism of $G$-representations
$$\kappa_{\chi,\eta}: \ind_{\St(\chi)}^G \S'_{\St(\chi)} \xrightarrow{\simeq} \S'_{\Omega_\eta,\chi}, $$
where $(\kappa_{\chi,\eta} f)(g):=\sum_{\gamma \in (\St(\chi)\cap \Gamma) \backslash \Gamma} f(\gamma g). $
\end{prop}
 
\subsubsection{$K$-invariant vectors of induced representations} We keep the notation in \ref{globalmackey}. We further assume that $N$ is the union of an increasing family of compact subgroups \footnote{For example, the additive group $\bA'$ of adeles satisfies this condition, while $\bA'^*$ does not.}.

Let $W$ be a smooth subrepresentation of $\S'_\eta$.

\begin{prop}
(The global version of \cite[Proposition 3.5]{BKP})\label{sphvecgen}
    Let $K$ be a compact open subgroup of $G$, then we have the following formula that computes the $K$-invariants of the induced representation, namely,
    $$(\ind_{G_\eta}^{G }W)^K \simeq \bigoplus_{g\in G_\eta \backslash G / K, \text{ } \Ad(g^{-1}) \cdot \eta \in \Delta_\Omega(K)} W^{gKg^{-1} \cap G_\eta}$$    
\end{prop}
\begin{proof}(Sketch, the proof is essentially the same as Proposition 3.5 in $loc.$ $cit.$)
    Let $f \in (\ind_{G_\eta}^{G }W)^K $ i.e. for any $k \in K$ we have 
    $$ f(g^{-1})=f(g^{-1}k)= \eta(g^{-1}kg)f(g^{-1}),$$
    and this implies $f=0$ unless $\psi_\eta$ is trivial on $N \cap g^{-1}Kg$ i.e. $\Ad(g^{-1}) \cdot \eta \in \tilde{\Xi}_K$.
    
    Also, for any $g$ such that $\Ad(g^{-1})\cdot \eta \in  \tilde{\Xi}_K$ and a function $w\in W^{gKg^{-1} \cap \St_\eta}$, we can define a function $f_{w,g,K} \in (\ind_{G_\eta}^{G }W)^K$ by

    \[
f_{w,g,K}(g') =
\left\{ 
    \begin{array}{l}
        s \cdot w \text{   ,if $g'=sgk$ for some $s \in G_\eta$ and $k\in K$} \\
        0 \text{  if $g' \notin G_\eta g K$}
    \end{array} 
\right.
\]
And finally, observe that if $g \cdot \eta $ and $g' \cdot \eta$ agree in $ \Delta_\Omega(K) $, they define the same functions in $(\ind_{G_\eta}^{G}W)^K$.
\end{proof}

For each $\eta$, we denote by $J_\eta$ the Jacquet functor from the category $R_{\tS_\eta}$ of smooth $G$-subrepresentations of $\tS_\eta$ to the category $ R_{\S'_\eta}$ of smooth $G_\eta$-subrepresentations of $\S'_\eta$ , taking a smooth $G$-subrepresentations $V$ to its $\eta$-coinvariants with respect to $N$.
\begin{prop}\label{jacquet} We have:
    \begin{enumerate}
        \item The functor $J_\eta$ is left adjoint to the induction functor, and both of them are exact.
        \item The functors $J_\eta: R_{\tS_\eta} \to R_{\S'_\eta}$ and $\ind_{G_\eta}^G : R_{\S'_\eta} \to R_{\tS_\eta}$ are equivalences of categories. In particular, any smooth subrepresentation $V$ of $\tS_\eta$ is isomorphic to $\ind_{G_\eta}^G W$ for some smooth $G_\eta$-subrepresentation $W$ of $\S'_\eta$.
        
    \end{enumerate}
\end{prop}
\begin{proof}
    This is the global version of \cite[Lemma 3.4 and Proposition 3.5(3)]{BKP}. Since $N$ is an increasing union of compact subgroups, (1) follows from \cite[Proposition 10]{B}. The proof for (2) is essentially the same as that of in \cite[Proposition 3.5(3)]{BKP}.
\end{proof}

\subsection{The orbit decomposition for square-zero extension}\label{orbitdecomp} Now we specialize our construction from the previous subsection to construct the orbit decomposition for the automorphic representation of $G(\A')$. Nothing is new here, but we include it for the reader's convenience.  From now on, we fix a nontrivial character $\psi:k \to U(1)$ which induces 
$$ \psi_{\bC}: \omega_{\bC}(\bA')/\omega_{\bC}(\bF) \to U(1): \alpha \mapsto \psi(\Sigma_p Res_p\alpha).$$

Let $\Na$ be the kernel of the reduction map $G(\A') \to G(\bA')$. We have an isomorphism $$\g \otimes \N(\bA') \xrightarrow{\sim }\Na : X \mapsto 1+\epsilon X. $$
Let $\Nf=\Na \cap G(F)$, then the above isomorphism restricts to an isomorphism $\g \otimes \N (\overline{F}) \simeq \Nf$.

Let $\Xi$ be the group of characters of $\g \otimes \N(\bA') /\g \otimes \N (\overline{F})$ which is isomorphic to $\g^\vee \otimes \N^{-1} \omegac(\bF)$. For each $\eta \in \g^\vee \otimes \N^{-1} \omegac(\bF)$, we define a character
$$ \g \otimes \N(\bA') /\g \otimes \N (\overline{F}) \to U(1):  X \mapsto \psi_\eta(X):=\psi(\langle \eta, X \rangle).$$

Now, we consider the space $\mathbb{C}_{l.c.}(\Nf \backslash G(\A'))$ of locally constant functions on $\Nf \backslash G(\A')$, then it is equipped with a left $\Na \!  / \!\Nf$-action (by left translation) and a right $G(\A')$-action (by right translation). Since  the group $\Na \!  / \!\Nf$ is compact abelian, we have a decomposition of $G(\A')$-representation
$$\mathbb{C}_{l.c.}(\Nf \backslash G(\A')) = \bigoplus_{\eta \in \g^\vee \otimes \N^{-1} \omegac(\bF) } \mathbb{C}_{l.c.}(\Nf \backslash G(\A'))_\eta, $$
where $\mathbb{C}_{l.c.}(\Nf \backslash G(\A'))_\eta= \lbrace f \in \mathbb{C}_{l.c.}(\Nf \backslash G(\A') | f(lg)=\psi_\eta(l)f(g) \text{  for $l\in \Na$} \rbrace$. There is a natural projection $\Pi_\eta: \mathbb{C}_{l.c.}(\Nf \backslash G(\A') )\to \mathbb{C}_{l.c.}(\Nf \backslash G(\A'))_\eta$ given by
$$\Pi_\eta f (g) = \int_{\Na \!  / \!\Nf} \psi_\eta(u)^{-1} f(ug)du, $$
where we normalize $du$ such that the volume of $\Na \!  / \!\Nf$ is 1. Then for each $G(\bF)$-orbit $\Omega$ in $\g^\vee \otimes \N^{-1} \omegac(\bF)$ we can define an operator (Proposition \ref{basicmackey})
$$ \Pi_\Omega = \sum_{\eta\in \Omega} \Pi_\eta : \S' (G(F) \backslash G(\A')) \to \S' (G(F)\backslash G(\A')),  $$
and this gives a decomposition $$\S' ( G(F) \backslash G(\A')) = \bigoplus_\Omega \S' (G(F)\backslash G(\A'))_\Omega.$$ 
From now on,  we will call $\S' (G(F)\backslash G(\A'))_\Omega$ an orbit summand in $\S' ( G(F) \backslash G(\A'))$).

\subsubsection{An alternative description for $\S' (G(F)\backslash G(\A'))_\Omega$}

Let $\eta  \in \g^\vee \otimes \N^{-1} \omegac(\bF)$. There is an action of $G(\A')$ on $\bigoplus_{\eta' \in \Omega_\eta}\mathbb{C}_{l.c.}(\Nf \backslash G(\A'))_{\eta'}$. We denote the stabilizer by $\St_\eta$, its image in $G(\bA')$ by $\St_\eta(\bA')$ and $\St_\eta(F):=\St_\eta \cap G(F)$.  We have the short exact sequence:
$$ 1 \to \Na \to \St_\eta \to \St_\eta(\bA') \to 1.$$

Now we let $$\tS_\eta=\lbrace f\in \mathbb{C}_{l.c.}(\St_\eta(F) \backslash G(\A')) | f(ug)=\psi_\eta(u)f(g) \text{  for $u\in \Na$}\rbrace $$
We have the following isomorphism of $G(\A')$-representation (Proposition \ref{pushforward formula} )
$$ \kappa_\eta: \tS_\eta \xrightarrow{\sim}  \S' (G(F)\backslash G(\A'))_\Omega, $$
where $\kappa_\eta f(g):=\sum_{\gamma\in \St_\eta(F) \backslash G(F)} f(\gamma g)$.

Let $\S'_{\St_\eta}$ be the space of locally constant functions $f$ on $\St_\eta(F)\backslash \St_\eta$ such that $f(ug)=\psi_\eta(u)f(g)$ for $u\in \Na$ and its support is compact modulo $\Na$. Proposition \ref{cpt induction} gives the following isomorphism of $G(\A')$-representations $$\ind_{\St_\eta}^{G(\A')} \S'_{\St_\eta} \simeq \tS_\eta.$$

\subsubsection{Induced representation and $K$-invariants}
Let $\eta \in \Xi \subset \tilde{\Xi} \simeq \g^\vee\otimes \N^{-1} \omegac (\bA')$, where $\tilde{\Xi}$ denotes character group of $\Na$. Let $\tilde\Omega_\eta$ be the $G(\bA')$-orbit of $\eta$ in $\tilde{\Xi}$, and $K$ be a compact open subgroup of $G(\A')$. We denote $\tilde{\Xi}_K \subset \tilde{\Xi}$ to be the subgroup of characters whose restriction to $\Na \cap K$ are trivial. Let $\Delta_\Omega(K):= (\Omega\cap \tilde{\Xi}_K)/ \overline{K}  $, where $\overline{K}$ is the image of $K$ along the reduction map $G(\A')\to G(\bA')$.
Let $W$ be a smooth subrepresentation of $\S'_{\St_\eta}$. Applying Proposition \ref{sphvecgen} we have the following proposition:
\begin{prop}\label{sphvec}
    Let $K$ be a compact open subgroup of $G(\A')$, then we have the following formula computing the $K$-invariants of the induced representation, namely,
    $$(\ind_{\St_\eta}^{G(\A') }W)^K \simeq \bigoplus_{g\in \St_\eta \backslash G(\A') / K, \text{ } \Ad(g^{-1}) \cdot \eta \in \Delta_\Omega(K)} W^{gKg^{-1} \cap \St_\eta}$$
\end{prop}

This yields the following corollary, which reduces the non-admissibility of $\tS_\eta$ to that of $\S'_{\St_\eta}$.
\begin{lem}\label{nonadmissible stablizer}
    If $\S'_{\St_\eta}$ has no admissible subrepresentation, then neither does $\tS_\eta \simeq \ind_{\St_\eta}^{G(\A') }\S'_{\St_\eta}$.
\end{lem}
\begin{proof}
    First, let $\eta \in \tilde{\Xi}_{K''}$ for some compact open subgroup $K''$ of $G(\A')$.
    Now, suppose $V$ is an admissible subrepresentation of $\ind_{\St_\eta}^{G(\A') }\S'_{\St_\eta}$, then there exists a smooth subrepresentation $W$ of $\S'_{\St_\eta}$ such that $V\simeq \ind _{\St_\eta}^{G(\A') } W$. Since $W$ is not admissible, there exists a compact open subgroup $K'\subset \St_\eta$ such that $dim_\mathbb{C}W^{K'}=\infty$. Let $K'''$ be a compact open subgroup such that $K''' \cap \St_\eta$ is contained in $K'$, then let $K:=K''\cap K'''$ and we have $W^{K\cap \St_\eta} \supset W^{K'}$ being infinite-dimensional. It is clear that $\eta \in \tilde{\Xi}_{K}$, so it follows that $V^K$ is infinite-dimensional.
\end{proof}

\subsection{Hitchin fibers and Fourier transform}
We recall how $\tS_\eta^{G(\O')}$ is related to some functions on the Hitchin fibers and interpret the embedding $\kappa_\eta: \tS_\eta^{G(\O')} \to \S'(\bun_G(C))$ via a Fourier transform following \cite{KP}.

\subsubsection{Hitchin fibers}
\begin{defn}
    Let $\M_\eta^{\N^{-1}} (\overline{C})$ be the groupoid of $\N^{-1}$-twisted $G$-Higgs bundles on $\overline{C}$ i.e., the pairs $(P_0 , \phi )$, where $P_0 \in \bun_G(\overline{C})$ and $  \phi \in H^0(\overline{C}, \omegac \N^{-1} \otimes \g_{P_0} )$ and $\phi$ is generically in the $G(\bF)$-orbit of $\eta$.  Similarly, we let  $\M_\eta^{\N^{-1}} (C)$ be the groupoid of the pairs $(P \in \bun_G(C), \phi \in H^0(\overline{C}, \omegac \N^{-1} \otimes \g_{\overline{P}} ))$ such that the restriction of $\phi$ to the generic point of $\bC$ lies in the $G(F)$-orbit $\Omega_\eta$ of $\eta$.
\end{defn}

 We have the following diagram
\[
\begin{tikzcd}[row sep=2.5em, column sep=4em]
& \M_\eta^{\N^{-1}}(C)
    \arrow[dl,"p'"']
    \arrow[dr,"q'"] & \\
\M_\eta^{\N^{-1}}(\overline{C})
    \arrow[dr,"p"]
&
&
\bun_G(C)
    \arrow[dl,"q"'] \\
& \bun_G(\overline{C}) &
\end{tikzcd}
\]
where $p$ (and $q'$) forgets the Higgs field $\phi$ and $q$ (and $p'$) pulls back the underlying $G$-bundles along $\overline{C} \to C$. Note that $q$ (resp. $p'$) is a torsor where for each point $P_0 \in \bun_G(\overline{C})$ (resp. $(P_0,\phi)\in \M_\eta^{\N^{-1}}(\overline{C})$) the fiber is the vector space $H^1(\overline{C},\mathfrak{g}_P \otimes \N)$, and hence the fibers of $p$ and $q$ are dual vector spaces by Serre's duality.

\begin{rem}
    When $C=\overline{C} \times_{\Spec k} \Spec k[\epsilon]/\epsilon^2$, by deformation theory, we can view $\bun_G(C)$ as the groupoid of pairs $(P_0,\xi)$, where $\xi \in H^1(\overline{C},\mathfrak{g}_P)$, and the map $q$ just forgets the lifting $\xi$.
\end{rem}

There are natural identifications (\cite{BKP}, Proposition 3.20)
\begin{align}
    \M_\eta^{\N^{-1}} (C) \simeq \St_\eta(F)\backslash G(\A')_\eta/G(\O') \text{  and  }\M_\eta^{\N^{-1}} (\overline{C}) \simeq \St_\eta(\bF)\backslash G(\bA')_\eta/G(\bO),
\end{align}
where $  G(\bA')_\eta=\lbrace g\in G(\bA')| \Ad(g^{-1})\eta \in \mathfrak{g} \otimes \N\inv \omega (\bO) \rbrace$ and  $G(\A')_\eta=\lbrace g\in G(\A')| \Ad(g^{-1})\eta \in \mathfrak{g} \otimes \N\inv \omega (\O') \rbrace$. 

\subsubsection{$\mathbb{C}^*$-torsor $L_\psi$ and Fourier transform} Note that $p'$ is a $H^1(\overline{C},\N\otimes \mathfrak{g}_{P_0})$-torsor, and given $\phi \in H^0(\overline{C},\omega\N\inv\otimes \mathfrak{g}_{P_0}^\vee)$, we compose this with the character $\psi:k \to \mathbb{C}^*$ to get a map
$$H^1(\overline{C},\N\otimes \mathfrak{g}_{P_0}) \to \mathbb{C}^*. $$
Now, we push out $p'$ along this map and obtain a $\mathbb{C}^*$-torsor $L_\psi$ on the groupoid $\M^{\N^{-1}} (\overline{C}) $, and we have an identification of $\tS_\eta^{G(\O')}$ with the space $\S'(\M_\eta^{\N^{-1}} (\overline{C}),L_\psi)$ of finitely supported sections on $\M_\eta^{\N^{-1}}(\overline{C})$.

\begin{prop}(See 2.3 and Proposition 2.6 in  \cite{KP} for the precise formulation) 
The embedding $\kappa_\eta$ is identified with the Fourier transform  $\S'(\M_\eta^{\N^{-1}} (\overline{C}),L_\psi) \to \S'(\bun_G(C))$.
    
\end{prop}

Unraveling the idelic description of $\M_\eta^{\N^{-1}}(C)$ and using the Fourier description of $\kappa_\eta$,  we have the following geometric interpretation of Proposition \ref{sphvec}:
\begin{prop}\label{geom_sphvec}
    An element $g$ with $\Ad(g^{-1}) \cdot \eta \in \tilde{\Xi}_{G(\O')}$ and a function $w \in \S'_{\St_\eta} ^{gG(\O')g^{-1}\cap \St_\eta}$ define a function $\delta_{g,w} \in \S'(\M_\eta^{\N^{-1}} (C))$ (see proof of Proposition 3.7 how such a function is constructed). If an element $P$ is in the support of $\delta_{g,w}$, then $pp'(P)$ is the $G$-bundle on $\overline{C}$ corresponding to $\overline{sg} \in G(\bA')$, where $s$ is in the support of $w$.
    
    Also, given $f\in \S'(\M_\eta^{\N^{-1}} (C))$, then support of $\kappa_\eta f$ is equal to $q^{-1}(pp'(\text{supp}(f)))$.
\end{prop}

\section{Almost regular elliptic orbits and mixed orbits for connected split reductive groups}\label{regellip and mixed}

Let $G$ be a connected split reductive group. In Section \ref{regellip}, we introduce a notion called \textit{almost regular elliptic} for semisimple elements in the Lie algebra $\mathfrak{g}$ of $G$. We prove that a function $f\in \S'(G(F)\backslash G(\A'))$ is strongly cuspidal if and only if $f \in \bigoplus_{\Omega} \S'_\Omega$ where $\Omega$ ranges over the almost regular elliptic orbits. We also show that the unramified strongly cuspidal functions form a finite-dimensional space when $G$ is semisimple. In Section \ref{nonregss and mixed}, we show that $\S'_{\Omega_\eta}$ contains no finitary or Hecke-finite functions when $\eta$ is a mixed type element whose semisimple part is not elliptic. As a result, to prove Conjecture \ref{conj1}(1) for $\PGL_3$, it remains to study the nilpotent orbits because all elliptic elements in $\mathfrak{pgl}_3$ are regular.

\subsection{Almost regular elliptic orbits}\label{regellip}
In this subsection, we assume the characteristic of $\overline{F}$ is greater than $c(G)-1$, where $c(G)$ is the Coxeter number of $G$. This guarantees that every element in $\mathfrak{g}$ has Jordan decomposition \cite[Proposition 48]{Mc}.  Recall the following definition. 
\begin{defn}
   Let $\eta \in \mathfrak{g}(\bF)$ and let $C(\eta)^\circ$ be the identity component of the centralizer of $\eta$ in $G$. Then $\eta$ is elliptic if it is nonzero semisimple and  $C_G (\eta)^\circ/Z(G)^\circ$ contains no nontrivial split central torus. An elliptic element is regular if $C_G (\eta)^\circ$ is a torus. 
\end{defn}
For our purpose of characterizing orbits $\Omega$ such that $\S'_\Omega$ contains strongly cuspidal functions, we also introduce the following definition.
\begin{defn}\label{are def}
    An element $\eta \in \mathfrak{g}(\bF)$ is \textit{almost regular elliptic} if it is nonzero semisimple and $C_G (\eta)^\circ/Z(G)^\circ$ contains no nontrivial split torus. We will also denote almost regular elliptic by ARE.
\end{defn}
\begin{lem}
  Let $G$ be a split reductive group over $\overline{F}$ whose Lie algebra is $\mathfrak{g}$. Let $x \in \mathfrak{g}(\bF)$. Then we have
  \begin{enumerate}
      \item $x$ is regular elliptic $\implies$ $x$ is almost regular elliptic $\implies$ $x$ is elliptic.
      \item $x$ is almost regular elliptic if and only if $x$ is not contained in any $G(\overline{F})$-conjugate of a standard $\overline{F}$-parabolic subalgebra of $\mathfrak{g}$.

  \end{enumerate}

\end{lem}
\begin{proof}
(1) follows from the definitions. We now prove (2). Since any proper parabolic subalgebra over $\overline{F}$ is $G(\overline{F})$-conjugate to standard parabolic subalgebra, it suffices to prove $x$ is almost regular elliptic if and only if $x$ is not contained in any proper parabolic subalgebra over $\overline{F}$. 

First, we assume that $x$ is semisimple. Suppose $x$ is almost regular elliptic. If $x$ is contained in a proper parabolic subalgebra, then $C_G (x)^\circ$ contains a split torus $\mathbb{G}_m$ that is not in $Z(G)^\circ$, but this contradicts to the assumption that $x$ is almost regular elliptic. Conversely, suppose $x$ is not contained in any proper parabolic. If $x$ is not almost regular elliptic, then there exists a split torus $\mathbb{G}_m \to C_G(x)^\circ$ and its image in $G$ is not in the center of $G$. This defines a grading $\mathfrak{g}=\bigoplus_{i\in \mathbb{Z}}\mathfrak{g}_i$ and $x$ is in the Levi subalgebra $\mathfrak{g}_0$ of the proper parabolic subalgebra $\bigoplus_{i\geq 0} \mathfrak{g}_i$, hence a contradiction.

In general for $x$ not necessarily semisimple, we consider the Jordan decomposition $x=x_n+x_{ss}$, and assume $x_n \neq 0$. But $x_n$ is contained in some parabolic subalgebra $\mathfrak{p}$ (the canonical choice is the parabolic associated to the Jacobson–Morozov triple of $x_n$) that contains its centralizer in $\mathfrak{g}$, and this implies $x_{ss} \in \mathfrak{p}$, and hence $x \in \mathfrak{p}$. 
\end{proof}

\begin{prop}\label{are for gln}
    If $G=\GL_n$, then $x\in \mathfrak{g}(\overline{F})$ is regular elliptic if and only if it is almost regular elliptic.
\end{prop}
\begin{proof}
    Suppose $x$ is almost regular elliptic, we need to show it is regular elliptic. Note that $C_G(x)^\circ$ is a twisted Levi subgroup of $G$ and has maximal split torus equal to $Z(G)^\circ$. But \cite[Lemma 4.3.4]{FS} implies that $C_G(x)^\circ \simeq \operatorname{Res}_{E/\overline{F}}\GL_d$, where $E$ is a degree $\frac{n}{d}$ extension of $\overline{F}$. If $d>1$, then $\operatorname{Res}_{E/\overline{F}}\GL_d$ contains a split torus of dimension $d$ and this contradicts to the almost regularity of $x.$ Therefore, $d=1$ and $C_G(x)^\circ \simeq \text{Res}_{E/F} \mathbb{G}_m$ is a torus, and hence $x$ is regular elliptic.
\end{proof}
\begin{rem}
    For more general groups $G$, Proposition \ref{are for gln} does not necessarily hold. For example, when $G=\text{Sp}_4$, there are nontoral twisted Levi subgroups $M$ of $G$ that contain no split torus. See  \cite[Example 5.9.1]{De}.
\end{rem}

    The following result is a generalization of \cite[Proposition 6.5]{BKP}. 
\begin{prop}\label{str cusp = are}
A function $f$ is strongly cuspidal if and only if $f \in \bigoplus_{\Omega: \text{ARE}} \S'_{\Omega}$.
\end{prop}
\begin{proof}

    Let $f \in \S'(G(F)\backslash G(\A'))$, $P$ be a maximal parabolic subgroup of $G$ and $U$ be its unipotent radical.We show the vanishing of the constant-terms of $\Pi_{\Omega_\eta} f$. Consider
    
  \[\int_{U(\N F) \backslash U(\N \A')} \Pi_{\Omega_\eta} f(ug) du=\sum_{\eta\in \Omega_\eta}  \int_{X \in \mathfrak{g}(\N\A')/\mathfrak{g}(\N F)}  \int_{a \in \mathfrak{ u}(\N\A')/\mathfrak{u}(\N F)} \psi_\eta(-X)f((1+X+a)g)dadX  \]
   
   \[=\sum_{\eta\in \Omega_\eta}  \int_{X \in \mathfrak{g}(\N \A')/\mathfrak{g}(\N F)}  \int_{a \in \mathfrak{u}(\N\A')/\mathfrak{u}(\N F)} \psi_\eta(-X)\psi_\eta(a)f((1+X)g)dadX
    \]

Since $\eta$ is almost regular elliptic, it is not contained any proper parabolic subgroup. Hence $\psi_\eta$ is nontrivial when restricted to $\mathfrak{u}$ for any parabolic subalgebra $\mathfrak{p}$, and hence the integral vanishes.

For the other direction, we need to show that if $f$ is strongly cuspidal then the projected image $\Pi_{\Omega_\eta} f=0$ for all orbits $\Omega_\eta$ that are not almost regular elliptic.

Now let $\eta$ is not almost regular elliptic, and hence it is contained in some parabolic subalgebra (by choosing a representative in $\Omega_\eta$ we may assume $\eta$ is contained in some standard parabolic subalgebra $\mathfrak{p}$ with nilpotent radical $\mathfrak{u}$), and hence the restriction $\psi_\eta|_{\mathfrak{u}(\N \A')}$ is trivial. Now for $g_o\in G(F)$, we have
  \begin{align*} \Pi_{g_0\eta} f(g) &= \int_{X \in \mathfrak{g}(\N\A')/\mathfrak{g}(\N F)}   \psi_{g_0\eta}(-X)f((1+X)g)dX \\
 &= \int_{X \in \mathfrak{g}(\N\A')/\mathfrak{g}(\N F)}   \psi_{\eta}(-X)f(g_0(1+X)g_0^{-1}g)dX  \\
 &= \int_{X \in \mathfrak{g}(\N\A')/(\mathfrak{g}(\N F)+\mathfrak{u}(\N\A'))} \int_{a \in \mathfrak{u}(\N\A')/\mathfrak{u}(\N F)}  \psi_{\eta}(-(X+a))f((1+X+a)g_0^{-1}g)da\;dX  \\
 &= \int_{X \in \mathfrak{g}(\N\A')/(\mathfrak{g}(\N F)+\mathfrak{u}(\N\A'))} \int_{a \in \mathfrak{u}(\N\A')/\mathfrak{u}(\N F)}  \psi_{\eta}(-X)f((1+X+a)g_0^{-1}g)da\;dX, 
 \end{align*}

 The last equality holds because $\psi_\eta$ is trivial on $\mathfrak{u}(\N\A')$. Finally, the strong cuspidality implies the last integral vanishes, and hence $\Pi_{\Omega_\eta}f$ is zero.

\end{proof}

\begin{cor}\label{regellipfd}
 Let $G$ be semisimple. The space of unramified strongly cuspidal functions is finite-dimensional.
\end{cor}
\begin{proof}
    First, we observe that a Higgs bundle with almost regular elliptic Higgs field is necessarily semistable. Since the number of $k$-points of the moduli of semistable Higgs bundles are finite and the unramifield strongly cuspidal functions are finitely supported on the semistable Higgs bundles, the finite-dimensionality follows.
\end{proof}

Proposition \ref{are for gln} and Proposition \ref{str cusp = are} imply that
\begin{prop}
    Let $G=\GL_n$. A function $f$ is strongly cuspidal if and only if $f \in \bigoplus_{\Omega: \text{reg. ellip.}} \S'_{\Omega}$.
\end{prop}

\subsection {Non-regular semisimple and mixed orbits}\label{nonregss and mixed}
We show that for certain mixed-type orbits (see Proposition \ref{mixed}), i.e. the orbits $\Omega_\eta$ where $\eta$ has nontrivial semisimple and nilpotent part, the $G(\A')$-representations $\S'_{\Omega_\eta}$ contain no finitary or Hecke-finite functions.

Recall that we have the $\St_\eta$-representation $\S'_{\St_\eta}$ and an isomorphism $\ind_{\St_\eta}^{G(\A')}\S'_{\St_\eta} \simeq \S'_{\Omega_\eta}$ of $G(\A')$-representations. 
\begin{lem}\label{gmquot}
Let $\eta \in \mathfrak{g} \otimes \N\inv\omega(\bF)$. Suppose that the center of $\St_\eta(\bA')$ contains a subgroup $L \simeq \bA'^*$ (and we denote by $L_{\bF}$ the subgroup of $L$ that is isomorphic to $\bF^*$), and there is a homomorphism  $p:\St_\eta(\bA')\to \bA'^*$. Let $K$ be a compact open subgroup of $\St_\eta$, and suppose that there exists $h \in L$ such that $\deg(p(h))\neq0$. Then for every smooth $\St_\eta$-subrepresentation $V$ of $\S'_{\St_\eta}$, we have $\dim V^K=\infty \text{ or  }0  $ .
\end{lem}
\begin{proof}
 Consider the following composite
 $$ d:(\Na \St_\eta(F)) \backslash \St_\eta \simeq \St_\eta(\bF) \backslash \St_\eta(\bA')  \to \bA'^*/\bF^* \to \Pic(\overline{C}) \to \mathbb{Z},$$
 where the first map is induced by the natural surjection $\St_\eta \to \St_\eta(\bA')$, the second map is induced by the surjection $p$ and the last map is the degree map.

Let $\eta^\perp \subset \Na$ be the kernel of $\eta$ and consider the following commutative diagram with exact rows
    \[
\xymatrix{
  1 \ar[r] & N_{\mathbb{A}} \ar[d]^\eta \ar[r] & \St_{\eta} \ar[d] \ar[r] & \St_{\eta}(\bA') \ar[r] \ar[d]& 1 \\
  1 \ar[r] & \omega(\bA') \ar[r] & \St_{\eta}/\eta^\perp \ar[r] & \St_{\eta}(\bA') \ar[r] & 1
}
\]
We denote by $p'$ the quotient map $St_\eta/\eta^\perp \to \St_\eta(\bA')$. Since the action of $\St_\eta$ on $\S'_{\St_\eta}$ (and hence on $V$) factors through $\St_\eta/\eta^\perp$, it suffices to show $ V^{K'}$ is infinite-dimensional, where $K'$ is the images $K$ in $\St_\eta/\eta^\perp$. Now, we fix a lift $h''$ of $h$ in $\St_\eta / \eta^\perp$. For $s \in \St_\eta/\eta^\perp$, we let $c(s)=h''^{-1}sh''s^{-1}$. Since $h$ is in the center of $\St_\eta(\bA')$, we have $p'(c(s))=1$ for any $s \in \St_\eta/\eta^\perp$, and hence $c(s)\in \omega(\bA')$. Note that $c$ defines a continuous homomorphism $\St_\eta/\eta^\perp \to \omega(\bA')$.

Let $v \in V^{K' }$ be a nonzero element, and $k\in K'$, for any integer $m$, we have
$$ k \cdot (h''^m\cdot v)= h''^m \cdot ((h''^{-m}kh''^m) \cdot v).$$
Since $c(k)$ is central, we have $ h''^{-m}kh''^m= c(k)^mk$, and therefore
$$  k \cdot (h''^m\cdot v) =  h''^m \cdot ((c(k)^mk) \cdot v) =h''^m(c(k)^m\cdot v)= \psi_\eta(c(k))^m ( h''^m \cdot v).$$
Since $\psi_\eta$ has finite image, and there are infinitely many choices of $m$ such that $\psi_\eta(c(k))^m=1$ for all $k \in K'$, and hence $h''^m \cdot v \in V^{K'}$. Since the union of the supports of $h''^{m}\cdot v$'s is not compact (because its image in $\mathbb{Z}$ via the above map $d$ is not compact), $V^{K'}$ is infinite-dimensional.

\end{proof}

\begin{prop}\label{mixed} Let $\eta=\eta_{ss}+\eta_n$ be the Jordan decomposition of a nonzero element $\eta \in \mathfrak{g} \otimes \N \inv \omega(\bF)
$. 
\begin{enumerate}
    \item If $\eta=\eta_{ss}$ and is not elliptic, then all functions in $\S'_{\Omega_\eta}$ are Hecke-infinite.
    \item If both the semisimple and nilpotent parts are nonzero and $\eta_{ss}$ is not elliptic, then all functions in $\S'_{\Omega_\eta}$ are Hecke-infinite
\end{enumerate}

\end{prop}

\begin{proof}
We first prove (1). Let $K$ be a compact open subgroup of $\St_\eta$ and $V$ be a smooth $\St_\eta$-subrepresentation of $\S'_{\St_\eta}$. By Proposition \ref{sphvec}, it suffices to show $V^K$ is infinite-dimensional if it is not zero. Now, by the lemma, it suffices to find a central subgroup $\bA'^* \simeq L \subset \St_\eta(\bA')$, a homomorphism $p:\St_\eta(\bA')\to \bA'^*$ and $h \in L$ with $\deg(p(h))\neq 0 $.

    The centralizer $\St_\eta(\bA')=M(\bA')$ is a proper (twisted) Levi subgroup of $G(\bA')$. $M$ is a proper Levi and its center contains a split torus $\mathbb{G}_m'$ because $\eta$ is not elliptic and we can take $L=\mathbb{G}_m'(\bA')$. Now, we consider 
    $$\pi: \mathbb{G}_m'(\bA') \hookrightarrow M(\bA') \to M/[M,M] (\bA') \to \mathbb{G}_m(\bA'),$$
    where the second map is the map induced by the canonical homomorphism $M\to M/[M,M]$ of algebraic groups, and the third map is a quotient to a split torus (this exists again because $\eta$ is not elliptic). The resulting composition $\pi$ is the induced map on $\bA'$-points of a nontrivial isogeny. Hence we can find $h\in \mathbb{G}_m'(\bA') $ such that $\deg(\pi(h))\neq0$. This proves (1).
    
    To prove (2), we note that the central split torus $\mathbb{G}_m'$ that centralizes $\eta_{ss}$ also centralizes $\eta_n$, because $\eta_n \in C_{\mathfrak{g}}(\eta_{ss})$. So, we have a central split torus contained in $C_G(\eta)$. Note that we have
    $$C_G(\eta) \hookrightarrow M \to M/[M,M]\to \mathbb{G}_m$$
    Now, the same argument for the proof of (1) shows that all functions in $\S'_{\Omega_\eta}$ are Hecke-infinite.

\end{proof}

\begin{cor}\label{pgl3mixed} For $G=\PGL_3$, and let $\eta \in \N \inv \mathfrak{g}\otimes \omega(\bF)$ be a mixed type element (i.e. with nontrivial semisimple and nilpotent parts in the Jordan decomposition), then all functions in $\S'_{\Omega_\eta}$ are Hecke-infinite. 
\end{cor}
\begin{proof}
    We observe that the semisimple part $\eta_{ss}$ of $\eta$ cannot be a regular element, and all elliptic elements in $\mathfrak{pgl}_3$ are regular, and hence the corollary follows from the proposition . Indeed, if $\eta_{ss}$ is regular, then the centralizer of $\eta_{ss}$ is a Cartan. But $[\eta_{ss},\eta_n]=0$ and the only nilpotent element in a Cartan is 0. Also, let $\eta$ be semisimple, we note that $\eta$ is elliptic if and only if any lift of $\eta$ to $\mathfrak{gl}_3$ is elliptic, and an element $\eta$ in $\mathfrak{gl}_3$ is elliptic if and only if its minimal polynomial $f_\eta$ is irreducible. But that $f_\eta$ is irreducible implies that $f_\eta$ is either of degree 1 or degree 3. If $f_\eta$ is of degree 1, then $\eta=0$ because the characteristic of $k$ is not 3. If $f_\eta$ is of degree 3, then the characteristic polynomial of $\eta$ is irreducible and hence separable, and $\eta$ is regular in this case. 
\end{proof}

\section{Regular nilpotent orbits} \label{regnilpmain}
We collect some general facts about nilpotent orbits in Section \ref{prelimnilp} and study the regular nilpotent orbit for split reductive groups $G$ in \ref{regnilp}. 
\subsection{Some preliminary facts about nilpotent orbits }\label{prelimnilp}

We collect some facts about the nilpotent elements and nilpotent orbits $\S'_\eta$, which allow us to pass the study of $\tS_\eta$ to the automorphic representation for $\St_\eta(\bA')$.  First, we need a fact for nilpotent orbits for general split reductive groups $G$. This was already used in \cite{BKP} for $G=\PGL_2$, and their observation generalizes easily.
\begin{lem}

  Let $\eta^\perp$ be the subspace of elements in $\Na$ that is orthogonal to $\eta$. Consider the diagram
    \[
\xymatrix{
  1 \ar[r] & \Na \ar[d]^\eta \ar[r] & \St_{\eta} \ar[d] \ar[r] & \St_{\eta}(\bA') \ar[r] \ar[d]& 1 \\
  1 \ar[r] & \omega(\bA') \ar[r] & \St_{\eta}/\eta^\perp \ar[r] & \St_{\eta}(\bA') \ar[r] & 1
}
\]
Then we have the following:
\begin{enumerate}
    \item The bottom sequence is a central extension.
    \item Let $\mathfrak{st}_\eta$ be the Lie algebra of $\St_\eta$. If $\mathfrak{st}_\eta (\N \bA')$ is contained in $\eta^\perp$, then the bottom sequence is right split and hence $\St_\eta/\eta^\perp \simeq \omega(\bA') \times \St_\eta(\bA')$
    \item If $\eta$ is nilpotent, then $\mathfrak{st}_\eta(\N \bA')$ is contained in  $\eta^\perp$.
\end{enumerate}
    
\end{lem}
\begin{proof}
    (1) follows from the definitions. (2) can be seen by noting that $\St_\eta(\mathbb{A})/\mathfrak{st}_\eta(\bA') \simeq \St_\eta(\bA') $, so the right split to the central extension exists when $\mathfrak{st}_\eta(\N\bA') \in \eta^\perp $.

    For (3), let $x\in \mathfrak{st}_\eta(\N \bA')$, then $(\ad_x\ad_\eta)^r=\ad_x^r\ad_\eta^r$, and since $\eta$ is nilpotent there exists a positive integer $r$ such that $(\ad_x\ad_\eta)^r=0$, and hence $\text{tr}(\ad_x\ad_\eta)=0$ which means $x \in \eta^\perp$.
\end{proof}

Recall that $\S'_\eta$ is the $\St_\eta$-representation consisting of the locally constant and compactly supported functions $f$ on $\St_\eta(F)\backslash \St_\eta$ such that $f(lg)=\psi_\eta(l)f(g)$ for $l\in \Na$. Now, the lemma implies:
\begin{prop}
    The stabilizer $\St_\eta$ of a nilpotent element $\eta$ acts on $\S'_\eta$ via the quotient $\St_\eta \to \omega(\bA') \times \St_\eta(\bA')$, where $\omega(\bA')$ acts via $\psi_\eta$. We have an isomorphism $$\S'_{\eta} \simeq \S'(\St_\eta(\bF) \backslash \St_\eta(\bA'))\times \psi_\eta $$ of $\St_\eta$-representations
\end{prop}

We have the following proposition.

\begin{prop}\label{nilpnonadm}
Let $\eta\in \mathfrak{g}\otimes\omega\N\inv(\bF)$ be a nilpotent element. Let $V\simeq \ind_{\St_\eta}^{G(\A')} (W,\psi_\eta) $ for some smooth subrepresentation $W$ of $\S'(\St_\eta(\bF) \backslash \St_\eta(\bA'))$.

\begin{enumerate}
    \item Let $K$ be a compact open subgroup of $G(\A')$. If for any $g_0 \in G(\bA')$ such that $\Ad(g_0^{-1}) \cdot \eta \in \tilde{\Xi}_K$, there exist infinitely many $g  \in \St_\eta(\bA') \backslash G(\bA') / \overline{K} $ such that $\Ad(g^{-1}) \cdot \eta $ are in different $K$-orbits and for any smooth $\St_\eta(\bA')$-subreresentation $W'$ of $W$ we have
$$W'^{g\overline{K}g^{-1} \cap \St_\eta(\bA')}\supset W'^{g_0\overline{K}g_0^{-1}\cap \St_\eta(\bA')}, $$
then all nonzero elements in $V^K$ are $\Hk$-infinite.
\item The condition in (1) is satisfied for all $K$, then $V$ has no admissible subrepresentation.
\end{enumerate}
\end{prop}
\begin{proof}
 Let $f\in V^K$ and $V_f$ be the $G(\A')$-representation generated by $f$. By the previous proposition and proposition \ref{jacquet}, we have $V_f \simeq \ind_{\St_\eta}^{G(\A')}(W'_f,\psi_\eta)$, for some smooth $\St_\eta$-subrepresentation $W'_f$ of $W$. By Proposition \ref{heckefinite}, we must show that $\dim V_f^K =\infty$. Now applying Proposition \ref{sphvec}, we have 
 $$ V_f^K \simeq \bigoplus_{g \in \St_\eta(\bA') \backslash G(\bA') / \overline{K}, \text{ } \Ad(g^{-1}) \cdot \eta \in \Delta_K}  W_f'^{g\overline{K}g^{-1} \cap \St_\eta(\bA')},$$
 and there is a $g_0$ such that $\Ad(g_0^{-1}) \cdot \eta \in \tilde{\Xi}_K$ and ${W_f'}^{g_0\overline{K}g_0^{-1} \cap \St_\eta(\bA')}$. But now the assumption that there exists infinitely many $g$ such that ${W_f'}^{g\overline{K}g^{-1} \cap \St_\eta(\bA')}\supset {W_f'}^{g_0\overline{K}g_0^{-1}\cap \St_\eta(\bA')}$ and $\Ad(g^{-1}) \cdot \eta $ are in different $K$-orbits implies $\dim V_f^K =\infty$. This proves (1), and (2) follows by Proposition \ref{heckeadmiss}.
\end{proof}

\subsection{Nilpotent elements}
In this subsection, we recall some standard facts about nilpotent elements in $\mathfrak{g}$. Our reference is \cite{McT} and \cite{CM}.

\begin{defn} 
    A cocharacter $\phi: \mathbb{G}_m \to G$ defines a $\mathbb{Z}$-grading on $\mathfrak{g}=\oplus_{i \in \mathbb{Z}} \mathfrak{g}_i$, where $\mathfrak{g}_i:= \lbrace x \in \mathfrak{g}| \Ad(\phi(t)) \cdot x = t^i x\rbrace$. This also defines a parabolic subalgebra $$\mathfrak{p}_\phi:= \oplus_{i \geq 0} \mathfrak{g}_i$$
    We say that the cocharacter $\phi$ or the parabolic subalgebra $\mathfrak{p}_\phi$ is associated to a nilpotent element $\eta \in \mathfrak{g}$ if $\eta \in \mathfrak{g}_2$ and there is a maximal torus $S$ of the centralizer $C_G(\eta)$ such that the image of $\phi$ lies in $(L,L)$ where $L=C_G(S)$.
\end{defn}
\begin{prop}
    Let $\eta$ be a nilpotent element. We have:
    \begin{enumerate}
        \item There is a cocharacter $\phi$ associated to $\eta$. If $\phi$ and $\phi'$ are two such cocharacters, then $\mathfrak{p}_\phi=\mathfrak{p}_{\phi'}$.
        \item $\eta$ is regular i.e. $\dim C_G(\eta)$ is equal to the rank of $G$ if and only if $\eta$  lies in exactly one Borel subalgebra $\mathfrak{b}$.
    \end{enumerate}
\end{prop}
\begin{prop}\cite[5.2.4]{McT} \label{regnilpelem}
    Let $\eta$ be a regular nilpotent element. 
    \begin{enumerate}
        \item The centralizer $C_G(\eta)$ is commmutative
        \item The maximal torus of $C_G(\eta)$ is the identity component of the center $Z(G)$ of $G$.
        \item $C_G(\eta)=Z(G)R_u(C_G(\eta))$, where $R_u(C_G(\eta))$ is the unipotent radical of $C_G(\eta)$.
        \item If $\eta$ is contained in a Borel subalgebra $\mathfrak{b}$, then all the highest root subgroups are contained in $C_G(\eta)$. 
    \end{enumerate}
\end{prop}
\begin{proof}
    (1), (2) and (3) can be found in \cite[5.2.4]{McT}. (4) is a direct consequence of the Chevalley commutator formula \cite[Prop. 5.1.14]{Con}.
\end{proof}

\subsection{Regular nilpotent orbits}\label{regnilp} Throughout Section \ref{regnilp}, we assume the characteristic of $\overline{F}$ is at least $c(G)-1$, where $c(G)$ is the Coxeter number of the reductive group $G$.
\subsubsection{Simple groups $G$}Throughout this subsection, we assume $G$ is simple of adjoint type.\footnote{This assumption is not required for most arguments. } We denote by $\eta$ a regular nilpotent element of the Lie algebra $\mathfrak{g}$ of $G$ i.e. a nilpotent element with centralizer of minimal dimension. Since the centralizer of $\eta$ is commutative (Proposition \ref{regnilpelem}), for any character $\chi$ of $\St_\eta(\bF)\backslash \St_\eta(\bA')$, we can consider 
$\tS_{\eta,\chi} := \ind_{\St_\eta}^{G(\mathbb{A})} (\psi_C, \chi)$,
and this gives a decomposition $\tS_\eta =\bigoplus_\chi \tS_{\eta,\chi}$, because $\St_\eta(\bF)\backslash \St_\eta(\bA')$ is compact. But we will give another more useful decomposition below, using only the action of the subgroup $U_h(\bA')/U(\bF)$.

By Proposition \ref{regnilpelem}(4), we can instead consider the action of $U_h(\bA')/U_h(\bF) \subset \St_\eta(\bA')/\St_\eta(\bF)$ on $\S'_\eta$, which gives a decomposition $$\S'_\eta = \bigoplus_\chi \S'_{\eta, \chi}$$ where $\chi$ runs over the additive characters of  the compact group $U_h(\bA')/U_h(\bF)$. Let $\tS_{\eta,\chi}=\ind_{\St_\eta}^{G(\A')} \S'_{\eta,\chi}$, then we have
$$\tS_\eta = \bigoplus_\chi \tS_{\eta, \chi}$$

For the rest of the section, we will fix a regular nilpotent element $\eta$ in $\mathfrak{b}$ as follows. Let $\mathfrak{n}$ be the nilpotent radical of the Lie algebra $\mathfrak{b}$ of our fixed Borel subgroup $B$. We can fix a nonzero element $\eta \in \mathfrak{n}$, then the regularity is equivalent to that the image of $\eta$ in 
$\mathfrak{n} \to \mathfrak{n}/[\mathfrak{n},\mathfrak{n}]=\bigoplus_{\beta \in \Delta} \mathfrak{g}_\beta \to \mathfrak{g}_\beta$
is nonzero for all $\beta \in \Delta$. Let $\alpha \in \omega\N \inv (\bF)$, we set $$\eta=\alpha\sum_{\beta \in \Delta}e_\beta \in \mathfrak{g} \otimes \omega\N \inv (\bF),$$ where $e_\beta$ is the root vector associated to the simple root $\beta$.

\begin{lem}

\begin{enumerate}
    \item Suppose for every $\gamma \in G(F)$ and every maximal standard parabolic $P$ such that $\Ad(\gamma^{-1})\cdot \eta$ is in the Lie algebra $\mathfrak{p}$ of $P$, we have that $\gamma^{-1}U_h\gamma$ is contained in the unipotent radical $U_P$ of $P$. Then $\kappa_\eta \tS_{\eta,\chi}$ is cuspidal for non-trivial $\chi$.

    \item If $\eta$ is a regular nilpotent element in $\mathfrak{b}$, then the condition in (1) is satisfied.
\end{enumerate}
    
\end{lem}
\begin{proof}
    This is a direct generalization of \cite[Lemma 6.6 ]{BKP}. Recall the embedding  $\kappa_\eta: \tS_\eta \to \S'(G(F)\backslash G(\A'))$ is given by $$\kappa_\eta f(g)= \sum_{\gamma \in \St_\eta(F) \backslash G(F)} f(\gamma g).$$

    Let $P$ be a standard parabolic and $U$ be its unipotent radical. If $\Ad(\gamma^{-1})\cdot \eta$ is not in $\mathfrak{p}$, then $\psi_{\Ad(\gamma^{-1})\cdot \eta}|_{\Na \cap U(\A') }$ is nontrivial, and hence $\int_Uf(\gamma ug)du=0$.

    If $Ad(\gamma^{-1})\cdot \eta \in \mathfrak{p}$, then we have
    $$\int_{\gamma^{-1}U_h(\A')\gamma/\gamma^{-1}U_h(F)\gamma} f(\gamma u g)du= \int_{\gamma^{-1}U_h(\A')\gamma/\gamma^{-1}U_h(F)\gamma} \chi(\gamma u \gamma^{-1})f(\gamma g)du,$$
    which vanishes by the non-triviality of $\chi$, and this implies $\int_Uf(\gamma ug)du=0$. This proves (1).

     For (2), indeed we can prove that $\gamma$ lies in $P$ if $\Ad(\gamma^{-1})\cdot \eta\in \mathfrak{p}$, and hence $\gamma^{-1}U_P\gamma= U_P$. Note that since $\mathfrak{p}$ contains the regular nilpotent element $\Ad(\gamma^{-1})\cdot \eta$, it contains the Borel subalgebra $\Ad(\gamma^{-1})\cdot \mathfrak{b}$. Indeed, write $\mathfrak{p}=\mathfrak{l}+\mathfrak{u}$ where $\mathfrak{l}$ and $\mathfrak{u}$ are the Levi factor and nilpotent radical respectively, and since the image $\Ad(\gamma^{-1})\cdot \eta$ in $\mathfrak{l}$ is nilpotent, it is contained in a Borel subalgebra $\mathfrak{b}_\mathfrak{l}$ of $\mathfrak{l}$ but that implies $\Ad(\gamma^{-1})\cdot \eta$ is in $\mathfrak{b}_\mathfrak{l}\oplus \mathfrak{u}=:\mathfrak{b}'$ which is a Borel subalgebra of $\mathfrak{g}$. Since $\Ad(\gamma^{-1})\cdot \eta$ is regular nilpotent and any regular nilpotent element is contained in a unique Borel subalgebra (Proposition 5.5), we have $\mathfrak{b}'=\Ad(\gamma^{-1})\cdot \mathfrak{b}$, and hence $\gamma$ is in $P$ (because any two Borel subalgebras in $\mathfrak{p}$ are conjugated by an element in $P$).
\end{proof}

\begin{prop}\label{regnilpresults}
We have the following:
\begin{enumerate}

     \item  Let $\chi:U_h(\mathbb{A})/U_h(F) \to U(1)$ be a nontrivial character, then $\kappa_\eta\tS_{\eta,\chi}$ is cuspidal.
     \item Let $K$ be an open compact subgroup of $G(\A')$ and assume that $G$ is of type $A$, $B$, $C$ or $G_2$. If $\chi$ is trivial, then all elements in $\tS_{\eta,\chi}^{K}$ are Hecke-infinite.
     \item When $\chi$ is trivial, the $G(\A')$-representation $\tS_{\eta,\chi}$ has no admissible subrepresentation.
\end{enumerate}
\end{prop}
\begin{proof}
    (1) follows from Lemma 5.7. Now, we prove (2). By Proposition \ref{nilpnonadm}, it suffices to show that, for any $g_0\in G(\A')$ such that $\Ad(g_0^{-1}) \cdot\eta \in \tilde{\Xi}_K$ and  for any smooth $\St_\eta$-subrepresentation $W$ of $\S'_\eta$, there exist infinitely many $g \in G(\A')$ such that $\Ad(g^{-1}) \cdot \eta \in \tilde{\Xi}_K$ are in different $\overline{K}$-orbits and $W^{gKg^{-1} \cap \St_\eta} \supset W^{g_0Kg_0^{-1} \cap \St_\eta}$. More precisely, we will show that $$g\overline{K}g^{-1} \cap \St_\eta(\bA')+U_h(\bA') = g_0\overline{K}g_0^{-1} \cap \St_\eta(\bA')+U_h(\bA').$$
    Since the character $\chi$ is trivial,  it will follow that $\S'_{\eta,\chi}^{g\overline{K}g^{-1} \cap \St_\eta(\bA')}$ contains $\S'_{\eta,\chi}^{g_0\overline{K}g_0^{-1}\cap \St_\eta(\bA')}$. The existence of such an infinite family of $g$ follows from Lemma \ref{regnliplem}. Finally, (3) follows by (2) and Proposition \ref{heckeadmiss}.

\end{proof}

\begin{lem}\label{regnliplem}
    Let $G$ be of type $A_n, B_n, C_n, \text{ or }G_2$. Let $\eta$ be a regular nilpotent element and $K$ be any compact open subgroup of $G(\bA')$. Suppose $g_0 \in G(\bA')$ be an element such that $\Ad(g_0^{-1})  \cdot \eta \in \tilde{\Xi}_K$, then there exists an infinite family of elements $g \in G(\bA')$ such that $\Ad(g^{-1})  \cdot \eta \in \tilde{\Xi}_K$ lie in disjoint $K$-orbits and $$gKg^{-1} \cap \St_\eta(\bA')+U_h(\bA')=g_0Kg_0^{-1} \cap \St_\eta(\bA') +U_h(\bA').$$
\end{lem}

Before proving the lemma in full generality, we illustrate how this is done in the case of $G=\PGL_n$ as an example. Let $a\in \bA'^*$, we define the \textit{support} of $a$ to be the set of points $p \in \overline{C}$ such that $a_p \notin \bO_p^* $. Let $\Lambda=\lambda \bO$ be a lattice in $\bA'$, we call define the support of $\Lambda$ to be the support of $\lambda$.

\begin{exmp}When $G=\PGL_n$, the nilpotent element $\eta=\alpha\sum_{i=1}^{n-1} E_{i,i+1}$ is regular. A direct computation shows that the centralizer of $\eta$ in $\PGL_n$ consists of the matrices
    \[
\begin{pmatrix}
1      & x_1    & x_2    & \cdots & x_{n-1} \\
0      & 1      & x_1    & \cdots & x_{n-2} \\
0      & 0      & 1      & \cdots & \vdots   \\
\vdots & \vdots & \vdots & \ddots & x_1      \\
0      & 0      & 0      & \cdots & 1
\end{pmatrix}
\]

We assume $g_0=1$ (in fact the general statement reduces to $g_0=1$ by replacing $K$ with the compact open subgroup $g_0Kg_0^{-1}$). Since $K \cap \St_\eta(\bA')= (K\cap U(\bA')) \cap \St_\eta(\bA')$, we may assume $K$ is a compact subgroup of $U(\bA')$. Now, by considering the exponential map $\exp: \text{Lie}(U) \to U$, it suffices to construct elements $g \in G(\bA')$ such that $\Ad(g^{-1})\eta$ are in different $G(\bO)$-orbits (and hence in different $\overline{K}$-orbits) and show the equality $$(*): g^{-1}\Lambda_Kg \cap \mathfrak{st}_\eta(\bA')+\mathfrak{g}_h(\bA')=\Lambda_K \cap \mathfrak{st}_\eta(\bA')+\mathfrak{g}_h(\bA'), $$
where $\Lambda_K=\log (K  \cap \St_\eta(\bA'))$ is a compact open subgroup $\mathfrak{st}_\eta(\bA')$ and $\mathfrak{g}_h$ is the highest root space.

 If $\beta \in \Phi$, we denote by $\pi_\beta$ the projection map $\mathfrak{g} \to \mathfrak{g}_\beta$ and let $\Lambda_\beta:= \pi_\beta(\Lambda_K \cap \mathfrak{st}_\eta(\bA'))$. We also let $\Lambda_{ij} = \Lambda_{e_i-e_j}$ We consider the matrices $g=\diag(1,...,1, t)$, where $t \in \bO$ is chosen such that $t$ has disjoint support from $\Lambda_\beta$'s. It is clear that there exist infinitely many such $t$ such that $g\cdot \eta$'s are in different $\overline{K}$-orbits $\Ad(g^{-1})\cdot \eta$. It remains to show that the equality $(*)$. 

   Let 
   $$s= 
\begin{pmatrix}
0     & x_1    & x_2    & \cdots & x_{n-1} \\
0      & 0      & x_1    & \cdots & x_{n-2} \\
0      & 0      & 0      & \cdots & \vdots   \\
\vdots & \vdots & \vdots & \ddots & x_1      \\
0      & 0      & 0      & \cdots & 0
\end{pmatrix}
\in \mathfrak{st}_\eta(\bA'),$$
we have 
$$ g^{-1}sg^=\begin{pmatrix}
0      & x_1    & x_2    & \cdots & tx_{n-1} \\
0      & 0     & x_1    & \cdots & tx_{n-2} \\
0      & 0      & 0      & \cdots & \vdots   \\
\vdots & \vdots & \vdots & \ddots & tx_1      \\
0      & 0      & 0      & \cdots & 0
\end{pmatrix}$$

Therefore, we have $s\in g\Lambda_Kg^{-1} \cap \mathfrak{st}_\eta(\bA')+\mathfrak{g}_h(\bA')$ if and only if the $x_{i} \in t^{-1}\Lambda_{i,n} \bigcap_{j=2}^{n-1}\Lambda_{i,j}$ for $n-1 \geq i \geq 1$. But since $t$ has disjoint support from $\Lambda_{j,i}$, we have $t^{-1}\Lambda_{n,i} \bigcap_{j=2}^{n-1}\Lambda_{j,i}= \Lambda_{n,i} \bigcap_{j=2}^{n-1}\Lambda_{j,i}$, and this implies $s \in \Lambda_K\cap \mathfrak{st}_\eta(\bA') +\mathfrak{g}_h(\bA')$ and hence $g\Lambda_Kg^{-1}\cap \mathfrak{st}_\eta(\bA') +\mathfrak{g}_h(\bA') \subset \Lambda_K\cap \mathfrak{st}_\eta(\bA') +\mathfrak{g}_h(\bA')$. It is clear that we also have the reverse inclusion.

\end{exmp}

\begin{proof}[Proof of Lemma \ref{regnliplem}]
   First, we note that by replacing $K$ by $g_0Kg_0^{-1}$, it is enough to prove the proposition for $g_0=1$. 
   
   Let $\alpha_i$ (for $i=1,...,n$) be the simple roots for $\mathfrak{g}$ and $\alpha_n$ is connected to only one other node in the Dynkin diagram. Let $\rho$ be the half sum of the positive coroots. Let $\mathfrak{st}_\eta$ be the Lie algebra of $\St_\eta$, then to prove the equality of groups in the Proposition, it suffices to show 
   $$ \Ad(g^{-1})(\mathfrak{st}_\eta(\bA'))\cap \Lambda_K+\mathfrak{g}_h(\bA')=(\mathfrak{st}_\eta(\bA'))\cap \Lambda_K+\mathfrak{g}_h(\bA'),$$
  where $\mathfrak{g}_h$ is highest root space and $\Lambda_K=\log (K \cap \St_\eta(\bA'))$. Note that we have a weight decomposition 
   $$\mathfrak{st}_\eta(\mathbb{\bA'})= \oplus_{h_i} \mathfrak{st}_\eta(\mathbb{\bA'})(\rho, 2h_i)$$
   with respect to the action of $\rho(\bA'^*)$, where $h_i$ are the exponents of $G$ and each weight space is one dimensional because $G$ is of types $A,B,C,\text{ or } G_2$, . Since the weight of a root $x \in \Phi^+ $ is given by $\langle \rho, x \rangle$ which is also equal to the height of $x$, we have 
   $$ \mathfrak{st}_\eta(\mathbb{\bA'})(\rho, 2h_i) \subset \bigoplus_{\text{ht}(\alpha)=h_i}\mathfrak{g}_\alpha.$$
   
   If $\beta\in \Phi$, we denote by $\pi_\beta :  \mathfrak{st}_\eta(\bA') \to  \mathfrak{g}_{\beta}(\bA')$ the projection to the root space $\mathfrak{g}_{\beta}$. We observe that $ \mathfrak{st}_\eta(\mathbb{\bA'})(\rho, 2h_i)$ is not equal to any root subspace by the Chevalley commutator formula, except for the highest root subspace, and hence for each height $l$ (except the maximal height), we have at least two positive roots $\beta_l$ such that $\pi_{\beta_l}( \mathfrak{st}_\eta(\bA'))$ is nonzero.  Let $s \in \mathfrak{st}_\eta(\bA')$ and write $s=\sum_{h_i}x_is_i$, where $x_i\in \bA'$ and $s_i \in \bigoplus_{\text{ht}(\alpha)=2h_i}\mathfrak{g}_\alpha$. For any positive root $\beta$ we let $\Lambda_\beta \subset \bA'$ be the lattice such that $  \pi_\beta(\Lambda_K \cap \mathfrak{st}_\eta(\bA'))=\mathfrak{g}_\beta(\Lambda_\beta)$. Then we have $s \in \Lambda_K \cap \mathfrak{st}_\eta(\bA')+\mathfrak{g}_h(\bA')$ if and only if $x_i \in \bigcap_{\text{ht}(\beta)=h_i, \pi_\beta(\mathfrak{st}_\eta(\bA')\neq 0} \Lambda_\beta$.

   Let  $g=\omega^\vee(t)g_0$ where $t\in \bA'^*$ with no pole and $\omega^\vee$ is a fundamental coweight such that the following conditions hold:
   \begin{enumerate}
        \item For each exponent $l$ of $G$, except the largest exponent, there is exactly one positive root $\gamma_l$ of height $l$ such that $\pi_{\gamma_l}(\mathfrak{st}_\eta)\neq0$ and $\omega^\vee$ acts trivially on $\mathfrak{g}_{\gamma_l}$.
       \item $t$ has disjoint support from $\Lambda_\beta$ for all positive roots $\beta$.
      
   \end{enumerate}
Let us explain why such $g$ exists. We can choose $\omega$ to be fundamental coweight corresponding to the simple root $\alpha_1$. The condition (1) is satisfied because for every exponent $l$ the one dimensional space $\mathfrak{st}_{\eta,l}$ is not a root space by Chevalley commutator formula and $\omega^\vee$ acts non-trivially only on one root space of height $l$.

It is clear that there exist infinitely many $t$ such that (1) and (2) are satisfied and that $g \cdot \eta$ are in different $G(\bO)$-orbits. It remains to show that these conditions imply
$$\Ad(g^{-1})(\mathfrak{st}_\eta(\bA'))\cap \Lambda_K +\mathfrak{g}_h(\bA')=\mathfrak{st}_\eta(\bA'))\cap \Lambda_K+\mathfrak{g}_h(\bA').$$
Note that $s\in \Ad(g^{-1})((\mathfrak{st}_\eta(\bA'))\cap \Lambda_K+\mathfrak{g}_h(\bA')$ if and only if $$x_i \in \bigcap_{\text{ht}(\beta)=h_i, \pi_\beta(\mathfrak{st}_\eta(\bA')\neq 0, \beta \neq \gamma_i} \Lambda_\beta \bigcap t^{-w_i}\Lambda_{\gamma_i},$$ where $w_i$ denote the ($\omega^\vee$)- weight of $\mathfrak{g}_{\gamma_i}$. But we have 
$$\bigcap_{\text{ht}(\beta)=h_i, \pi_\beta(\mathfrak{st}_\eta(\bA'))\neq 0, \beta \neq \gamma_i} \Lambda_\beta \bigcap t^{-w_i}\Lambda_{\gamma_i}= \bigcap_{\text{ht}(\beta)=h_i, \pi_\beta(\mathfrak{st}_\eta(\bA'))\neq 0} \Lambda_\beta,$$
because $t$ has disjoint support from $\Lambda_{\beta}$'s.


   

\end{proof}

\begin{prop}\label{regnilfd}
 The space $\bigoplus_{\chi\neq1} \tS_{\eta,\chi}^{G(\O')}$ of $G(\O')$-invariant functions is finite-dimensional.
\end{prop}

\begin{proof}
    
    First, we show that for each nontrivial character $\chi$ the space $\tS_{\eta,\chi}^{G(\O')}$ is finite-dimensional. Using Proposition \ref{sphvec}, it suffices to show that there are finitely many $ g \in \St_\eta(\bA')\backslash G(\bA') / G(\bO)$ such that $\Ad(g^{-1}) \cdot \eta \in \tilde{\Xi}_{G(\bO)}$ are in disjoint $G(\bO)$-orbits and that $\S'_{\eta,\chi}^{gG(\bO) g^{-1} \cap \St_\eta (\bA')} \neq 0$, since $\St_\eta(\bF) \backslash \St_\eta(\bA')$ is compact (hence $\S'_{\eta,\chi}^{gG(\bO) g^{-1} \cap \St_\eta (\bA')}$ is finite-dimensional). Since $U_h(\bA')$ acts on $\S'_{\eta,\chi}$ through the character $\chi$,  it suffices to show for all but finitely many $g\in \St_\eta(\bA')\backslash G(\bA') / G(\bO)$, we have $gG(\bO) g^{-1} \cap U_h(\bA') \not\subset U_h(\Lambda_\chi)$, where $\Lambda_\chi$ is the conductor of $\chi$. 

    By the Iwasawa decomposition, it suffices to consider $g\in B(\bA')$, and we write $g=ut$ where $t\in T(\bA')$ and $u \in \St_\eta(\bA') \backslash U(\bA')$. Since
     $u^{-1}  U_h(\bA') u= U_h(\bA')$ by the Chevalley commutator formula, we have 
     $$g G(\bO) g^{-1} \cap U_h(\bA') =t G(\bO) t^{-1} \cap U_h(\bA')$$
     Let $t=\prod_i \alpha_i^\vee(t_{i}^{-1})$. We compute that
     $$t G(\bO) t^{-1} \cap U_h(\bA') = U_h((\prod_i t_i^{-a_i})\bO), $$
      where $a_i=\langle \omega_i^\vee, \alpha_h \rangle$ and $\alpha_h$:the hightest root, and $\omega_i^\vee$ are the $i$-th fundamental coweights.
      
      We first show that there are finitely many choices for $t \in T(\bO)\backslash T(\bA')$, and then we will show for each choice of $t$ there are finitely many choices of $u \in (t^{-1}B(\bO)\cap U(\bA')) \backslash U(\bA')/ \St_\eta(\bA')$. Note that $\Ad(t^{-1}u^{-1})\cdot \eta \in \Xi_{G(\bO)}$ implies that $t_i\alpha$ are regular, since $\pi_\theta(\Ad(u)\cdot \eta)= \pi_\theta(\eta)$ for any $\theta \in \Delta$. Let $\beta \in \omega(\bF)$ correspond to the character $\chi$, then $t G(\bO) t^{-1} \cap U_h(\bA') \subset \Lambda_\chi$ implies that $(\prod_i t_i^{-a_i })\beta $ is regular. Now the regularities of $t_i\alpha$ and $(\prod_i t_i^{-1})\beta $ imply that each $t_i$ has bounded zeros and poles, and hence there are only finitely many choices for $t$. 

    Now we need to show that for each $t$ there are finitely many choices of $u$ it suffices to show that the double coset
    $$ \St_\eta(\bA') \backslash \lbrace u \in U(\bA') | \Ad(t^{-1}u^{-1})\cdot \eta \in \Xi_{G(\bO)} \rbrace / (tG(\bO)t^{-1}\cap U(\bA'))  $$
    is finite. In fact, we can show 
    $$ H(t) :=\St_\eta(\bF) \backslash \lbrace u \in U(\bA') | \Ad(t^{-1}u^{-1})\cdot \eta \in \Xi_{G(\bO)} \rbrace / (tG(\bO)t^{-1}\cap U(\bA'))  $$ is finite. Using the geometric interpretation $\St_\eta(\bF)\backslash G(\bA')_\eta / G(\bO) \simeq \M_\eta^{\N \inv}(\overline{C})$, fixing $t$ amounts to fixing a $T$-bundle $\mathcal{F}$ on $\overline{C}$, and $H(t)$ is identified with the set of pairs $(\mathcal{F}_B,\phi) $ where $\mathcal{F}_B$ is a $B$-bundle whose induced $T$-bundle is $\mathcal{F}$ (we also denote by $\mathcal{F}_G$ the induced $G$-bundle of $\mathcal{F}_B$) and $\phi \in H^0( \overline{C}, \N \inv \omega \otimes \mathfrak{g}_{\mathcal{F}_G})$ such that $\phi$ is conjugate to $\eta$ at the generic point. This set is evidently finite.


    Finally, we show that there are only finitely many $\chi$ such that $\tS_{\eta,\chi}^{G(\mathcal{O})}\neq 0$. In view of the above the argument, it suffices to show there are only finitely $\beta$ such that there exists $t$ with $(\prod_i t_i^{-a_i})\beta $ and all $t_i\alpha$ being regular. But this implies $\alpha^{c-1}\beta\in H^0(\overline{C},\N ^{c-1} \omega^{c})$, where $c$ is the Coxeter number of $G$, and clearly there are only finitely choices of $\beta$.

\end{proof}
\begin{rem}
    We expect that the above proposition hold for all compact open subgroups $K$. However, for non-maximal compact open $K$, we have to consider $g_0 \notin B(\bA')$ which will be more complicated.
\end{rem}
\subsubsection{Semisimple groups $G$} Now, using the results from the previous subsection, we obtain the cuspidality, Hecke-infiniteness and finite-dimensionality of spherical cuspidal subspace for semisimple group $G$. Let $G$ be a semisimple group of adjoint type. We keep the notation from the previous subsection. Let $G_1,..., G_l $ be the simple factors of $G$ and denote their Lie algebras by $\mathfrak{g}_1,..., \mathfrak{g}_l$. We fix a regular nilpotent element $\eta= (\eta_1,...,\eta_n) \in \mathfrak{g} \otimes \N\inv \omega(\bF)$, where $\eta_i \in \mathfrak{g_i} \otimes \N\inv \omega(\bF)$ are regular nilpotent. Let $U_{h_i}$ be the highest root subgroup of $G_i$, and let $U_h= \prod_i U_{h_i}$. Then we have
$$ \tS_\eta = \bigoplus_\chi \tS_{\eta, \chi},$$
where $\chi=(\chi_1,...,\chi_l)$ ranges over the characters of $U_h(\bA')/U_h(\bF)$. and $\chi_i$'s are the characters of $U_{h_i}(\bA')/U_{h_i}(\bF)$.

 Proposition \ref{regnilpresults}, Lemma \ref{regnilpelem} and Proposition \ref{regnilfd} immediately generalize. We obtain:

 \begin{prop}
     Let $G$ be a semisimple group of adjoint type, and let $K$ be a compact open subgroup of $G(\A')$.
     \begin{enumerate}
         \item If all $\chi_i$'s are nontrivial, then $\kappa_\eta(\tS_{\eta,\chi})$ is cuspidal.
         \item Suppose all simple factors of $G$ are of types $A$, $B$, $C$ or $G_2$. If at least one of $\chi_i$'s is nontrivial, then all elements in $\tS_{\eta,\chi}^K$ are $\Hk$-infinite.
         \item Let $K=G(\O')$. Then $\bigoplus_{\chi: \chi_i\neq 1 \text{ for all }i} \tS_{\eta,\chi}^K$ is finite-dimensional.
     \end{enumerate}
 \end{prop}

\section{Subregular orbits in $\mathfrak{pgl}_n$}\label{subregnilp}

Recall that a nilpotent element in $\mathfrak{g}$ is called subregular if it belongs to a nilpotent orbit of the second largest dimension in $\mathfrak{g}$ . When $\mathfrak{g} =\mathfrak{pgl}_n$, a subregular nilpotent element is a nilpotent element of type $(n-1,1)$. Throughout this section, we let $n\geq3$ and 
$$ \eta= \sum_{i=2}^{n-1}\alpha E_{i, i+1} \in \mathfrak{pgl}_n\otimes \N \inv \omega  (\bF),$$
where $\alpha \in \N \inv \omega(\bF)$. A direct computation shows that the group $\St_\eta(\bA')$ consists of the matrices of the following form:
 \[
\begin{pmatrix}
a & 0 & 0 & \cdots & 0 & y \\
x & 1 & z_{n-3} & z_{n-4} & \cdots & z_0 \\
0 & 0 & 1 & z_{n-3} & \cdots & z_1 \\
0 & 0 & 0 & 1 & \ddots & \vdots \\
\vdots & \vdots & \vdots & \ddots & \ddots & z_{n-3} \\
0 & 0 & 0 & \cdots & 0 & 1
\end{pmatrix},
\]
where $a\in \bA'^*$ and $x,y,z_i \in \bA'$. Here we normalize the diagonal entries except $(1,1)$-th to be $1$'s.  More formally, these matrices are 

 $$\diag(a,1,...,1)+ xE_{21}+yE_{1,n}+\sum_{i=0}^{n-3}z_i(E_{2,n-i}+...+E_{2+i,n}).$$
 
 From this explicit description of the centralizer, we also obtain the following structural description of $\St_\eta(\bA')$:
 \begin{prop}
Consider the following subgroups of $\St_\eta(\bA')$:

$$H=\lbrace I+ xE_{21}+yE_{1,n}+z_0E_{2,n} \rbrace$$
$$ H_x= \lbrace I+ xE_{21} \rbrace, \text{      } H_y= \lbrace I+ yE_{1,n} \rbrace $$
$$Z= \lbrace I+\sum_{i=0}^{n-3}z_i(E_{2,n-i}+...+E_{2+i,n})\rbrace$$

$$ Z_0 = \lbrace I+z_0E_{2n} \rbrace$$
$$ A = \lbrace \diag(a,1,...,1) \rbrace$$
where $x,y, z_0,..., z_{n-3} \in \bA'$ and $a \in \bA'^*$

  \begin{enumerate}
      \item The group $H$ is a three-dimensional Heisenberg group with the center $Z_0$. 
      \item  The group $\St_\eta(\bA')$ is isomorphic to $((H_x \times H_y) \ltimes Z ) \rtimes \bA'^* $, and its center is $Z$. The $H_x$ (resp., $H_y$) has weight $-1$ (resp., $1$) under the $\bA'^*$-action.

  \end{enumerate}
    
 \end{prop}

 Now we have a decomposition 
 $$\S'(\St_\eta (\bF) \backslash \St_\eta(\bA') )= \bigoplus_{\psi_Z} \S'_{\psi_Z}(\St_\eta (\bF) \backslash \St_\eta(\bA')) $$
 into the $\psi_Z$-isotypic components $\S'_{\psi_Z}(\St_\eta (\bF) \backslash \St_\eta(\bA'))$, for all central characters $\psi_Z$. In Section 6.1, we prove that when $\psi_Z$ is trivial there is a class of cuspidal functions. In Section 6.2, we assume $\psi_Z$ is trivial on $Z_0$ and give a decomposition of the $G(\A')$-representation $\ind_{\St_\eta}^{G(\A')} (\S'_{\psi_Z}(\St_\eta (\bF) \backslash \St_\eta(\bA')),\psi_\eta)$ into Hecke-finite and Hecke-infinite pieces. In Section 6.3, we deal with the case $\psi_Z$ is nontrivial on $Z_0$ and prove that $\ind_{\St_\eta}^{G(\A')} (\S'_{\psi_Z}(\St_\eta (\bF) \backslash \St_\eta(\bA')),\psi_\eta)$ contains no Hecke-finite elements.

\subsection{Trivial central character}
We denote $H':=H_x\times H_y$, and $\widetilde{H}(\bA')=\St_\eta(\bA')$. When the central character is trivial, it reduces to the study of automorphic representation of $\widetilde{H}':=\widetilde{H}/Z=(H_x\times H_y)\rtimes \mathbb{G}_m$. 

Let $S_{\widetilde{H}'}$ be the space of automorphic functions of $\tilde{H}'$. By the machinery of Mackey theory (Section \ref{basicmackey}), we have the following decomposition $\widetilde{H}'(\bA')$-representation 
$$ S_{\widetilde{H}'} = \bigoplus_\chi S_{\widetilde{H}',\chi}, $$
where $\chi=(\chi_x,\chi_y)$ runs over $F^*$-orbits of the characters of $H'(\bA')/H'(\bF)$ and $S_{\widetilde{H}',\chi}\simeq \ind_{H'(\bA')}^{\widetilde{H'}(\bA') } \S'_\chi$, where $\S'_\chi$ is the $H'(\bA')$-representation consisting of locally constant and compactly supported functions $f$ on $\St_\chi(\bF)\backslash \St_\chi$ and $\St_\chi \subset H'(\bA')$ is the stabilizer of $\chi$.

Let $S_{\widetilde{H},\chi}$ be the space of functions $f$ on $(H'(\bF)\rtimes Z(\bF) )\backslash \widetilde{H}(\bA')$ such that $f(h,z)=\phi(h)$ where $h\in \tilde{H}'(\bA')$, $z\in Z(\bA')$ and $\phi \in \ind_{H'(\bA')}^{\widetilde{H'}(\bA') } \chi $, and note that it is isomorphic to the inflation of the $\widetilde{H}'(\bA')$-representation $S_{\widetilde{H}', \chi}$ to $\widetilde{H}(\bA')$-representation.

We have an embedding $\kappa_\chi: S_{\widetilde{H},\chi} \hookrightarrow \S'_{\psi_Z=1} (\St_\eta(\bF) \backslash \St_\eta(\bA'))$ of $\St_\eta(\bA')$-representations. When $\chi$ is nontrivial, it given by 
$$ \kappa_\chi f (g) = \sum_{\gamma \in F^*} f(\gamma g).$$

Let $ \S'_{\widetilde{H},\chi} =\ind_{\widetilde{H}(\bA') \times \Na}^{G(\mathbb{A})} (S_{\widetilde{H},\chi}\times \psi_\eta)$. We will consider the embedding of $G(\A')$-representation
$$ \S'_{\widetilde{H},\chi} \hookrightarrow \tS_\eta \xrightarrow{\kappa_\eta}  \S'(G(F)\backslash G(\A')). $$

\begin{prop}\label{subreg cusp}
We have the following cuspidal subrepresentations and Eisenstein series (cf. Definition \ref{eis defn}).
\begin{enumerate}
    \item When $n$ is odd, and both $\chi_x$ and $\chi_y$ are nontrivial, $\kappa_\eta( \S'_{\widetilde{H},\chi})$ is cuspidal.
    \item When $n$ is even, both $\chi_x$ and $\chi_y$ are nontrivial, and the ratio $a_x/a_y \notin (\bF^*)^2$ (where $a_x,a_y \in \omega(\bF)$ correspond the characters $\chi_x$ and $\chi_y$), then $\kappa_\eta( \S'_{\widetilde{H},\chi})$ is cuspidal.
    \item For $n = 3$, if at least one of the characters $\chi_x$ and $\chi_y$ is trivial, then $\kappa_{\eta} (
    \S'_{\widetilde{H},\chi}) \subset  Eis$.
\end{enumerate}
    
\end{prop}

\begin{proof}
     To prove the cuspidality in both odd and even $n$ cases, we let $f\in \S'_{\widetilde{H}',\chi}$, and consider $\kappa_\eta f(g)=\sum_{\gamma \in \St_\eta(F)\backslash G(F)} f(\gamma g)$. Let $P$ be a standard maximal parabolic with unipotent radical $U_P$, we will show that $\int_{U_P(F)\backslash U_P(\A')} f(\gamma u g) du$ vanishes for all $\gamma \in G(F)$ if the conditions in (1) or (2) are satisfied.

If $\gamma^{-1}\eta\gamma \not\in \mathfrak{p} $, then $\Ad(\gamma^{-1})\cdot \eta$ is nontrivial when restricted to $\Na \cap U_P(\A')$ and hence we obtain

$$\int_{\Na \cap U_P(\A')/\Nf \cap U_P(F)} f(\gamma ug)du= 0, $$ 
which implies the desired vanishing.

Next, we analyze the situation for $\gamma^{-1}\eta\gamma \in \mathfrak{p}$, for the standard parabolic subgroup $P$ of type $(m,n-m)$. First, we set $W_m:= \gamma\langle e_1,...,e_m \rangle$. We claim that $\gamma^{-1}\eta\gamma \in \mathfrak{p}$ implies that
$$ W_m = \langle e_2,..,e_m,ae_1+be_{m+1} \rangle,$$
for some $a,b \in F$. Indeed, if $\sum_{i=1}^n a_ie_i \in W_m$, that would imply 
$$\eta(\sum_{i=1}^n a_ie_i)= \alpha (a_3e_2+...+a_n e_{n-1}) , ..., \eta^{n-2}(\sum_{i=1}^n a_ie_i)=\alpha^{n-2}a_ne_2 \in W_m.$$
This implies $a_i=0$ for $i\geq m+2$ and $e_2,...,e_m \in W_m$, and and hence the claim.

Now we claim that the intersection
$$ L:=\widetilde{H} \cap \gamma U_P \gamma^{-1}$$
contains exactly one of the subgroups $H_x$ and $H_y$ if $m \neq m-n$. This claim (without loss of generality, say $L$ contains $H_x$) implies that  
$$\int_{\gamma^{-1} H_x(\A') \gamma/ \gamma^{-1}H_x(F)\gamma} f(\gamma ug)du= \int_{\gamma^{-1} H_x(\A') \gamma/ \gamma^{-1}H_x(F)\gamma} \chi_x(\gamma u \gamma^{-1})f( \gamma g)du =0, $$
because $\chi_x$ is nontrivial, and that implies $\int_{ U_P(\A')/ U_P(F)} f(\gamma ug)du= 0 $.

Now we will compute $L$ and will prove the claim. First, we note that if $b=0$, then $\gamma \in P$ which would imply $\gamma \in P$ and $L=\widetilde{H} \cap U_P$. In this case, we can directly see $L$ contains exactly one of the subgroups $H_x$ and $H_y$. Now we assume $b\neq0$. It suffices to (and indeed easier to) compute it at the level of Lie algebra, namely compute $\mathfrak{l}=\tilde{\mathfrak{h}} \cap \gamma \mathfrak{u}_P \gamma^{-1}$, where $\tilde{\mathfrak{h}}= \lbrace xE_{21}+yE_{1,n}+\sum_{i=0}^{n-3}z_i(E_{2,n-i}+...+E_{2+i,n}) | x,y,z_i \in \bA' \rbrace $.  Note that $u \in \gamma \mathfrak{u}_P \gamma^{-1}$ if and only if both of the following conditions hold:
\begin{enumerate}
    \item The kernel of $u$ contains $W_m$,
    \item The image of $u$ is contained in $W_m$.   
\end{enumerate}
It is then easy to observe that when $\frac{n}{2}>m \geq1$, we have
$$\mathfrak{l}= 
\lbrace -atE_{1,n}+bt\sum_{i=0}^{m-2} E_{m-i,n-i} + \sum_{i=0}^{m-2}z_i(E_{2,n-i}+...+E_{2+i,n})| \text{ $t, z_i \in \bA'$} \rbrace,$$
i.e. $\mathfrak{l}$ consists of matrices of the form
\[
\begin{pmatrix}
0 & 0 & \cdots & 0 & 0 & \cdots & 0 & -at\\
0 & 0 & \cdots & 0 & bt & z_{m-2} & \cdots & z_0\\
0 & 0 & \cdots & 0 & 0 & bt & \cdots & z_1\\
\vdots & \vdots &  & \vdots & \vdots & \ddots & \ddots & \vdots\\
0 & 0 & \cdots & 0 & 0 & 0 & bt & z_{m-2}\\
0 & 0 & \cdots & 0 & 0 & 0 & 0 & bt\\
0 & 0 & \cdots & 0 & 0 & 0 & 0 & 0\\
\vdots & \vdots &  & \vdots & \vdots & \vdots & \vdots & \vdots\\
0 & 0 & \cdots & 0 & 0 & 0 & 0 & 0
\end{pmatrix}
\]
and when $m > \frac{n}{2}$, 
$$\mathfrak{l}= 
\lbrace -btE_{2,1}+at\sum_{i=0}^{n-m}(E_{2+i,m+i})+ \sum_{i=0}^{n-m-2}z_i(E_{2,n-i}+...+E_{2+i,n})| \text{ $t, z_i \in \bA'$} \rbrace.$$
i.e.  $\mathfrak{l}$ consists of matrices of the form
\[
\begin{pmatrix}
0 & 0 & \cdots & 0 & 0 & 0 & \cdots & 0 & 0\\
-bt & 0 & \cdots & 0 & at & z_{n-m-1} & z_{n-m-2} & \cdots & z_0\\
0 & 0 & \cdots & 0 & 0 & at & z_{n-m-1} & \cdots & z_1\\
0 & 0 & \cdots & 0 & 0 & 0 & at & \ddots & z_2\\
\vdots & \vdots &  & \vdots & \vdots & \vdots & \ddots & \ddots & \vdots\\
0 & 0 & \cdots & 0 & 0 & 0 & \cdots & at & z_{n-m-1}\\
0 & 0 & \cdots & 0 & 0 & 0 & \cdots & 0 & at\\
0 & 0 & \cdots & 0 & 0 & 0 & \cdots & 0 & 0\\
\vdots & \vdots &  & \vdots & \vdots & \vdots &  & \vdots & \vdots\\
0 & 0 & \cdots & 0 & 0 & 0 & \cdots & 0 & 0
\end{pmatrix}
\]

This finishes the proof of the claim for $n \neq 2m$, and hence the cuspidality follows assuming the condition in (1) when $n$ is odd.

To prove (2) it remains to compute the intersection for the case $n=2m$, i.e., for the standard parabolic $P$ of type $(m,m)$. In this case, we have
$$\mathfrak{l}= 
\lbrace -btE_{2,1}+\frac{a^2}{b}tE_{1,n}+at(E_{2,m+1}+...+E_{m+1,n})+ \sum_{i=0}^{m-2}z_i(E_{2,n-i}+...+E_{2+i,n})| \text{ $t, z_i \in \bA'$} \rbrace.$$
i.e.  $\mathfrak{l}$ consists of matrices of the form
\[
\begin{pmatrix}
0 & 0 & \cdots & 0 & 0 & 0 & \cdots & 0 & \dfrac{a^{2}}{b}t\\
-bt & 0 & \cdots & 0 & at & z_{m-2} & z_{m-3} & \cdots & z_0\\
0 & 0 & \cdots & 0 & 0 & at & z_{m-2} & \cdots & z_1\\
0 & 0 & \cdots & 0 & 0 & 0 & \ddots & \ddots & \vdots\\
\vdots & \vdots &  & \vdots & \vdots & \vdots & \ddots & z_{m-2} & z_{m-3}\\
0 & 0 & \cdots & 0 & 0 & 0 & \cdots & at & z_{m-2}\\
0 & 0 & \cdots & 0 & 0 & 0 & \cdots & 0 & at\\
0 & 0 & \cdots & 0 & 0 & 0 & \cdots & 0 & 0\\
\vdots & \vdots &  & \vdots & \vdots & \vdots &  & \vdots & \vdots\\
0 & 0 & \cdots & 0 & 0 & 0 & \cdots & 0 & 0
\end{pmatrix}
\]
When $a=0$, it is still true that $\mathfrak{l}$ contains $H_x$. When $a \neq 0$, we consider the subgroup $L':=\lbrace I -btE_{2,1}+\frac{a^2}{b}tE_{1,n} +z_0E_{2,n} \rbrace \subset L$ and we get
\begin{align*}  \int_{\gamma^{-1} L'(\A') \gamma/ \gamma^{-1}L'(F)\gamma} f(\gamma ug)du &= \int_{\gamma^{-1} L'(\A') \gamma/ \gamma^{-1}L'(F)\gamma} \chi(\gamma u \gamma^{-1})f( \gamma g)du \\
&= \int_{\A'/F} \chi_x(-b^2 t) \chi_y(a^2 t) f(\gamma g)dt \\
&= 0
\end{align*}
Indeed, the last integral vanishes because $\chi_x|_{L'(\bA')} \neq \chi_y|_{L'(\bA')}$ by our assumption that $a_x/a_y \notin (\bF^*)^2$  This finishes the proof of (2).

To prove (3), it is more convenient to work with another representative $\eta$, namely, $\alpha E_{13} $ so that $\St_\eta(\bA')\subset B(\bA')$. More precisely, $\St_\eta(\bA')$ consists of the matrices of the form $\begin{pmatrix}

1 & x & z \\
 & a &  y\\
 &  & 1 
\end{pmatrix}$, where $x,y,z \in \bA'$ and $a\in \bA'^*$.

Now we assume one of the $\chi_i$'s, say $\chi_x$, is trivial. Let $P$ be the standard parabolic subgroup of type $(1,2)$ with unipotent radical $U$. Then 
$\S'_{\widetilde{H},\chi}$ is left $U(\bA')$-invariant. Let $\mathfrak{u}$ be the Lie algebra of $U$. Since $\mathfrak{u} \subset \langle \eta \rangle^\perp$, $\S'_{\widetilde{H},\chi}$ is also left $\mathfrak{u}(\N\bA')$-invariant, and hence it is left $U(\A')$-invariant. Let $K$ be a compact open subgroup of $G(\A')$. Consider the following diagram
    \[
\begin{tikzcd}
 & \St_\eta(F) \backslash G(\A')/K \arrow[d , "\pi"] \arrow[dr , "obv"] & \\
QBun_{M,K} & QBun_{P,K} \arrow[l , "p"] \arrow[r , "q" ] & Bun_{G,K} , 
\end{tikzcd}
\]
then we have $\pi_* (\S'_{\widetilde{H},\chi})^K \subset p^*(\S'(QBun_M))$ because all functions in $\S'_{\widetilde{H},\chi}$ are left $U(\A')$-invariant. This finishes the proof that $\S'_{\widetilde{H},\chi} \subset Eis$ are Eisenstein series. The case of $\chi_y$ is similar.

\end{proof}

\subsection{Central characters $\psi_Z$ that are trivial on $Z_0$}
In this subsection, we assume that $\psi_Z$ is trivial on $Z_0$. We classify all spherical Hecke-finite functions in $\kappa_\eta(\ind_{\St_\eta}^{G(\A')} (\S'_{\psi_Z}(\St_\eta (\bF) \backslash \St_\eta(\bA')),\psi_\eta))$. In this subsection, we will choose a different representative of $\eta$, namely, 
$$ \eta = \alpha (E_{13}+E_{34}+...+E_{n-1,n}).$$
With this representative, we have the centralizer $\St_\eta(\bA')=:\tilde{H}$ consisting of the matrices of the form 
 \[
\begin{pmatrix}
1 & x & z_{n-3} & \cdots & z_1 & z_0 \\
0 & a & 0& 0 & \cdots & y\\
0 & 0 & 1 & 0 & \cdots & z_1 \\
0 & 0 & 0 & 1 & \ddots & \vdots \\
\vdots & \vdots & \vdots & \ddots & \ddots & z_{n-3} \\
0 & 0 & 0 & \cdots & 0 & 1
\end{pmatrix},
\]
where $a\in \bA'^*$ and $x,y,z_i \in \bA'.$ We have subgroups $H,Z,Z_0, A$ in $\tilde{H}$ as in the notation of the previous subsection (recall Proposition 6.1). This representative has the advantage that its stabilizer $\St_\eta(\bA')$ is contained in the Borel subgroup $B$ and will be useful for computing the support of Hecke-finite functions in Section 7.

\begin{prop}\label{subgreg_fd} Let $G=\PGL_n$ for $n \geq 3$ and $a_x$ (resp. $a_y$) be the element in $\omega(\bF)$ that corresponds to $\chi_x$ (resp. $\chi_y)$. Let $K=G(\O')$.
\begin{enumerate}
    \item If both $\chi_x$ and $\chi_y$ are nontrivial, then all functions in $\S'_{\tilde{H},\chi}^K$ are Hecke-finite. 
    
    \item The subspace of the spherical vectors $\bigoplus_{\chi } \S'_{\widetilde{H},\chi}^K$ is finite-dimensional, where $\chi$ runs through the $A$-conjugacy classes of characters $\chi=(\chi_x,\chi_y)$ and both $\chi_x$ and $\chi_y$ are nontrivial.

\end{enumerate}

\end{prop}
\begin{proof}

(1) Let $f \in \S'_{\widetilde{H},\chi}^K$. By Proposition \ref{nilpnonadm}, we must show that there are only finitely many $g \in \St_\eta(\bA') \backslash G(\bA')/ \overline{K}$ such that $\Ad(g^{-1}) \cdot \eta$ are in different $K$-orbits in $\tilde{\Xi}_K$ and  $\S'_{\widetilde{H},\chi}^{gKg^{-1}\cap \widetilde{H}(\bA')}\neq 0,$ 
and each of them is finite-dimensional. We prove this by a similar way as in the proof of Proposition \ref{regnilfd}.

    By the Iwasawa decomposition, it suffices to consider $g\in  \widetilde{H}(\bA') \backslash B(\bA') / (\overline{K}\cap B(\bA'))$, and similar to the proof of Proposition \ref{regnilfd}, it is enough to consider $g\in A\backslash T(\bA') / T(\bO)$ such that  $\Ad(g^{-1}) \cdot \eta$ are in different $\overline{K}$-orbits in $\tilde{\Xi}_{\overline{K}}$, and we may take $g$ of the form $g=\diag(1,1,t_1^{-1},...,t_{n-2}^{-1})$ where $t_i \in \bA'^*$ and 
   \begin{align}
        t_it_{i+1}^{-1} \alpha \in \N \inv \omega(\bO)
   \end{align} 
    for $n-3\geq i\geq 0$ (we set $t_0=1$).  We compute that
    
    $$ gKg^{-1} \cap \widetilde{H}(\bA') \subset (H_x(\bO)\times H_y(t_{n-2}\bO) \ltimes Z(\bA') )\rtimes \bO^*. $$
     Let $\Lambda_i$ be the conductor of $\chi_i$ where $i=x,y$. Since $\S'_{\widetilde{H},\chi}$ is isomorphic to inflating the $\widetilde{H}'(\bA')$-representation $\ind_{\widetilde{H}'(\bA')}^{\widetilde{H}(\bA')} \chi$, applying Proposition \ref{sphvecgen}, we have $\S'_{\widetilde{H},\chi}^{gKg^{-1}\cap \widetilde{H}(\bA')}\neq 0$  if and only if there exists 
     $a\in \bA'^*$ such that 
     \begin{align}   a^{-1}\in \Lambda_x \text{   and  }at_{n-2}\in \Lambda_y,\end{align}
     or equivalently $ a^{-1}a_x, at_{n-2} a_y\in \omega(\bO)$. But this, together with (6.4), implies that 
     \begin{align}
\begin{cases}
    t_{n-2}a_xa_y  \in \omega^2(\bO) \\
   t_{n-3}a_xa_y \alpha \in \N \inv \omega^3(\bO)\\
   ... \\
   t_1a_xa_y \alpha^{n-3}\in \N^{-n+3}\omega^{n-1}(\bO)
   \end{cases}
\end{align}

On the other hand, by successively multiplying the elements in (6.4) (starting with $i=0$), we have

   \begin{align}
\begin{cases}
   t_1^{-1}\alpha \in \N \inv \omega(\bO) \\
 t_2^{-1} \alpha^2 \in \N ^{-2} \omega^2(\bO)\\
   ... \\
    t_{n-2}^{-1} \alpha^{n-2}\in \N^{-n+2}\omega^{n-2}(\bO) 
\end{cases}
\end{align}

These imply that each $t_i$'s and $a$ have bounded poles and zeros, and hence there are only finitely many choices for $g \in A \backslash T(\bA') /  T(\bO)$. This finishes the proof of (1)

To prove (2), it remains to observe that there are finitely conjugacy classes of $\chi$ such that there exists $t$ and $a$ satisfying (6.4) and (6.5), and hence  only finitely many summands are nonzero. Indeed, (6.5) and (6.6) imply $a_xa_y\alpha^{n-2} \in H^0(\overline{C},\N ^{-n+2} \omega^n$). So there are finitely many possibilities for the product $a_xa_y$, but the conjugacy classes of $\chi$ is determined by the product $a_xa_y$ since $c \cdot (a_x,a_y) = (ca_x,c^{-1}a_y)$.

\end{proof}
\begin{rem}
    The above calculations show that if a $\PGL_n$- Higgs bundle $P$ is in the support of a function in $\S'_{\widetilde{H},\chi}^{G(\O')}$ (viewed as a function on the groupoid of Higgs bundles with generically subregular nilpotent Higgs fields), then the underlying vector bundle of $P$ on $\overline{C}$ can be obtained as iterative extensions of line bundles $L_i$ and $|\deg L_i - \deg L_j|$ is bounded by a constant depending on $n$ and the genus of $\overline{C}$. We will come back to this in Section 7.
\end{rem}
\begin{rem}
     When $n$ is even and $a_x/a_y \in (\bF^*)^2$, we know the functions are Hecke-finite but we do not know how to prove its cuspidality (it is cuspidal if Conjecture \ref{conj1}(1) is true). 
\end{rem}

\begin{prop}
   If at least one of the $\chi_i$'s is trivial then $\S'_{\widetilde{H},\chi}$ has no admissible subrepresentation and hence $\S'_{\widetilde{H},\chi}^K$ contains no nonzero Hecke-finite functions for every compact open subgroup $K$ of $G(\A')$.
\end{prop}
\begin{proof}
By Proposition \ref{heckeadmiss}, it suffices to show that $\S'_{\widetilde{H},\chi}^K$ has no nonzero $\Hk$-finite vector. To show that $\S'_{\widetilde{H},\chi}^K$ has no nonzero $\Hk$-finite vector, it suffices to show that for every smooth subrepresentation $V$ of $S_{\widetilde{H},\chi}$ and every compact open subgroup $K'$ of $\widetilde{H}(\bA')$, $V^{K'}$ is either zero or infinite-dimensional. 

Since the characteristic of $\bF$ is greater than $n$, we have a splitting $Z \simeq Z_0 \times Z'$, the problem is reduced to show that $V^K$ is zero or infinite-dimensional for any smooth $H(\bA') \rtimes \bA'^*$-subrepresentation $V$ of $S_{H \rtimes \mathbb{G}_m,\chi}$ and $K$ is any compact open subgroup of $H(\bA') \rtimes \bA'^*$, where $S_{H \rtimes \mathbb{G}_m,\chi}$ is the subrepresentation of $\S'((H(\bF) \rtimes \bF^* )\backslash (H(\bA') \rtimes \bA'^*))$ that $Z_0(\bA')$ acts trivially and $H_x(\bA') \times H_y(\bA')$ acts on the left via $\chi$. Since  $Z_0(\bA')$ acts trivially, the problem is further reduced to show the same statement for any smooth $H'(\bA') \rtimes \bA'^*$-subrepresentation $V$ of $S_{\widetilde{H}',\chi}$. Now, we let $K$ be a compact open subgroup of $H'(\bA') \rtimes \bA'^*$, and prove that $V^K$ is either zero or infinite-dimensional in the following two cases.
 
First, if $\chi$ is trivial, then the stabilizer of $\chi$ is $\widetilde{H}'(\bA')$ and all functions in $S_{\widetilde{H}',\chi}$ descend to functions on $\bF^* \backslash \bA'^*$, and the desired result follows by the automorphic representation theory of $\mathbb{G}_m$.

Second, if exactly one of $\chi_i$'s is trivial, say $\chi_x$ is trivial, then the stabilizer of $\chi$ is $H'(\bA')=H_x(\bA') \times H_y(\bA')$. Since $V$ is either zero or isomorphic to the induction of $\chi$, similar to Proposition \ref{nilpnonadm}, it suffices to show that there exists an infinite family of $\overline{a}=\diag(a^{-1},1,...,1) \in A$ such that $\overline{a}^{-1} \cdot \chi$ are trivial when restricted to $K':=K \cap H'(\bA')$, are in different $K$-orbits, and $\chi^{K'} \subset \chi^{\overline{a}K\overline{a}^{-1} \cap H'(\bA')}$. Now, we note that any $a \in \bA'^*$ with sufficiently many zeros would work because we have
$$\overline{a}K\overline{a}^{-1} \cap H'(\bA') \subset K' \text{  modulo  }H_x(\bA')$$
for any $a$ with no poles, and since $H_x(\bA')$ acts trivially, this implies $\chi^{K'} \subset \chi^{\overline{a}K\overline{a}^{-1} \cap H'(\bA')}$. This finishes the proof.

\end{proof}

\subsection{Central characters $\psi_Z$ that are nontrivial on $Z_0$}

We keep the notation from the previous subsection. Recall that, in particular, we have $$\St_\eta(\bA')= \widetilde{H}(\bA')\simeq ((H_x(\bA') \times H_y(\bA') )\ltimes Z(\bA')) \rtimes \bA'^*.$$

Since our ground field $\bF$ has characteristic greater than $n$, there is a splitting $Z=Z_0 \times Z'$, where $Z' \simeq \bA'^{n-3}$. Hence we have an isomorphism
$$ \widetilde{H}(\bA') \simeq (H(\bA') \rtimes \bA'^*) \times Z'(\bA'), $$
 where $H$ is the 3-dimensional Heisenberg group which is acted on by $\bA'^*$ with weights 1, -1, and 0. We denote $L_1(\bA')$ and $L_2(\bA')$ be the weight-1 and weight-(-1) lines in $H(\bA')$ (which we previously denoted by $H_y$ and $H_x$ respectively). Fixing a character $Z'(\bA')/Z'(\bF)$, the study of automorphic representation for $ \widetilde{H}(\bA')$ is reduced to that of $ H(\bA') \rtimes \bA'^* $. From now on (for the rest of this section), we set $\widetilde{H}= H \rtimes \mathbb{G}_m$.
 
 For the rest of this section, we further fix a nontrivial character $\psi_0$ of $Z_0(\bA')/Z(\bF)$, and show that
 $$\S'_{\psi_0}:= \S'_{\psi_0}(\widetilde{H}(\bF)\backslash \widetilde{H}(\bA'))$$
 has no admissible subrepresentation, and hence $\ind_{\widetilde{H}(
 \bA') \times \Na}^{G(\A')} (\tS_{\psi_0},\psi_\eta) $ has no admissible subrepresentation.

A direct computation shows that we have the following description $\widetilde{H}(\bA')$-representation $ \S'_{\psi_0} $, namely, it is the space of locally constant and compactly supported (modulo the $\overline{F}$-points from the left) functions $f$ on $(L_1(\bA')\times L_2(\bA'))\rtimes \bA'^*$ satisfying that 
\begin{enumerate}
    \item  $f(l_1+l_1',l_2+\lambda l_2',\lambda)=\psi_0(l_1l_2')f(l_1,l_2,\lambda)$
    \item  $f(\lambda'l_1, l_2,\lambda'\lambda)=f(l_1,l_2,\lambda)$,
\end{enumerate}
for every $(l_1',l_2',\lambda') \in L_1(\bF) \times L_2(\bF) \times \bF^*$.
 The $\widetilde{H}(\bA')$-action on $\S'_{\psi_0}$ can be described as follows, nameley, 
 $$ (((l_1,'l_2',\lambda') \cdot f)(l_1,l_2, \lambda) = \psi_0 (-l_1'l_2)f(l_1+\lambda l_1', \lambda' l_2+l_2', \lambda' \lambda),$$
 where $(l_1,l_2,\lambda), (l_1',l_2',\lambda') \in (L_1(\bA') \times L_2(\bA')) \rtimes \bA'^* $, and $Z_0$ acts via the character $\psi_0$. 


Next, we will give an alternative description for $\S'_{\psi_0}$. We rewrite $f\in \S'_{\psi_0}$ by considering its Fourier expansion along $L_1(\bA'/\bF)$, namely,
$$ f(l_1,l_2,\lambda)= \sum_{a \in \bF} \psi(al_1 )\phi (l_2- \lambda a,\lambda),$$
where $\phi$ are functions on $L_2(\bA')\times (\bA'^*/\bF^*)$ and the support of $\phi$ on the $(\bA'^*/\bF^*)$-projection is compact.  The action of $(L_1(\bA') \times L_2(\bA')) \rtimes \bA'^*$ on $S':=\lbrace \phi: \text{compactly supported functions on } L_2(\bA')\times (\bA'^*/\bF^*)\rbrace$ is given by 
$$ (l_1',l_2',\lambda') \cdot \phi(l_2, \lambda)= \psi_0(-l_1'l_2)\phi(\lambda'l_2+l_2',\lambda' \lambda),$$
and $Z_0$ acts via the character $\psi_0$. From now on, we will work with this realization of the  $\widetilde{H}(\bA')$-representation $\S'_{\psi_0}$.

\subsubsection{Subreps and non-admissibility}
    Let $V$ be a subrepresentation of $\S'_{\psi_0}$, we let $J^*$ to be the Jacquet functor (taking coinvariant with respect to the subgroup $L_1(\bA')$). The functor $J^*$ restricts to a functor from the category of $\St_\eta(\bA')$-subrepresentations of $\S'_{\psi_0}$ to the category of $\bA'^*$-subrepresentations of $\S'(\bA'^*/\bF^*)$. Let $V$ be a smooth $\St_\eta(\bA')$-subrepresentations of $\S'_{\psi_0}$, we have
    $$J^*(V) \simeq \lbrace \phi|_{0\times \bA'^*}|\phi\in V\rbrace.$$ 
    Indeed, it is straightforward to verify that the map $V \to \lbrace \phi|_{0\times \bA'^*}|\phi\in V\rbrace$ descends to a bijective map $J^*(V) \to \lbrace \phi|_{0\times \bA'^*}|\phi\in V\rbrace$.
    

    Also, $J^*$ admits a right adjoint $J_*$. When $W$ is a $\bA'^*$-subrepresentation of $\S'(\bA'^*/\bF^*)$, we have
    $$J_*: W \mapsto \lbrace \phi \in \S'_{\psi_0} | (\lambda \mapsto \phi(l_2,\lambda)) \in W \text{ for each $l_2$ in $L_2(\bA')$} \rbrace.$$
\begin{prop}
$J^*$ defines an equivalence between smooth subrepresentations of $\S'_{\psi_0}$ and smooth subrepresentations of $\S'(\bA'^*/\bF^*)$.
\end{prop}

\begin{proof}
It is direct to check from the definition that $J^*J_*W \simeq W$ for every smooth $\bA'^*$-subrepresentation $W$ of $S(\bA'^*/\bF^*)$. It remains to show that the natural inclusion
$$
V\hookrightarrow J_*J^*(V)
$$
is surjective for every smooth $\St_\eta(\bA')$-subrepresentation $V$ of $\S'_{\psi_0}$.

Let $\phi\in J_*J^*(V)$. By the definition of $J_*$, for every $x\in L_2(\bA')$, the function $\lambda\mapsto \phi(x,\lambda)$ belongs to $J^*(V)$. Hence there exists $v_x\in V$ whose restriction to $\{x\}\times \bA'^*$ agrees with $\phi$. By smoothness, there is a compact open subgroup $\Lambda_x\subset L_2(\bA')$ such that both $v_x$ and $\phi$ are $\Lambda_x$-invariant. Therefore $v_x(l_2,\lambda)=\phi(l_2,\lambda)$ for every $l_2\in x+\Lambda_x$.

We next observe that the projection of the support of $\phi$ to $L_2(\bA')$ is compact. Indeed, by smoothness there exists a compact open subgroup $\Lambda_1\subset L_1(\bA')$ fixing $\phi$. Since
$(a\cdot\phi)(l_2,\lambda)=\psi_0(-al_2)\phi(l_2,\lambda)$ for $a\in L_1(\bA')$, we have $\phi(l_2,\lambda)\neq 0$ only if $l_2\in {\Lambda_1'}^\perp:=\{a\in L_2(\bA')\mid \psi_0(ab)=1\text{ for every }b\in\Lambda_1\}$. As ${\Lambda_1'}^\perp$ is compact, the assertion follows.

Thus, after refining an open cover given by the open sets $x+\Lambda_x$ for $x$ in the $L_2$-projection of support of $\phi$, we may write the projection of the support of $\phi$ as a finite disjoint union of compact open cosets $C_i=x_i+\Lambda_i$, $1\leq i\leq r$, such that for each $i$ there exists $v_i\in V$ satisfying $v_i=\phi$ on $C_i\times \bA'^*$.

For each $i$, let
$$
\Lambda_i^\perp=\{a\in L_1(\bA')\mid \psi_0(ab)=1\text{ for every }b\in\Lambda_i\},
$$
and for every $v\in V$ we define
$$
P_i(v)=\frac{1}{vol(\Lambda_i^\perp)}
\int_{\Lambda_i^\perp}\psi_0(ax_i)(a\cdot v)\,da.
$$
Since $V$ is an $L_1(\bA')$-subrepresentation, we have
$P_i(v_i)\in V$. Here the integral can be rewritten as a finite sum:
by the smoothness of $v_i$ and $\psi_0$, the integrand is constant
on the cosets of some open subgroup of $\Lambda_i^\perp$, so the
normalized integral can be written as a finite linear combination
of translates of $v_i$. On the other hand, using the formula for the $L_1(\bA')$-action, we have
$$
P_i(v_i)(l_2,\lambda)
=
\left(
\frac{1}{vol(\Lambda_i^\perp)}
\int_{\Lambda_i^\perp}\psi_0(a(x_i-l_2))\,da
\right)v_i(l_2,\lambda).
$$
By orthogonality of characters, the factor in parentheses is the characteristic function of $x_i+\Lambda_i$. Hence
$P_i(v_i)=\delta_{C_i}\phi$. Since the $C_i$ form a disjoint cover of the $L_2(\bA')$-support of $\phi$, we obtain
$$
\phi=\sum_{i=1}^r P_i(v_i)\in V.
$$
Therefore $V=J_*J^*(V)$, and this finishes the proof.
\end{proof}

\begin{cor} We have:
\begin{enumerate}
    \item  Let $K$ be a compact open subgroup of $\widetilde{H}(\bA')$, and let $V$ be a smooth subrepresentation of $\S'_{\psi_0}^K$. The space $V^K$ is either zero or infinite-dimensional.
    \item For any compact open subgroup $K$ of $G(\A')$, the nonzero elements in $(\ind_{\tilde{H}(\bA') \times \Na}^{G(\A')}(\tS_{\psi_0},\psi_\eta))^{K}$ are Hecke-infinite.
\end{enumerate}
   
\end{cor}
\begin{proof}
 
  Let $W=J^*(V)$ and let $\phi \in V^K$ be a nonzero element. We claim that for any $l_2' \in \bA'^*$, the map $\phi_{\lambda'}:(l_2, \lambda) \mapsto \phi(l_2,\lambda \lambda')$ is nonzero and belongs to $V^K$. The claim immediately implies the infinite-dimensionality of $V^K$, because the union of the image of the support of $\phi_{\lambda'}$'s in $\bA'^*/\bF^*$ is noncompact. Now we prove the claim. Since $\phi \in V = J_*(W)$, for any $l_2 \in L_2(\bA')$, the map $\lambda \mapsto \phi(l_2,\lambda) $ belongs to $W$. This implies the map $\lambda \mapsto \phi(l_2,\lambda \lambda') $ also belongs to $W$, and hence $\phi_{\lambda'} \in V$. It remains to check $\phi_{\lambda'}$ is $K$-invariant, but this is immediate from the definition (note that the operator $\phi\mapsto \phi_{\lambda'}$ commutes with the $\widetilde{H}(\bA')$-action). This finishes the proof of (1). (2) follows by Propositions \ref{heckefinite}, \ref{heckeadmiss} and \ref{sphvec}.



\end{proof}

\section{Proof of Theorem \ref{main theorem}}
In this section, we gather the results obtained from the previous sections to prove Theorem \ref{main theorem}. Throughout we let $G=\PGL_3$ unless specified otherwise, and $K=G(\O')$. In Section 7.1, we explain how our results from previous sections imply the finite-dimensionality of $\S'_{HF}(\bun_G(C))$ and the containment $\S'_{HF}(\bun_G(C)) \subset \S'_{cusp}(\bun_G(C))$, and hence complete the proof of Theorem \ref{main theorem} (1) and (2). In section 7.2, we compute the support of Hecke-finite functions obtained in various orbits. As a result, we complete the proof of Theorem \ref{main theorem} (3).

\subsection{Finite-dimensionality } We need a simple lemma:

\begin{lem}
    Let $f,g \in \S'(\bun_G(C))$ and $\Hk \cdot f \cap \Hk\cdot g =0$, then $f+g$ is Hecke-finite if and only if both $f$ and $g$ are Hecke-finite.
\end{lem}
\begin{proof} Since $\Hk \cdot f \cap \Hk\cdot g =0$, we have an injection $\Hk \cdot (f+g) \hookrightarrow \Hk \cdot f \oplus \Hk \cdot g : h(f+g) \mapsto (hf,hg)$ of $\Hk$-modules. If $f$ and $g$ are Hecke-finite, it is clear that $f+g$ is also Hecke-finite. Now assume $f+g$ is Hecke-finite. Consider the surjective map $p_f: \Hk \cdot (f+g) \to \Hk \cdot f$ of $\Hk$-modules, then $\Hk \cdot f$ is finite-dimensional, and similarly we also have $\Hk \cdot g$ is finite-dimensional.
\end{proof}

\begin{proof}[Proof of Theorem \ref{main theorem} (1) and (2)]
By the Lemma, to prove that the space $\S'_{HF}(\bun_G(C))$ is finite-dimensional, it suffices to have a decomposition of $H_{G(\O')}$-modules 
$$\S'(\bun_G(C))=\S'_{HF}(\bun_G(C)) \bigoplus \S'_{HI}(\bun_G(C)),$$
where $\S'_{HI}(\bun_G(C))$ is a direct sum of Hecke-infinite components, and that $\S'_{HF}(\bun_G(C))$ is finite-dimensional. Indeed, we obtain this by first taking the $G(\O')$-invariants of the orbit decomposition
$$\S'(\bun_G(C)) =\bigoplus_{\Omega \text{}}\S'(\bun_G(C))_{\Omega}, $$
and for each orbit $\Omega$, we already have a decomposition of $\S'_\Omega^{G(\O')}$ into the direct sum of the Hecke-finite component (and all Hecke-finite pieces are also cuspidal) and the Hecke-infinite component and the finite-dimensionality of the Hecke-finite component:

\begin{itemize}
    \item For the zero orbit $\Omega$, this holds for $\S'(\bun_G(C))_{\Omega}=\S'(\bun_G(\overline{C}))$ by classical theory of automorphic functions for $\overline{C}$.

    \item For regular elliptic orbits $\Omega$, $\bigoplus_\Omega \S'(\bun_G(C)))_{\Omega}$ is finite-dimensional by Corollary \ref{regellipfd}. 

    \item For semisimple orbits that are not regular elliptic or mixed orbits $\Omega$, we have $\S'(\bun_G(C)))_{\Omega}$ contains only Hecke-infinite functions by Proposition \ref{mixed} and Corollary \ref{pgl3mixed}.

    \item For regular nilpotent orbit $\Omega$, we have $\S'(\bun_G(C)))_{\Omega}= \bigoplus_{\chi:U_h(\bA')/U_h(\bF) \to U(1)}\S'(\bun_G(C)))_{\Omega,\chi}$. By Propositions \ref{regnilpresults} and \ref{regnilfd}, we have the Hecke-finite (which is also cuspidal) part equals $\bigoplus_{\chi\neq1}\S'(\bun_G(C)))_{\Omega,\chi}$ and is finite-dimensional by Proposition \ref{regnilfd}.

    \item For subregular nilpotent orbit $\Omega_\eta$, following the notation in Section 6, we have a decomposition of $\Hk$-module $\S'(\bun_G(C)))_{\Omega}= \bigoplus_{\psi_Z}\S'(\bun_G(C)))_{\Omega, \psi_Z}$, where $\psi_Z$ are the central characters of $\St_\eta(\bA')$. When $\psi_Z=0$, we have a further decomposition $\S'(\bun_G(C)))_{\Omega, \psi_Z}= \bigoplus_\chi \S'(\bun_G(C)))_{\Omega, \psi_Z, \chi}$  By Propositions 6.2, 6.3, 6.10 and Corollary 6.12, the Hecke-finite component equals the direct sum of $\S'(\bun_G(C)))_{\Omega, \psi_Z, \chi}$, where $\psi_Z=0$ and both $\chi_x$ and $\chi_y$ are nontrivial, and the Hecke-finite component is finite-dimensional by Proposition 6.3.
\end{itemize}
This concludes the proof for Theorem \ref{main theorem}(1). In fact, Theorem \ref{main theorem}(2) also follows because all Hecke-finite pieces are indeed cuspidal as noted above.

\end{proof}

.

\subsection{Total support of Hecke-finite functions }

Now we compute the support $S_h$ of the spherical Hecke-finite functions directly and prove Theorem \ref{main theorem}(3). Recall that $r: \bun_G(C) \to \bun_G(\overline{C})$ is the reduction map from $G$-bundles on $C$ to $G$-bundles on $\overline{C}$.

\begin{thm}(Theorem \ref{main theorem}(3))
    Let $G=\PGL_3$. Then we have $S_h \subset r^{-1}(\mathfrak{U}_3)$. 
\end{thm}
\begin{proof}
    This follows by Lemma \ref{ext gap implies HN gap} and Proposition \ref{pgl3 iterative extension}.
\end{proof}

\begin{lem}\label{ext gap implies HN gap}
    Let $G=\GL_n$ and $\E$ be a rank $n$ vector bundles on $\overline{C}$ and $\E$ admits a filtration $0=\E_0 \subset \E_1 \subset \E_2 \subset ... \subset \E_m=\E$ such that each $\E_{i}/\E_{i-1}$ is a semistable bundle of slope $d_i$. If $\E \in \bun_G^\lambda(\overline{C})$ for some dominant coweight $\lambda = (\lambda_1,...,\lambda_n)$, then $\lambda_1-\lambda_n \leq \max \lbrace d_i \rbrace - \min \lbrace d_i \rbrace $.
\end{lem}
\begin{proof}
    First, for a vector bundle $V$ we define $\mu_{max}(V)$ (resp., $\mu_{min}(V))$ to be the maximal (resp., minimal) slope in the Harder-Narasimhan slope sequence of $V$. We claim that if $0 \to A' \to A \to A'' \to 0$ is a short exact sequence of vector bundles on $\overline{C}$, then we have \begin{enumerate}
        \item $\mu_{max}(A) \leq \max \lbrace \mu_{max}(A'), \mu_{max}(A'')\rbrace$
        \item $\mu_{min}(A) \geq \min \lbrace \mu_{min}(A'), \mu_{min}(A'')\rbrace$
    \end{enumerate}
    
    \begin{proof}[proof of the claim]
        It is enough to prove that $\mu(B) \leq \max \lbrace \mu_{max}(A'), \mu_{max}(A'')\rbrace $ for all subbundle $B$ of $A$. Let $B'=A'\cap B $ and $B''=\text{Im}(B \to A'')$. Then we have a short exact sequence
        $$ 0 \to B' \to B \to B'' \to 0.$$
        
        First, suppose $ B'$ and $B''$ are nonzero. Note that $\mu_{max}(A') \geq \mu(B')$ and $\mu_{max}(A'') \geq \mu(B'')$, and hence we have

        \begin{align*}
            \mu(B) = \frac{(\deg B' + \deg B'')}{ \rk B} 
        &\leq \frac{\mu_{max}(A') \rk B' + \mu_{max}(A'')\rk B'' }{ \rk B}
        \\& \leq  \frac{\max \lbrace\mu_{max}(A'),\mu_{max}(A'')\rbrace (\rk B' + \rk B'') }{\rk B}
        \\& = \max \lbrace\mu_{max}(A'),\mu_{max}(A'')\rbrace 
        \end{align*}
        If either $B'$ or $B''$ is zero, it is clear that we still have the desired inequality. This finishes the proof of (1). The proof of (2) is similar (or by applying (1) to the dual bundles of $A$, $A'$ and $A''$).
    \end{proof}
    Applying the claim successively to the filtration $0=\E_0 \subset \E_1 \subset \E_2 \subset ... \subset \E_m=\E$, we get 
    \begin{itemize}
        \item $\mu_{max}(\E) \leq \max \lbrace \mu (\E_1) , \mu_{max}(\E/\E_0)    \rbrace \leq ... \leq  \max \lbrace \mu (\E_1) ,..., \mu (\E_m/\E_{m-1})   \rbrace$
        \item $\mu_{min}(\E) \geq \min \lbrace \mu (\E_1) , \mu_{min}(\E/\E_0)    \rbrace \geq ... \geq  \min \lbrace \mu (\E_1) ,..., \mu(\E_m/\E_{m-1})   \rbrace$
    \end{itemize}
    Finally these inequalities imply
    \begin{align*}
     \mu_{max}(\E)-\mu_{min}(\E)   \leq  \max \lbrace d_1 ,..., d_m   \rbrace-  \min \lbrace d_1 ,..., d_m   \rbrace        
    \end{align*}
        This finishes the proof of the lemma.

\end{proof}

\begin{prop}\label{pgl3 iterative extension}
  Let $G=\PGL_3$ and $E \in S_h$, then $r(E) \in \bun_G(\overline{C})$ can be obtained by iterative extensions of line bundles $L_i$ for $i=1,2,3$, with $|\deg L_i -\deg L_{j}|\leq 6g-6$ for $1\leq i <j \leq 3$.
\end{prop}
\begin{proof}

Since the space $\S'_{HF}(\bun_G(C))$ of Hecke-finite functions is the direct sum of the regular elliptic pieces, the cuspidal part of the zero piece, the Hecke-finite part of the regular nilpotent piece (cf. Propositons 5.10(1) and 5.13) and the Hecke-finite part of the subregular nilpotent piece (cf. Proposition 6.3), it suffices to prove that the support of functions in each of these orbits satisfy the proposition. This follows from Corollary \ref{type A cusp support} and  Lemma \ref{supportnilp} below, which are valid for more general $G$. In fact, for $G=\PGL_n$, we show that something stronger holds for the zero and almost regular orbits $\Omega$, namely, if $\E$ is in the support of a function in $\S'_\Omega$ and $r(\E) \in \bun_G^\lambda(\overline{C})$ then $\langle \lambda, \alpha \rangle \leq (n-1)(2g-2)$ for all simple root $\alpha$ of $G$ (see Lemma \ref{supportregellip} and Corollary \ref{type A cusp support}).

\end{proof}

\subsubsection{Drinfeld-Gaitsgory Strangeness and $H^1$-vanishing}\label{drinfeld-gaitsgory} For computing the cuspidal support of the zero and regular elliptic orbits, we need the notion of strangeness introduced by Drinfeld and Gaitsgory in \cite{DG2}, which is already important for bounding Harder-Narasimhan factors for cuspidal functions in the classical setting for curves over finite fields. In this subsection, we assume that $C$ is a smooth projective curve over a field $k$, $G$ is a connected split reductive group over $k$ and $Z_0(G)$ is the identity component of the center of $G$. In this subsection, we do not have any assumption on the characteristic of $k$, unless specified otherwise.

\begin{defn/prop}(\cite[Lemma 10.3.2 and Definition 10.3.4]{DG2})
    Let $V$ be a $G$-representation over $k$ where $Z_0(G)$ acts via a character $\chi_V$, the \textit{strangeness} $\str(G,V))$ is the smallest rational number $c$ such that for any semistable bundle $\mathcal{E} \in \bun_G(C)$, all line subbundles of the associated bundle $V_{\mathcal{E}}$ have degrees less than or equal to $\langle \deg(\mathcal{E}) , \chi_V \rangle + c$. 
\end{defn/prop}

\begin{rem} \label{strangeness remark}

\begin{enumerate}
\item  When $G=\GL_n$ and $V$ is the standard representation, the number $\langle \deg(\mathcal{E}),\chi \rangle = \mu(V_{\mathcal{E}})$, then it follows by the definition of semistability of $\mathcal{E}$ that $\str(G,V)=0$. 
    \item If the ground field $k$ is of characteristic $0$, then $\str(G,V)=0$. (\cite[10.4.3]{DG2}, \cite{RR})
    \item If every simple factor of $G$ is of type $A$, then $\str(G,V)=0$ for any ground field $k$ (\cite[Corollary 10.5.2]{DG2}). In this case, Proposition \ref{vanishingH1} can be proved using Harder-Narasimhan filtration of vector bundles.
    \item There are nonzero strangeness known in the literature, but they seem to be all in small characteristics. It would be interesting to see whether the condition that the characteristic of the ground field is greater than Coxeter number of the group would guarantee zero strangeness. 
    
\end{enumerate}
    
\end{rem}

\begin{defns}
Let $M$ be a Levi subgroup of $G$, and let $\alpha$ be a root of $G$ which is not a root of $M$. Following \cite{DG2}, we consider the $M$-module
$$
V_{M,\alpha}:=\bigoplus_{\gamma,\ \gamma-\alpha\in R(M)}\mathfrak{g}_\gamma,
$$
where $R(M)$ is the root lattice of $M$.

For each simple root $\beta$ of $G$, we define $c_\beta$ to be the smallest rational number $c$ satisfying
$$
c\geq
\frac{2g-2+\str(M,V_{M,\alpha}^*)}
{\operatorname{coeff}_\beta(\alpha)}
$$
for every Levi subgroup $M$ such that $\beta$ is not a root of $M$, and every root $\alpha$ such that $\operatorname{coeff}_\beta(\alpha)>0$. Here $\operatorname{coeff}_\beta(\alpha)$ denotes the coefficient of $\beta$ in $\alpha$.
\end{defns}

\begin{exmp}
    When the simple factors of $G$ are all of type A, by Remark we have \ref{strangeness remark}, $c_\beta=2g-2$ for all simple roots $\beta$.
\end{exmp}

\begin{prop}\label{vanishingH1}(\cite[Proposition 10.4.5]{DG2})
Let $M$ be a Levi quotient of a parabolic subgroup $P$ of $G$. If $\E \in \bun_M^\lambda(C)$, where $\lambda \in \Lambda_G^{+,\mathbb{Q}}$, and $ \langle \lambda,\alpha \rangle > c_\alpha$ for every $\alpha \in \Delta_G-\Delta_M$, and let $L\in \Pic^0(C)$, then $H^1(C,\text{Lie}(N)_{\E} \otimes L)=0$ where $\text{Lie}(N)_{\E}$ is the associated vector bundle.
\end{prop}
\begin{proof}(Sketch)
    When $L=\mathcal{O}_C$, this is exactly the Proposition 10.4.5 in \cite{DG2}. We sketch the proof here, and explain how it easily adapts to arbitrary line bundle $L$ of degree zero.
    Let $P_\lambda$ be the parabolic subgroup of $M$ corresponding the subset of roots $\lbrace \beta: \text{roots of $M$}| \langle \lambda, \beta \rangle=0 \rbrace$ and $M_\lambda$ the corresponding Levi. Since $\E \in  \bun_M^\lambda(C)$ . It admits a $P_\lambda$-reduction and we denote the corresponding semistable $M_\lambda$-bundle by $\E_\lambda$. Then $\text{Lie}(N)_{\E}$ has a canonical filtration whose associated graded is $\text{Lie}(N)_{\E_\lambda}$. As an $M_\lambda$-module, $\text{Lie}(N)$ decomposes into $V_{M_\lambda, \gamma}$'s.  Now, it suffices to show that
    $$H^1(C,{V_{M_\lambda, \gamma,}}_{\E_\lambda} 
    \otimes L)\simeq \Hom(L\omega_C^{-1},{V_{M_\lambda, \gamma,}}_{\E_\lambda} ^*
    )^*=0.$$
   By Remark 10.3.5 \cite{DG}, it is enough to show for any positive root $\gamma$ of $G$ which is not a root of $M$ we have
    $$ 2g-2 < \langle \lambda ,\gamma \rangle - \str(M_\lambda, V_{M_\lambda, \gamma} ^*)$$
Finally, the above inequality follows by our assumption on $\lambda$.

\end{proof}

\subsubsection{The zero and regular elliptic orbits} We resume the assumption that the characteristic of $\overline{F}$ is at least $c(G)-1$, where $c(G)$ is the Coxeter number.
\begin{lem}\label{supportregellip}
    Let $G$ be a connected split reductive group and set $\mathfrak{B}=\bigcup_\lambda \bun^\lambda_G$ where $\langle \lambda, \beta\rangle \leq c_\beta$ for every simple root $\beta$ of $G$. Then we have:
    \begin{enumerate}
        \item The total support of $\S'_{HF}(\bun_G(\overline{C}))=\S'_{cusp}(\bun_G(\overline{C}))$ is contained in $\mathfrak{B}$
        \item Let $\Omega$ be a regular elliptic orbit. The total support of $\S'_{\Omega}(\bun_G(C))$ is contained in $\mathfrak{B}$.
    \end{enumerate}
\end{lem}
\begin{proof}
For curve $\overline{C}$ over a finite field, it is known that $\S'_{HF}(\bun_G(\overline{C}))=\S'_{cusp}(\bun_G(\overline{C}))$ (\cite[Lemme 8.15]{Laf}), and we show that the support of cuspidal functions is in $\mathfrak{B}$. Let $E \in \bun_G^\lambda(\overline{C}) $ for some dominant coweight $\lambda$ and denote by $P_E$ (resp. $M_E$ and $N_E$) the canonical parabolic of $E$ (resp. its Levi component and unipotent radical). Let $f$ be a cuspidal function, if there is a simple root $\beta$ such that $\langle \lambda, \beta \rangle > 2g-2+c_\beta$, then we can consider the maximal parabolic subgroup $P=P_\beta$ corresponding to the Dynkin diagram with $\beta$ removed, and denote by $M$ the Levi factor, and consider the constant-term $CT_{P}$. Let $E_{P}$ be a $P$-reduction of $E$ (this exists because $P$ contains the canonical parabolic of $E$) and $E_{M}$ the corresponding $M$-bundle, we must show that $E_{M}$ has a unique $P$-structure and hence $CT_{P}f(E_{M})= f(E)$ and hence vanishes by the cuspidality of $f$. By \cite[Remark B.3.3]{DG}, it suffices to show $H^1(\overline{C},\mathfrak{n}_{E_{M}})=0$, where $\mathfrak{n}$ is the Lie algebra of the unipotent radical of $P$ and $\mathfrak{n}_{E_{M}}$ is the associated vector bundle. Now, the vanishing follows by Proposition \ref{vanishingH1}. This finishes the proof for (1).
 
Now we prove (2). Let $f\in \S'_{\Omega_\eta}$ for some regular elliptic orbit $\Omega$ and $E\in \bun_G(C)$ such that $f(E)\neq0 $. Then $r(E)$ is the underyling $G$-bundle of a semistable Higgs $G$-bundle on $\overline{C}$ (see the proof of Corollary \ref{regellipfd}) , and we will show that $r(E) \in \mathfrak{B}$. If $r(E)$ is a semistable $G$-bundle then we are done. Now assume the canonical parabolic of $\overline{E}:=r(E)$ is is not $G$ and $\overline{E}\in \bun_G^\lambda(\overline{C})$, and there is a simple root $\beta$ such that $\langle \lambda, \beta\rangle >c_\beta$. Then by Serre's duality and Proposition \ref{vanishingH1},  we have 
$$H^0(\overline{C},\mathfrak{n}_{\overline{E}_M}^\vee \otimes \N \inv \omegac)) \simeq H^1(\overline{C}, \mathfrak{n}_{\overline{E}_M}\otimes \N)^*=0.$$
But we have $\mathfrak{n}_{\overline{E}_M}^\vee \simeq \mathfrak{n}_{\overline{E}_M}^-$, where $\mathfrak{n}^-$ is the cokernel of $P$-module map $\mathfrak{p} \to \mathfrak{g}$, so we have $H^0(\overline{C}, \mathfrak{n}^-_{\overline{E}_M} \otimes \N \inv \omegac)=0$. Since $\mathfrak{n}^-_{\overline{E}_M}$ can be identified with the direct sum of the graded pieces of the filtration on $\mathfrak{n}^-_{\overline{E}_P}$ (see the proof of \cite[Lemma 10.2.1]{DG}), we have $H^0(\overline{C}, \mathfrak{n}^-_{\overline{E}_P} \otimes \N \inv \omegac)=0$. Now, we consider the short exact sequence of vector bundles
$$ 0 \to \mathfrak{p}_{\overline{E}_P} \otimes \N \inv \omegac \to \mathfrak{g}_{\overline{E}_P} \otimes \N \inv \omegac \to \mathfrak{n}^-_{\overline{E}_P} \otimes \N \inv \omegac  \to 0,$$
and since $\mathfrak{g}_{\overline{E}_P}\simeq \mathfrak{g}_{\overline{E}}$, we get $$H^0(\overline{C},\mathfrak{g}_{\overline{E}}\otimes \N \inv \omegac)) \simeq H^0(\overline{C}, \mathfrak{p}_{\overline{E}_P} \otimes \N \inv \omegac).$$
But this means $\overline{E}$ cannot be the underlying $G$-bundle of a semistable $G$-Higgs bundle, and hence it cannot be in the support of a function in $\tS_\eta^{G(\O')}$.
\end{proof}

\begin{rem}
    In the proof of Lemma \ref{supportregellip}(2), the identification $\mathfrak{n}^\vee \simeq \mathfrak{n}^-$ of $M$-representations does not need to exist over ground field of arbitrary characteristic. In general, without such an identification we may need to choose different constants $c_\beta$ for the almost regular elliptic orbits. This reflects the fact that in \cite[Proposition 10.4.5]{DG}, $c_i'-c_i''$ may not be $2g-2$.
\end{rem}

The lemma and Remark \ref{strangeness remark} yields:
\begin{cor}\label{type A cusp support}
    Let $G$ be a product of adjoint groups of type $A$, and $\Omega$ be any regular elliptic orbit. Then the total support of $\S'_{HF}(\bun_G(\overline{C}))=\S'_{cusp}(\bun_G(\overline{C}))$ and $\S'_\Omega(\bun_G(C))$ are both contained in $\mathfrak{U}_{n-1}$, where $G$ is the semisimple rank of $G$.
\end{cor}

\subsubsection{Regular and subregular nilpotent orbits}
\begin{lem} \label{supportnilp} We have:
\begin{enumerate}
    \item Let $G$ be a semisimple adjoint group with simple factors of types $A_n$, $B_n$, $C_n$ or $G_2$. Let $E \in \bun_G(C)$ be in the support of a Hecke-finite function in the regular nilpotent summand. Then $r(E)$ admits a $B$-structure whose induced $T$-bundle $r(E)_T$ satisfying that the associated line bundle $(\mathfrak{g}_\gamma)_{r(E)_T}$ has degree bounded by $-(2g-2)(c-1)$ and $(2g-2)(c-1)$ for each positive root $\gamma$, where $c$ is the Coxeter number of $G$.
    \item Let $G=\PGL_n$. Let $E \in \bun_G(C)$ be in the support of a Hecke-finite function in the subregular nilpotent summand. Then $r(E)$ can be obtained by iterative extensions of line bundles $M_i$ with $|\deg M_i -\deg M_{j}|\leq (2g-2)(2n-3)$ for $1 \leq i <j \leq n$.
    
\end{enumerate}
    
\end{lem}

\begin{proof}
The proofs for the regular and subregular nilpotent orbits rely on the geometric interpretation (Proposition \ref{geom_sphvec}) of Propositions \ref{regnilfd} and \ref{subgreg_fd}.

We first prove (1), i.e., the regular nilpotent case. Note that the statement for $G$ holds if that holds for each of its simple factors, so it suffices to prove it for simple $G$. Now we assume in addition that $G$ is simple, and we will follow the notation in Proposition \ref{regnilfd}. We have that if $E$ is in the support of some function in $\S'_{\eta, \chi }^{G(\O')}$, where $\eta$ is a regular nilpotent element and the character $\chi:U_h(\bA')/U_h(\bF) \to U(1)$ is nontrivial, then $r(E)$ is an iterative extension of the line bundles $L_i$ (where $(L_1,...,L_n)$ is a $T$-bundle associated to $t=(t_1^{-1},...,t_n^{-1}) \in T(\bA')$ in the proposition). It was shown that $\gamma(t)\alpha \in \N\inv \omega(\bO)$ for all simple roots $\gamma$ which implies $\deg \gamma(t) \geq -2g+2$. But this implies that for any positive roots $\gamma'$ we have 
$$\deg \gamma'(t) \geq -(c-1)(2g-2),$$ and the equality holds only if $\gamma'$ is the highest root $\gamma_h$.  It was also shown that we have $\gamma_h(t)^{-1}\beta \in \omega(\bO) $, which implies that $\deg \gamma_h(t) \leq 2g-2$. Let $\gamma'$ be a positive root that is not the highest root. Then we have
 $$ \deg \gamma'(t) \leq (2g-2)+(c-2)(2g-2)=(c-1)(2g-2).$$

For the subregular nilpotent case, we follow the notation in Proposition \ref{subgreg_fd}. It is immediate by the computation in the proposition that if $E$ is in the support of a function $f$ in $\S'_{\tilde{H},\chi}^{G(\O')}$ (where $\chi_x$ and $\chi_y$ are nontrivial), then it has a $B$-structure whose induced $T$-bundle is $(M_1:=\mathcal{O}_{\overline{C}},M_2:= L_a,M_3:= L_1,..., M_n:=L_{n-2})$  (given by $(1,a^{-1},t_1^{-1},...,t_{n-2}^{-1}) \in T(\bA')$). The condition (6.4) implies that for for $1\leq i <j \leq n-2$ we have $$\deg L_{i} -\deg L_{j} \leq (j-i)(2g-2) \leq (n-3)(2g-2). $$ 

Note that (6.6) implies $ \deg L_r \geq -(n-r)(2g-2)$, and that (6.7) implies
$\deg L_r \leq r(2g-2)$. Using these inequalities we get 
$$ \deg L_i -\deg L_j \geq -(n-i+j)(2g-2) \geq -(2n-3)(2g-2)$$

By (6.5), $a^{-1}a_x$ and $aa_yt_{n-2}$ are regular, we have 
$$ -(n-1)(2g-2)=(2-n)(2g-2)-(2g-2)\leq \deg L_a \leq 2g-2.$$
Finally, we can easily see that $|\deg M_i -\deg M_j| \leq (2n-3)(2g-2)$ for any $i,j$.

\end{proof}

\begin{rem}
    By the main result of the orbit decomposition for $G=\PGL_2$ in \cite{BKP}, the space of spherical Hecke-finite (equivalently cuspidal) functions is contained in $\mathfrak{U}_1$ by the above arguments
\end{rem}

\section{ (2,...,2)-nilpotent orbits in $\mathfrak{pgl}_{2n}$}\label{misc}
In this section, we let $G=\PGL_{2n}$ and keep our usual notation that $\overline{C}$ is a smooth projective curve over a finite field $k$ and $C$ is a square-zero extension. As usual, we assume the characteristic of $k$ is greater than $2n$. We study automorphic functions in the nilpotent orbit $\S'_{\Omega_\eta}$ of type $(2,...,2)$. The main observation and calculations relate a subrepresentation of $\S'_{\Omega_\eta}$ to the automorphic representation of $\PGL_n(\A'_0)$, where $\A'_0$ is the ring of adeles of the split extension $\overline{C}\times_k k[\epsilon]$.

Throughout this section we will work with the following representative   

$$\eta= 
\begin{pmatrix}
0 &  \alpha I_{n\times n} \\
 & 0  

\end{pmatrix}
 \in \mathfrak{pgl}_{2n} \otimes \N \inv \omega(\bF),$$
where $\alpha \in \N\inv \omegac(\bF)$. We may and do assume that $\alpha$ is regular. Then a straightforward calculation shows 
\begin{prop}
The stabilizer $\St_\eta(\bA')$ consists of the matrices 
      \[
\begin{pmatrix}
A &  X \\
 & A  

\end{pmatrix},
\]
where $A\in \GL_n$ and $X \in \mathfrak{gl}_n$. We have a short exact sequence
$$ 1 \to U_{n,n}(\bA') \to \St_\eta(\bA') \to M'_{n,n}(\bA')/\bA'^* \to 1,$$
where $U_{n,n}(\bA')$ is the unipotent radical of the $(n,n)$- standard parabolic subgroup and $M'_{n,n}(\bA')$ is the subgroup of the Levi quotient consisting of elements $\diag(A,A)$. We also have $\St_\eta(\bA')$
isomorphic to  $\PGL_n(\bA')\ltimes \mathfrak{gl}_n(\bA')$, with one dimensional center $Z\simeq \bA'$, and the action of $\PGL_n(\bA')$ on the quotient $\mathfrak{gl}_n(\bA')/Z \simeq \mathfrak{pgl}_n(\bA')$ is given by the adjoint action.
        
\end{prop}

Note that we have $$\St_\eta(\bA')\simeq \PGL_{2n}(\bA')\ltimes (\bA'^{n^2}/Z) \simeq \PGL_n(\mathbb{A}_0),$$ where $\mathbb{A}_0$ is the adeles of $C_0 :=\bar{C} \times_k \Spec k[\epsilon]$ (we will also denote the total ring of fractions of $\mathcal{O}_{C_0}$ by $F_0$), and  where we identify $(U_{n,n}/Z)$ with $\mathfrak{pgl}_n$. We will use this identification for the rest of this section.

Let $\psi_Z$ be the central characters of $\St_\eta(\bA')$, then we have the following decomposition
$$ \S'(\St_\eta(\bF) \backslash \St_{\eta}(\bA')) \simeq \bigoplus_{\psi_Z}  \S'_{\psi_Z}(\St_\eta(\bF) \backslash \St_{\eta}(\bA'))  $$
$\St_\eta(\bA')$-representations, and each $\S'_{\psi_Z}(\St_\eta(\bF) \backslash \St_{\eta}(\bA'))$ is isomorphic to the automorphic representation $\S'(\PGL_n(F_0)\backslash \PGL_n(\mathbb{A}_0))$ twisted by the central character $\psi_Z$.

When $\psi_Z=1$, there is a decomposition of $\PGL_n(\mathbb{A}_0)$-representations
    $$ \S'(\PGL_n(F_0)\backslash \PGL_n(\mathbb{A}_0))\simeq \bigoplus_{\eta'\in \mathfrak{pgl}_n\otimes \omegac(\bF)/ \PGL_n(\bF)} \S'_{\Omega_{\eta'}}, $$
     by the orbit decomposition, where $ \S'_{\Omega_{\eta'}}:= \ind_{\St_{\eta'}}^{(U_{2,2}/Z)(\bA') \rtimes M'_{2,2}(\bA') } \S'_{\eta'}$ and $\S'_{\eta'}$ is the $\St_{\eta'}$-representation consisting of functions on $\St_{\eta'}$.... In general, for any $\psi_Z$, we let $V_{\eta', \psi_Z}:=\ind_{\St_\eta}^{G(\mathbb{A})} ((\S'_{\Omega_{\eta'}}, \psi_Z), \psi_{\eta})$, then we have a decomposition of $G(\A')$-representations $$\tS_{\eta, \psi_Z} \simeq \bigoplus_{\eta'} V_{\eta',\psi_Z} .$$ 
    When $\psi_Z=1$, we denote  $ V_{\eta'}:=V_{\eta',\psi_Z}$. When $\eta'=0$, we have 
    $$ V_{\eta',\psi_Z} = V_{\eta',\psi_Z, cusp} \oplus V_{\eta',\psi_Z, Eis},$$
    where the summands correspond to the cuspidal $\PGL_n(\bA')$-representation and Eisenstein series respectively.

\subsection{Hecke-infinite functions}
In this subsection, we assume $G=\PGL_4$, unless specified otherwise. We deduce Hecke-infiniteness of some subrepresentation of the $(2,2)$-nilpotent orbit by using results from the orbit decomposition of the automorphic representation of $\PGL_2(\A'_0)$.

If $\eta'= \alpha' E_{14} \in U_{2,2} (\omega(\bF))$, i.e. $\eta'$ is nonzero nilpotent, we further have 
    $$ V_{\eta',\psi_Z} = \bigoplus_\chi V_{\eta',\psi_Z, \chi},$$
    where $\chi$ runs through the characters $U'_h(\bA')/U'_h(\bF)$, where $U'_h:=\St_{\eta'}$.
\begin{prop} Let $K$ be a compact open subgroup of $G(\A')$. We have:
\begin{enumerate}

    \item Any nonzero element in each of the following $\Hk$-modules is Hecke infinite:
    
    \begin{itemize}
        \item $V_{\eta', \psi_Z}$ for split semisimple element $\eta'$
        \item $V_{\eta', \psi_Z , \chi}$ for nonzero nilpotent $\eta'$ and $\chi=1$
        \item $V_{\eta', \psi_Z, Eis}$ for $\eta'=0$ 
    \end{itemize}

    \item If $\psi_Z$ is trivial and $\eta'=0$, then the $G(\mathbb{A})$-representations $V_{\eta'}$ has no admissible subrepresentation.
\end{enumerate}

\end{prop}

\begin{proof}
  Since $\S'_{\Omega_{\eta'}}$ has no admissible subrepresentation except possibly when $\eta'$ is nilpotent or regular elliptic, (1) follows immediately.

Now we prove (2). First, we note that $ (U_{2,2}/Z)(\bA') \subset \St_{\eta'}=\St_\eta(\bA')=(U_{2,2}/Z)(\bA') \rtimes M'_{2,2}(\bA')$ acts trivially on $\S'_{\Omega_{\eta'}}$. Let $f\in \S'_{\Omega_{\eta'}} $ and $V_f$ be the $\St_{\eta'}$-representation generated by $f$, by Proposition \ref{nilpnonadm} we must construct an infinite family of $g\in \overline{K}\backslash G(\bA')/\St_\eta(\bF)$ such that $\Ad(g)\eta$ are in different orbits and $ V_f^{\overline{K} \cap \St_\eta(\bA')} = V_f^{g^{-1}\overline{K}g \cap \St_\eta(\bA')}$. For this, we consider $g=\diag(1,1,t^{-1},1)$, where $t\alpha \in \omega(\bO)$. Now, we compute that $g^{-1}\overline{K}g \cap \St_\eta(\bA')$ equals $\overline{K} \cap \St_\eta(\bA') $ modulo $ (U_{2,2}/Z)(\bA')$ . But that immediately implies 
$$ V_f^{\overline{K} \cap \St_\eta(\bA')} = V_f^{g^{-1}\overline{K}g \cap \St_\eta(\bA')},$$
because $(U_{2,2}/Z)(\bA')$ acts trivially on $V_f$. This finishes the proof for (2).

\end{proof}

\begin{prop}
Let $G=\PGL_{2n}$. When $\eta'=0$ and $\psi_Z=0$, all elements in $\kappa_nV_{\eta', \psi_Z}$ are Eisenstein series. 
\end{prop}

\begin{proof}
    Let $P\subset \PGL_{2n}$ be the parabolic subgroup of type $(n,n)$ with unipotent radical $U=U_{n,n}$ with Lie algebra $\mathfrak{u}$ and Levi $M$. Since both $U(\bA')$ and $1+\epsilon \mathfrak{u}(\bA') \subset \Na$ act trivially on ${\S'}_{\eta', \psi}$ (because $\eta'$ and $\psi_Z$ are trivial and $\mathfrak{u} \subset \eta^\perp$), we have 
    $$V_{\eta', \psi_Z}=V_{\eta', \psi_Z}^{U(\mathbb{A})}. $$

    Let $K$ be a compact open subgroup of $G(\A')$. We consider the diagram
    \[
\begin{tikzcd}
 & \St_\eta(F) \backslash G(\A') /K  \arrow[d , "\pi"] \arrow[dr , "obv"] & \\
QBun_{M,K} & QBun_{P,K} \arrow[l , "p"] \arrow[r , "q" ] & Bun_{G,K} , 
\end{tikzcd}
\]
then we have $\pi_* (V_{\eta', \psi_Z})^K \subset p^*(\S'(QBun_M))$ because all functions in $V_{\eta', \psi_Z}$ are left $U(\A')$-invariant. This finishes the proof.
\end{proof}

\subsection{Hecke-finite functions} In the subsection, we work with $\eta'=0$ and $\psi_Z\neq 0$, and hence $\St_{\eta'}=\St_\eta(\bA')$. We also assume the genus $g$ of $\overline{C}$ at least 2.
\begin{prop} 
   Assume $\psi_Z$ is nontrivial and $\eta'=0$. Then we have
   
   \begin{enumerate}
    \item All elements in $V_{ \psi_Z, 0, cusp }^{G(\O')}$ are Hecke-finite, and $V_{ \psi, 0, cusp }^{G(\O')}$ is finite-dimensional.
    \item The support of any function in $\kappa_\eta V_{ \psi_Z, 0, cusp }^{G(\O')}$ is contained in the $r^{-1}(\mathfrak{U}_{3n+3})$.
   \end{enumerate}
   
\end{prop}

\begin{proof}
    Let $f\in V_{ \psi_Z, 0, cusp }^{G(\O')}$ and $W=V_f$ be the $\St_{\eta}(\bA')$-representation generated by $f$. By Lemme 8.15 \cite{Laf}, $W_{\psi_Z,0, cusp}$ is admissible $\PGL_n(\bA')$-representation, and indeed for any compact open subgroup $K$ of  $\PGL_n(\bA')$, $W_{\psi_Z,0, cusp}^K$ is Hecke-finite. In view of Proposition \ref{sphvec}, it suffices to show that $W_{\psi_Z,0, cusp}^{gG(\mathcal{O)}g^{-1} \cap \St_\eta }=0$ for all but finitely many $g$. 

     As in the proof of \ref{regnilfd}, it suffices to consider $g=\diag(1,..., 1, a_1,...,a_n)$ where $a_i \in \bA'^*$ such that $a_i\alpha \in \N \inv \omega(\bO)$. Next, we check that $W_{\psi_Z,0, cusp}^{gG(\mathcal{O)}g^{-1} \cap \St_\eta }\neq 0$ only if  $a_i^{-1}\bO $ is contained in the conductor $\Lambda_{\psi_z}$ of the $\psi_z$ for all $i$. First, we compute that 
     $$gG(\bO) g^{-1} \cap \St_\eta = \Lambda \rtimes \PGL_n(\bO),$$
     where $\Lambda \subset \mathfrak{gl}_n(\bA')$ consists of the matrices $\left( \begin{smallmatrix}
I_n & A \\
   & I_n\\
\end{smallmatrix} \right)$, where the entries of the $i$-th columns of $A$ are in $a_i^{-1}\bO$.

We claim that $ W_{\eta',0, \psi_Z}^{gG(\bO) g^{-1} \cap \St_\eta }\neq 0$ only if $a_i^{-1}\bO$ are contained in the conductor of $\psi_Z$ for all $a_i$'s. Suppose $a_i^{-1}\bO $ is not contained in the conductor of $\psi$ for at least one $i$ (WLOG, say for $i=1$), then we consider $f\in W_{\eta',0, \psi_Z}^{gG(\bO) g^{-1} \cap \St_\eta}$, $a \in a_1^{-1}\bO \backslash \Lambda_{\psi_z}$ and $a' \in a_2^{-1}\bO$, then we have

\begin{align*}
f(g)=f(g \begin{pmatrix}
I_n & \operatorname{diag}(a,0,\dots,0)\\
0 & I_n
\end{pmatrix})
&= \psi_Z(\frac{a}{n}) f(g \begin{pmatrix}
I_n & \operatorname{diag}(\frac{(n-1)a}{n},-\frac{a}{n},\dots,-\frac{a}{n})\\
0 & I_n
\end{pmatrix})
\\ &=  \psi_Z(\frac{a}{n}) f(g),\end{align*}
the last equality holds because every function in $W_{\eta',0, \psi_Z}$ is right invariant under the subgroup $\mathfrak{pgl}_n(\bA') \subset \St_\eta(\bA')$.  This implies $f=0$. This finishes the proof for (1).

To prove (2), we note that the above calculation shows that if $E'\in \bun_G(C)$ is in the support of a function $f \in V_{ \psi_Z, 0, cusp }^{G(\O')}$, then $E:=r(E')\in \bun_G(\overline{C})$ is the bundle corresponding to 
$$g_E=u \diag(A,A)  \diag(h,h) \diag(I_n, \diag(a_1,...,a_n)u')  \in G(\bA') $$
for some $A \in \GL_n(\bA')$, $u \in U_{n,n}(\bA')$, $h\in \GL_n(\bO)$ and $u'$ is a unipotent upper triangular matrix with $\bA'$-coefficients in $\GL_n(\bA')$. Since $A$ is in the support of a cuspidal function on $\PGL_n(\bF)\backslash \PGL_n(\bA') / \PGL_n(\bO)$, by Lemma \ref{semistable bound}(3), $A$ can be written as $A= A' g$ where $g \in \GL_n(\bO)$ and $A'$ is upper triangular with diagonal $\diag(a_1',...,a_n')$ such that $|\deg a_i' -\deg a_j' |\leq (n+1)(2g-2)$. So we have 
$$ g_E = u \diag (A',A') \diag (gh,gh) \diag (I_n, \diag(a_1,...,a_n)u') $$
Since $(\alpha I_n) (\diag(a_1,...,a_n)u') = \alpha\diag(a_1,...,a_n)u' \in \mathfrak{gl}_n \otimes \N \inv \omega (\bO)$, we have $a_i\alpha \in \N \inv \omega (\bO)$ and hence $v_x(a_i) \geq -v_x(\alpha)$ for every $x\in \overline{C}$. On the other hand, since $a_i^{-1} \bO \subset \Lambda_{\psi_Z}$, we also have $v_x(a_i) \leq v_x(a_Z) $, where $a_Z \in \omega(\bF)$ corresponds to the character $\psi_Z$. Also, since all entries $u'_{ij}$ of $u'$ satisfy that $\alpha a_i u_{ij} \in \N \inv \omega(\bO)$, we have $pole_x(u_{ij}) \leq v_x(\alpha)+v_x(a_i)$ for every $x \in \overline{C}$.  Now by Lemma \ref{cartan type}, we have
$$ g_E = u \diag (A',A') \diag (I_n, \diag(a_1'',...,a_n'')\diag(a_1''',...,a_n''')y) \diag (gh,g'),$$
where $g' \in \GL_n(\bO)$, $a_i'', a_i''' \in \bA'^*$ with $|\deg a_i'' -\deg a_j''| \leq 4g-4$ and $|\deg a_i''' -\deg a_j'''| \leq 2n(2g-2)$. This shows that $E$ is an iterative extension of line bundles $M_i$ such that $|\deg M_i-\deg M_j| \leq (3n+3)(2g-2)$ for $1 \leq i <j \leq 2n$.

\end{proof}
\begin{defn}\label{pole}
   Let $a\in \bA'$.  For each $p \in \overline{C}$, we define the local pole of $a$ at $p$ by $pole_p(a)= \max(-v_p(a), 0)$, and the total pole $pole(a)= \sum_p pole_p(a)\deg(p)$.
\end{defn}

\begin{lem}\label{cartan type} Let $G=\GL_n$, $B(\bA')$ be the Borel subgroup of upper triangular matrices with unipotent radical $U(\bA')$ and quotient $T(\bA')$. For $t=(t_1,...,t_n) \in T(\bA')$, we set $m(t):= \max \lbrace \deg t_i -\deg t_j \rbrace$.
    \begin{enumerate}
        \item Let $t \in T(\bA')$ and $g \in G(\bO)$, then there exists $t' \in T(\bA')$, $u\in U(\bA')$ and $g' \in G(\bO)$ such that $gt=t'ug'$ and $$m(t')\leq \sum_{x \in \overline{C}} (\max_i(v_x(t_i))-\min_i(v_x(t_i)) )\deg(x).$$
        \item Let $u \in U(\bA')$ and $g \in G(\bO)$. Let the $(i,j)$-th entry of $u$ be $u_{ij}$. Then there exists $t \in T(\bA')$, $u' \in U(\bA')$ and $g' \in G(\bO)$ such that $gu=tu'g'$ and 
     $$m(t) \leq n \sum_{x\in \overline{C}} \max_{i<j} (pole_x(u_{ij})) \deg(x).$$

    \end{enumerate}
\end{lem}
\begin{proof} Both parts are reduced to local problems, and the existence, except the condition on $m(t')$ (resp., on $m(t)$) in (1) (resp., in (2)), is due to the Iwasawa decomposition. It remains to show the inequalities $m(t')\leq \sum_{x \in \overline{C}} (\max_i(v_x(t_i))-\min_i(v_x(t_i)) )\deg(x)$ in (1) and  $m(t) \leq n \sum_{x\in \overline{C}} \max_{i<j} (pole_x(u_{ij})) \deg(x)$ in (2). Note that for the first inequality, it suffices to prove $\max_i v_x(t_i') - \min_i v_x(t_i')\leq \max_i v_x(t_i) - \min_i v_x(t_i)$ for all $x \in \overline{C}$, and for the second inequality, it suffices to prove $\max_i v_x(t_i) - \min_i v_x(t_i) \leq n\max_{i<j} pole_x(u_{ij})$ for all $x\in \overline{C}$. We now work locally at a point on $\overline{C}$ for the rest of the proof, and denote by $v$ the corresponding valuation.

For (1), we write $t'=(t_1',...,t_n')$. Note that the fractional ideal generated by the entries of $t$ is the same as the fractional ideal generated by the entries of $t'u$, and this implies the $v(t_i') \geq \min \lbrace v(t_i)\rbrace$. Applying the same argument to $t^{-1}$ and $(t'u)^{-1}$, we obtain $-v(t_i') \geq - \max\lbrace v(t_i)\rbrace $ and hence  $v(t_i') \leq  \max\lbrace v(t_i)\rbrace $. These imply 
$$\max_i v(t_i') - \min_i v(t_i')\leq \max_i v(t_i) - \min_i v(t_i).$$

For (2), we denote by $\max(u)$ the maximal pole of all entries of $u$, i.e. $\max(u)=\max_{i<j} pole(u_{ij})$. By considering the fractional ideal generated by the entries of $u$ and that by entries of $tu'$ we have $v(t_i) \geq -\max(u)$. On the other hand, by considering $u^{-1}$ and $(tu')^{-1}$, we have $-v(t_i) \geq -\max(u^{-1})$ and hence $v(t_i) \leq \max(u^{-1})\leq (n-1)\max(u)$, where the last inequality can be seen by explicitly looking at the entries of $u^{-1}$.
    
\end{proof}

\begin{lem}\label{semistable bound}
    Let $\E$ be a  vector bundle of rank $n$ and slope $\mu$ on a smooth projective curve over a finite field. 
    \begin{enumerate}
        \item There exists a line subbundle $L$ of $\E$ whose degree is at least $\mu-g$. If $\E$ is semistable rank 2 bundle, then $\E$ is an extension of a line bundle of degree at most $\lfloor \mu \rfloor +g$ by a line bundle of degree at least $\lceil \mu \rceil -g$.
        \item If $\E$ is semistable, then $\E$ is an iterative extension of line bundles $L_i$ such that $ \mu-g\leq \deg L_i  \leq \mu+b$ and $b\leq (1+\frac{n-1}{e})g \leq \frac{n+1}{2}(2g-2)$.
        \item Let $(\mu_1,...,\mu_n)$ be the slope sequence of $\E$. If $\mu_i-\mu_{i+1}\leq 2g-2$ for all $i$, then $\E$ is an iterative extension of line bundles $L_i$ such that  $$\max_i{ \deg L_i} - \min_i{ \deg L_i }\leq (n+1)(2g-2).$$
    \end{enumerate}

\end{lem}
\begin{proof}
(1) Let $L'$ be a line bundle of degree $\lceil\mu \rceil-g$. We can construct $L'$ by choosing a degree 1 line bundle bundle $M$ and take $L'=M^{\otimes \lceil \mu \rceil-g}$. Such a line bundle $M$ exists because the Weil bounds of $\overline{C}$ implies that for sufficiently large $N$ (relative to $g$), $\overline{C}$ has a point of degree $N$, and we can take $M= \mathcal{O}(p-p')$ where $p$ (resp. $p'$) is a point in $\overline{C}$ of degree $N+1$ (resp. $N$). By Riemann-Roch we have 
$$ h^0(L'^\vee \otimes \E) \geq n (\mu-\deg L') - n (g-1)>0$$
which is nonzero by the assumption on $\deg L'$. This implies $\Hom (L',\E) \neq 0$. Replacing $L'$ by its saturation $L$ in $\E$ if necessary, we obtain a line subbundle $L$ of $\E$ satisfying $\deg L \geq \deg L'=\lceil\mu \rceil-g$. If $\E$ is a rank 2 bundle, then the quotient line bundle $\E/L$ has degree at most $\lceil\mu\rceil -g$. This finishes the proof of (1).

To prove (2), starting with $\E_0:=\E$, for $0\leq r \leq n-2$, we apply (1) successively to choose a line bundle $L_{r+1} \subset \E_r$ satisfying that 
$\deg L_{r+1} \geq \mu(\E_{r})-g$, and set $\E_{r+1}=\E_r/L_{r+1}$, and finally, set $L_n=\E_{n-1}$.

Let $d=\deg \E$ and $l_i= \deg L_i$. Since  $\E_r$ is a rank $n-r$ bundle with $\deg \E_r = d - \sum_{i=1}^{r} l_i$, we have the inequalities
\begin{align}
   l_{r+1} \geq  \frac{d-\sum_{i=1}^{r} l_i }{n-r}-g 
\end{align}

On the other hand, since $\E_r$'s are quotient of the semistable bundle $\E$, we have
\begin{align}
    d - \sum_{i=1}^{r} l_i \geq (n-r)\mu,
\end{align} or equivalently,
\begin{align}
    p_r:= r\mu-\sum_{i=1}^r l_i \geq 0.
\end{align}
We also have $p_0=p_n=0$. Note that (8.8) and (8.9) give a lower bound for each $l_{r+1}$, namely, 
$$l_{r+1} \geq  \mu+\frac{p_r}{n-r}-g \geq \mu-g.$$ It remains to find an upper bound for each $l_r$. Since $l_r=\mu+p_{r-1} - p_r$, and $p_r \geq 0$, it suffices find upper bounds for $p_r$'s. Note that (8.8) implies that
\begin{align} p_{r+1} \leq p_r - \frac{r\mu -\sum_{i=1}^r l_i }{n-r} +g \leq \frac{n-r-1}{n-r} p_r +g \end{align}
Solving these inequalities iteratively (recall $p_0=0$), for $0 \leq r \leq n-1$, we obtain

\begin{align}
    p_r \leq g(n-r)\sum_{j=n-r}^{n-1}\frac{1}{j}
\end{align}
By considering the usual logarithmic upper bound of harmonic sum, one can directly check that $(n-r)\sum_{j=n-r}^{n-1}\frac{1}{j} \leq 1+\frac{n-1}{e}$, and hence $$l_r \leq \mu+ (1+\frac{n-1}{e})g$$ Finally, it remains to check that for $g\geq 2$ we have 
$$ (1+\frac{n-1}{e})g\leq  \frac{n+1}{2}(2g-2). $$

We now prove (3). Let $0=\E_0\subset \E_1\subset\cdots\subset \E_k=\E$
be the Harder--Narasimhan filtration of $E$, and let
$$
Q_s:=\E_s/\E_{s-1},\qquad 1\leq s\leq k,
$$
be the semistable quotients. Write $r_s=\operatorname{rk}(Q_s)$ and $\lambda_s=\mu(Q_s)$, so that $\lambda_1>\lambda_2>\cdots > \lambda_k$ and $\lambda_s-\lambda_{s+1}\leq 2g-2$
for $1\leq s<k$, and hence we have
$$
\lambda_1-\lambda_k\leq (k-1)(2g-2).
$$

Applying (2) to each $Q_s$, we can write $Q_s$ as an iterative extension of line bundles whose degrees lie between $\lambda_s-g$ and $\lambda_s+\frac{r_s+1}{2}(2g-2)$. It follows that $E$ is an iterative extension of all these line bundles.

Since $r_1+\cdots+r_k=n$ and $r_s\geq 1$, we have $r_s\leq n-k+1$ for every $s$. Therefore, every line bundle occurring in the above iterative extension has degree at most $\lambda_1+\frac{n-k+2}{2}(2g-2)$ and at least $\lambda_k-g$.
Thus we have 
$$
\begin{aligned}
\max_i{\deg L_i}-\min_i{\deg L_i}
&\leq
\lambda_1-\lambda_k
+\frac{n-k+2}{2}(2g-2)+g\
\\&\leq
(k-1)(2g-2)
+\frac{n-k+2}{2}(2g-2)+g.
\end{aligned}
$$
Since $g\geq 2$, we have $g\leq 2g-2$. Therefore, we obtain
$$\begin{aligned}
\max_i{\deg L_i}-\min_i{\deg L_i}
&\leq
\frac{n+k+2}{2}(2g-2)
&\leq (n+1)(2g-2)
\end{aligned}$$

\end{proof}

\bibliographystyle{alpha}
\bibliography{references}

\end{document}